\documentclass[11pt,oneside]{amsart}

\usepackage[
    left=1in,
    right=1in,
    top=1in,
    bottom=1in
]{geometry}
\usepackage{graphicx}
\usepackage{amssymb}
\usepackage{color}
\usepackage{bbm, dsfont}
\usepackage[hidelinks]{hyperref}

\hypersetup{
	colorlinks=true,
	pdfborder={0 0 0},
	pdfborderstyle={/S/U/W 0},
}

\numberwithin{equation}{section}

\usepackage{verbatim}
\usepackage{verbatim}
\usepackage{amsmath,amssymb,amsthm}  
\allowdisplaybreaks
\usepackage{amssymb}            
\usepackage{amsfonts}            
\usepackage{mathrsfs}          
\usepackage{amsthm}
\usepackage{algorithm}  
\usepackage{algorithmicx}  
\usepackage{algpseudocode}  
\usepackage{mathtools}
\usepackage{commath}

\usepackage{physics}
\let\exp\exponential

\usepackage{bm}
\usepackage{graphicx}
\usepackage{float}
\usepackage{listings}
\usepackage{subfigure}
\usepackage{multirow}
\usepackage{color}
\usepackage{bbm}
\usepackage[export]{adjustbox}
\usepackage{enumerate}
\usepackage{bbm}
\usepackage{amsmath}
\usepackage{mathrsfs}
\usepackage{amsmath}
\usepackage{mathrsfs}
\usepackage[dvipsnames]{xcolor}
\usepackage{cases}

\usepackage{caption}
\usepackage{subcaption}

\usepackage{ulem}
\newtheorem{theorem}{Theorem}[section]
\newtheorem{lemma}[theorem]{Lemma}
\newtheorem{corollary}[theorem]{Corollary}
\newtheorem{proposition}[theorem]{Proposition}

\theoremstyle{definition}
\newtheorem{definition}[theorem]{Definition}

\theoremstyle{remark}
\newtheorem{remark}[theorem]{Remark}
\newtheorem*{remark*}{Note}

\numberwithin{equation}{section}

\newcommand{\Pb}{\mathbb P}

\newcommand{\Eb}{\mathbb E}

\newcommand{\Ecal}{\mathcal E}

\newcommand{\RNum}[1]{\uppercase\expandafter{\romannumeral #1\relax}}

\newcommand{\specificthanks}[1]{\@fnsymbol{#1}}

\DeclareFontFamily{OML}{rsfs}{\skewchar\font'177}
\DeclareFontShape{OML}{rsfs}{m}{n}{ <5> <6> rsfs5 <7> <8> <9>
	rsfs7 <10> <10.95> <12> <14.4> <17.28> <20.74> <24.88> rsfs10 }{}
\DeclareMathAlphabet{\mathfs}{OML}{rsfs}{m}{n}

\newcounter{cnstcnt}

\newcommand{\cref}[1]{\ensuremath{c_{\ref*{#1}}}}

\newcounter{newcnstcnt}

\newcommand{\Cref}[1]{\ensuremath{C_{\ref*{#1}}}}

\DeclareFontFamily{U}{mathx}{}
\DeclareFontShape{U}{mathx}{m}{n}{<-> mathx10}{}
\DeclareSymbolFont{mathx}{U}{mathx}{m}{n}
\DeclareMathAccent{\widehat}{0}{mathx}{"70}
\DeclareMathAccent{\widecheck}{0}{mathx}{"71}

\begin{document}
	
\title{The conclave process}

\author{Itai Benjamini$^1$}
\author{Zhenhao Cai$^1$}\thanks{$^1$Faculty of Mathematics and Computer Science, Weizmann Institute of Science}
\author{Guanyi Chen$^2$}\thanks{$^2$School of Mathematical Sciences, Peking University}
\author{Shuyang Gong$^2$}
\author{Zhangsong Li$^2$}

\begin{abstract}
We introduce a stochastic model for the papal conclave in which $n$ cardinals vote repeatedly among themselves until one cardinal receives all the votes. In each round, the probability that a cardinal votes for a given candidate is proportional to the $\alpha$-th power of that candidate's vote count in the preceding round. For $\alpha=1$, the model reduces to the Wright-Fisher model and is dual to Kingman's n-coalescent.

We reveal a sharp transition in the absorption time \(\mathcal{T}\) at $\alpha=1$. It was known that when $\alpha=1$, $\mathcal{T}$ is typically of order $n$. We prove that for $\alpha>1$, it drops to order 
$$\textit{loglog n.}$$
In contrast, for $\alpha<1$, $\mathcal{T}$ is typically at least $\exp(\Omega(n))$.

We also prove a sharp phase transition in the identity of the winner when $\alpha>1$. For every positive integer $k$, if $2^{1/k}<\alpha<2^{1/(k-1)}$ (where we write $2^{1/0} = +\infty$), with probability tending to 1 as $n\to\infty$, the eventual winner is the unique leader after round $k$. These results show that reinforced voting processes reach consensus remarkably quickly even for large electorates.
\end{abstract}




\maketitle

\section{Introduction}

When the papacy becomes vacant, a new pope is elected through a procedure known as the {\it 
conclave}. Cardinal-electors from around the world gather in the Sistine Chapel and vote in successive rounds. In each round, every elector casts one ballot for a candidate. If some candidate receives at least two-thirds of the votes, the conclave ends and that candidate is elected pope; otherwise, the ballots are burned and another round is held. In recent conclaves, the number of electors has been on the order of one hundred, while the election has typically concluded in fewer than ten rounds.

The purpose of this paper is to introduce a stochastic model for the conclave and to establish several basic properties of this model. In particular, we obtain rigorous asymptotics for the absorption time and identify a phase transition in the predictability of the eventual winner. As we discuss below, the model is closely related to several classical stochastic processes, including the Wright-Fisher model and Kingman's $n$-coalescent. Precisely, for any $n\in \mathbb{N}^+$, define
\begin{equation}{\label{eq-def-Omega-n}}
    \Omega_n = \Big\{ \mathbf x=(\mathbf x_1,\ldots,\mathbf x_n): \mathbf x_i \in \mathbb Z_{\geq 0},\ \sum_{i=1}^{n} \mathbf x_i=n \Big\} \,.
\end{equation}
Consider a strictly increasing function $\mu:\mathbb{R}_{\geq 0} \to \mathbb{R}_{\geq 0}$ such that $\mu(0)=0$. The model is formally defined as follows. 
\begin{definition}{\label{def-conclave}}
    The random process $\mathbf{X}_k=(\mathbf{X}_k^{(1)},...,\mathbf{X}_k^{(n)})$ for $k \in \mathbb Z_{\geq 0}$ (with size $n$ and weight function $\mu$) is a Markov process on $\Omega_n$ defined by the following evolution:
    \begin{itemize}
        \item For $k=0$, set $\mathbf{X}_0:=(1,1,...,1)$. 
        \item For $k\ge 1$, given that $\mathbf{X}_{k-1}=(a_{k-1}^{(1)},...,a_{k-1}^{(n)})\in \Omega_n$, sample i.i.d.\ random variables 
        \begin{align}
            \{\mathbf{B}_{k,j}\}_{j=1}^{n} \mbox{ such that } \mathbb{P}(\mathbf{B}_{k,j}=i)=\frac{\mu(a_{k-1}^{(i)})}{\sum_{1\le l\le n}\mu(a_{k-1}^{(l)})}  \label{eq-vote-law}
        \end{align}
        for all $i,j\in [n]:=\{1,2,...,n\}$, and then define $\mathbf{X}_{k}^{(i)}:=\sum_{j=1}^{n} \mathbf{1}_{ \{ \mathbf{B}_{k,j}=i \} }$ for $i\in [n]$. 
    \end{itemize}
\end{definition}
In the context of the conclave, the labels $1,2,\ldots,n$ correspond to the
$n$ cardinals where $n \geq 2$ is a positive integer, and $\mathbf X_k^{(i)}$ represents the number of votes received by cardinal $i$ in the $k$-th round. Moreover, the event $\{\mathbf B_{k,j}=i\}$ means that, in round $k$, cardinal $j$ votes for cardinal $i$. The central idea behind the model is to incorporate a reinforcement mechanism: cardinals who receive more votes in one round become increasingly likely to receive votes in the next. The strength of this reinforcement is encoded by the weight function $\mu$. The conclave process belongs to the broad class of reinforced stochastic processes; see \cite{pemantle2007survey} for a general survey. Related Pólya-type dynamics on hypergraphs are studied in \cite{alves2023p}; while reinforcement also plays a central role in preferential attachment models; see \cite{van2024random} for a comprehensive treatment. In this paper, we focus on the power-law case where $\mu$ is a power function, i.e., $\mu(m)=m^{\alpha}$ with $\alpha>0$, and regard $\alpha$ as the main parameter of the model. We denote by $\mathbb P_n^\alpha$ the law of the process $\{\mathbf X_k\}_{k\in\mathbb Z_{\geq 0}}$ with this choice of parameter.
Note that the absorbing states of this Markov process are the $n$-tuples in $\Omega_n$ with one coordinate equal to $n$ and all remaining coordinates equal to zero. We define the {\it absorption time} $\mathcal T$ as the smallest $k \in \mathbb Z_{\geq 0}$ such that $\mathbf{X}_k$ is in an absorbing state. One of the main goals of this paper is to estimate $\mathcal T$.

\begin{remark}
Note that the absorption time is defined by the first time when one cardinal receives all votes. This definition is not fully consistent with the real-world rule of the papal election, which requires some candidate to receive at least two-thirds of the votes. However, it can be derived following the same methods in our paper that under our defined voting dynamics, the expected time a candidate wins two-thirds of the votes is of the same order of $\mathbb E[\mathcal T]$, and is $(1+o(1))\mathbb E[\mathcal T]$ when $\alpha > 1$.
\end{remark}

As we discuss in Section~\ref{section1.1_Kingman}, when $\alpha=1$ the conclave process is a special case of the Wright-Fisher model and, under time reversal, corresponds exactly to Kingman's $n$-coalescent \cite{kingman1982genealogy}. This connection allows for a precise analysis of the distribution of $\mathcal T$ in the case $\alpha=1$. However, the correspondence relies on a special feature of the linear case that the partition function $\mathbf{Z}_k:=\sum_{1\le i\le n}\mu(\mathbf{X}_k^{(i)})$
remains constant. For general $\alpha\neq 1$, this property is lost, and the analysis of $\mathcal T$ requires new ideas.

\subsection{Main results}\label{section1.2_main_result}

In this subsection, we will introduce the main results we will derive. Our first result settles the asymptotics of $\mathcal T$ when $\alpha>1$. Intuitively, compared with the $\alpha=1$ case (where $\mathcal{T}$ is typically of order $n$), $\mathcal{T}$ is expected to be much smaller for $\alpha>1$ and much larger for $\alpha<1$. Our result shows that when $\alpha>1$, $\mathcal{T}$ is asymptotically of order $\log\log(n)$. In other words, even if the entire human race were to participate in the vote ($n\approx 10^{10}$), the election would still conclude in a few rounds, as in a real conclave ($n\approx 100$). 
\begin{theorem}{\label{thm-rapid-ordering-supercritical}}
    Define $\lambda_k = 2^{1/k}$ for $k \geq 1$ and $\lambda_0=+\infty$. Suppose that $k \geq 2$ and $\alpha\in(\lambda_k,\lambda_{k-1}]$. Then for every $\epsilon > 0$ we have
    \begin{equation}{\label{eq-order-mathcal-T-non-critical-small-alpha}}
 		\lim\limits_{n\to \infty} \mathbb{P}_n^{\alpha}\left( \left| \frac{ \log(\alpha) \cdot \mathcal T}{ 2\log\log(n) } -1 \right|<\epsilon \right) =1 \,. 
 	\end{equation}
    In addition, suppose that $\alpha>2$ (corresponding to $k=1$ and $\alpha\in(\lambda_1,\lambda_0)$), and denote by $E_{\mathsf{unique}}$ the event that $\{ \mathbf X_1^{(i)}: 1 \leq i \leq n \}$ has a unique maximum. Then
    \begin{align}
        & \lim\limits_{n\to \infty} \mathbb{P}_n^{\alpha}\left( \left| \frac{ \log(\alpha) \cdot \mathcal T}{ 2\log\log(n) } -1 \right|<\epsilon \mid E_{\mathsf{unique}} \right) =1 \,; \label{eq-order-mathcal-T-alpha>2-case-1} \\
        & \lim\limits_{n\to \infty} \mathbb{P}_n^{\alpha}\left( \left| \frac{ \log(\alpha) \cdot \mathcal T}{ (1+\frac{\alpha}{2})\log\log(n) } -1 \right|<\epsilon \mid E_{\mathsf{unique}}^c \right) =1 \,. \label{eq-order-mathcal-T-alpha>2-case-2}
    \end{align}
\end{theorem} 
\begin{remark}
    We can see that $\Pb_n^{\alpha}(E_{\mathsf{unique}})=\Omega(1)$ if and only if $N!/C \leq n\leq C N!$ for some fixed constant $C$ and a positive integer $N = N(n)$ by the Poisson coupling method we will derive later in Section~\ref{subsec:coupling}. Therefore, there is a non-trivial phase transition of the distribution of $\mathcal T$ with respect to the number-theoretic properties of $n$.
\end{remark}
By contrast, for $\alpha<1$, we show that the absorption time $\mathcal{T}$ grows at least exponentially with $n$. Thus, $\mathcal{T}$ exhibits an unexpectedly high sensitivity to $\alpha$ when $\alpha$ is near $1$. 
\begin{theorem}{\label{thm-slow-ordering-sub-critical}}
    For any $0<\alpha<1$ and $M \in \mathbb Z_{>0}$, there exists $C_{\ref{thm-slow-ordering-sub-critical}}=C_{\ref{thm-slow-ordering-sub-critical}}(\alpha)>0$ such that
    \begin{equation}{\label{eq-order-mathcal-T-sub-critical}}
  	    \mathbb{P}_n^{\alpha}\big( \mathcal{T} > M \big) \ge \big(1-e^{-C_{\ref{thm-slow-ordering-sub-critical}}n} \big)^{M} \,. 
    \end{equation}  
    As a result, $\mathcal{T}$ is at least $\exp(\Omega_{\alpha}(n))$ with probability $1-o(1)$.
\end{theorem}

Our next result shows a surprising ``all-or-nothing'' transition phenomenon in the survival probability of the leader in a specific round. Precisely, for each $k \in \mathbb Z_+$ we define the set of leaders at time $k$ to be 
\begin{equation}{\label{eq-def-leader-set}}
    \mathfrak L_k:= \left\{ 1 \leq i \leq n: \mathbf X^{(i)}_k = \max\left\{ \mathbf X_k^{(1)},\ldots,\mathbf X_k^{(n)} \right\} \right\} \,.
\end{equation}
Also, define $\mathfrak L_{0}=\emptyset$. Next, for any $\mathbf x = (\mathbf x_1\,, \ldots\,, \mathbf x_n) \in \Omega_n$ we define $\mathcal L(\mathbf x) = \{i\in[n]:\mathbf x_i = \max_{j \in [n]}\mathbf x_j\}$ to be the set of leaders in $\mathbf x$. In addition, recall the definition of the absorption time $\mathcal T$ and define $\mathcal W \in [n] \mbox{ such that } \mathbf X_{\mathcal T}^{(\mathcal W)} =n$.
The following result shows the following transition in the probability that $\mathcal W \in \mathfrak L_k$.
\begin{theorem}{\label{thm-all-or-nothing-transition}}
    Define $\lambda_0=+\infty$ and define $\lambda_k= 2^{1/k}$ for each $k \geq 1$. Then for any $k \geq 1$ and $\alpha \in (\lambda_k,\lambda_{k-1})$, we have
    \begin{align*}
        \Pb_n^{\alpha}\big( \mathcal W \in \mathfrak L_k \big)=1-o(1) \mbox{ and } \Pb_n^{\alpha}\big( \mathcal W \in \mathfrak L_{k-1} \big) = o(1) \,.
    \end{align*} 
\end{theorem}
\begin{remark}{\label{rmk-uniqueness-leader}}
    The reader may wonder whether $\mathfrak L_k$ consists of a single element when $\alpha\in(\lambda_k,\lambda_{k-1})$. We will show later that, for $k \geq 2$ and $\alpha\in( \lambda_k,\lambda_{k-1} )$, 
    \begin{align*}
        \Pb_n^{\alpha}\left( |\mathfrak L_k|=1 \right)=1-o(1) \,.
    \end{align*}
    In other words, with high probability there is a unique leader at round $k$, and this leader is also the eventual winner.

    The analogous statement, however, does not hold for $k=1$ when $\alpha\in (\lambda_1,\lambda_0)=(2,+\infty)$. Indeed, since $\mathbf{X}_0=(1,1,\ldots,1)$, the distribution of $\mathbf{X}_1$ is the same as that obtained by throwing $n$ balls independently and uniformly into $n$ bins. This is the classical occupancy problem, which has been extensively studied. In particular, it is known that with high probability
    \begin{align*}
        \max_{i\in [n]} \left\{ \mathbf{X}_1^{(i)} \right\} = O\left( \frac{\log n}{\log\log n} \right)
    \end{align*}
    and that the difference between the maximum and the second maximum among $\mathbf{X}_1^{(i)}$ is typically $O(1)$ (see e.g., \cite{raab1998balls}). Consequently, round $1$ need not have a unique leader with high probability. Nevertheless, we will prove that in this regime, with high probability there is a unique leader by round $2$, and that this leader is the eventual winner. 
\end{remark}

\subsection{Simulations}{\label{subsec:simulations}}

To give a better understanding of how our theorems can be applied, we carried out three numerical simulations on the conclave model and summarize the results in Figure~\ref{fig:simulations}, where:

\begin{itemize}
\item Figure (a) (in our first simulation) shows how the average value of $\min\{\mathcal T, 100\}$ over $1000$ runs changes with $n$ under different $\alpha$'s. In the subcritical regime $\alpha < 1$, $\mathcal T$ grows exponentially and $\Eb[\min\{\mathcal T, 100\}] \to 100$ as we expected, which is a finite-sample evidence consistent with our claim in Theorem~\ref{thm-slow-ordering-sub-critical}. In the critical regime $\alpha = 1$, $\Eb[\min\{\mathcal T, 100\}]$ grows linearly before being close to $100$ as shown by our figure. In the supercritical regime $\alpha > 1$, $\Eb[\min\{\mathcal T, 100\}]$ grows at an $o(\log n)$ rate, and never becomes close to $100$ as shown in the figure, which is consistent with our claim in Theorem~\ref{thm-rapid-ordering-supercritical}.

\item Figures (b) and (c) (in our second simulation) show how the probability $\Pb_n^\alpha(\mathcal{W} \in \mathfrak L_k)$ changes with $k$ for different $n$ and fixed $\alpha$. Note that $\alpha = 1.9 \in (\lambda_2 , \lambda_1)$ in Figure (b), where $\Pb_{n}^\alpha(\mathcal{W} \in \mathfrak L_k)$ tends to 1 as $n \to +\infty$ when $k \geq 2$, and 0 when $k = 1$. Also, $\alpha = 2.5 \in (\lambda_1 , \lambda_0)$ in Figure (c), where $\Pb_{n}^\alpha(\mathcal{W} \in \mathfrak L_k)$ tends to $1$ as $n \to +\infty$ for all $k \geq 1$. Therefore, the simulation results are consistent with Theorem~\ref{thm-all-or-nothing-transition}, which claims that for large enough $n$, with high probability the eventual winner is determined in the second round when $\alpha = 1.9$, and the eventual winner belongs to the first-round leader set when $\alpha = 2.5$.

\end{itemize}

\begin{figure}[htbp]
    \centering
    \begin{tabular}{cc}
        \includegraphics[width=0.45\textwidth]{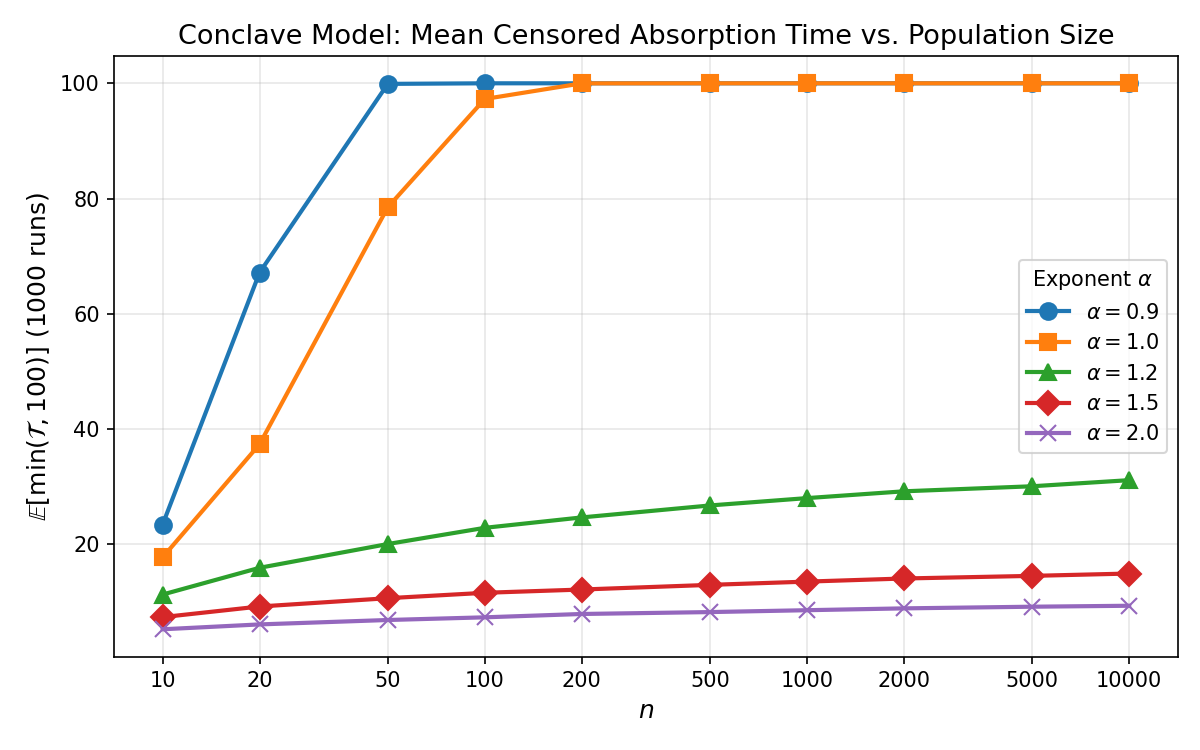} \\ 
        (a) Plot of $\Eb[\min\{\mathcal T, 100\}]$  \\
        \includegraphics[width=0.45\textwidth]{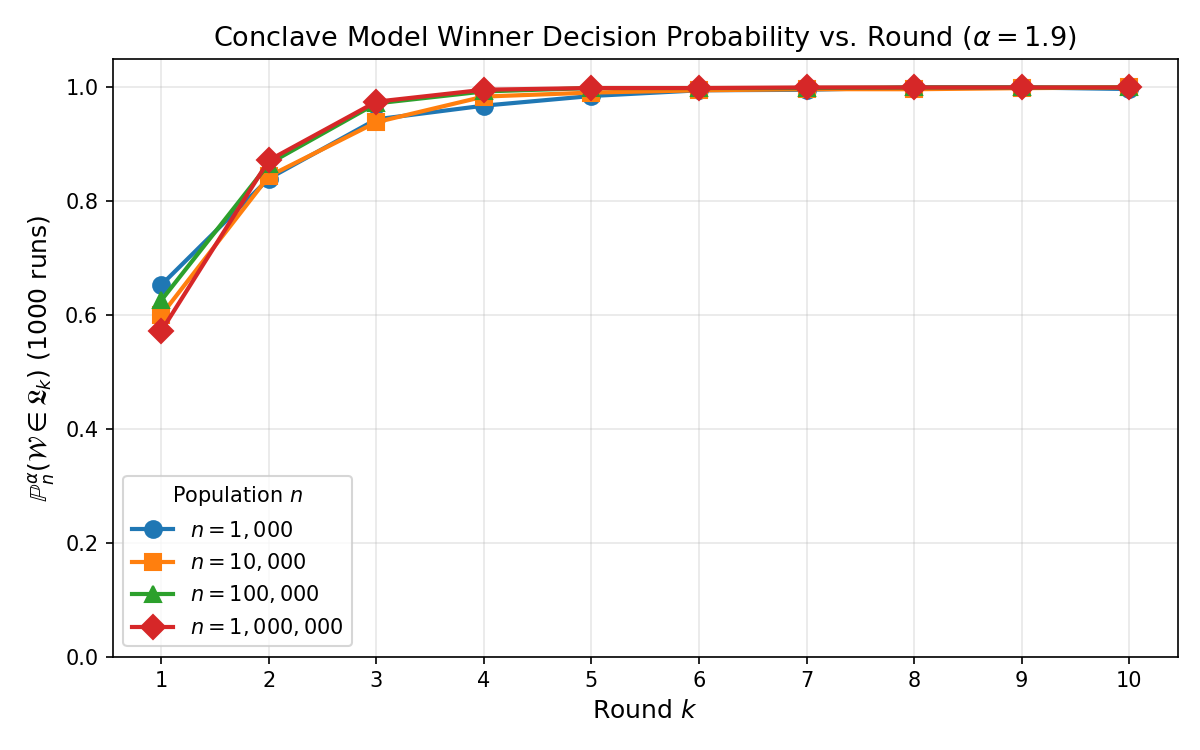}  \includegraphics[width=0.45\textwidth]{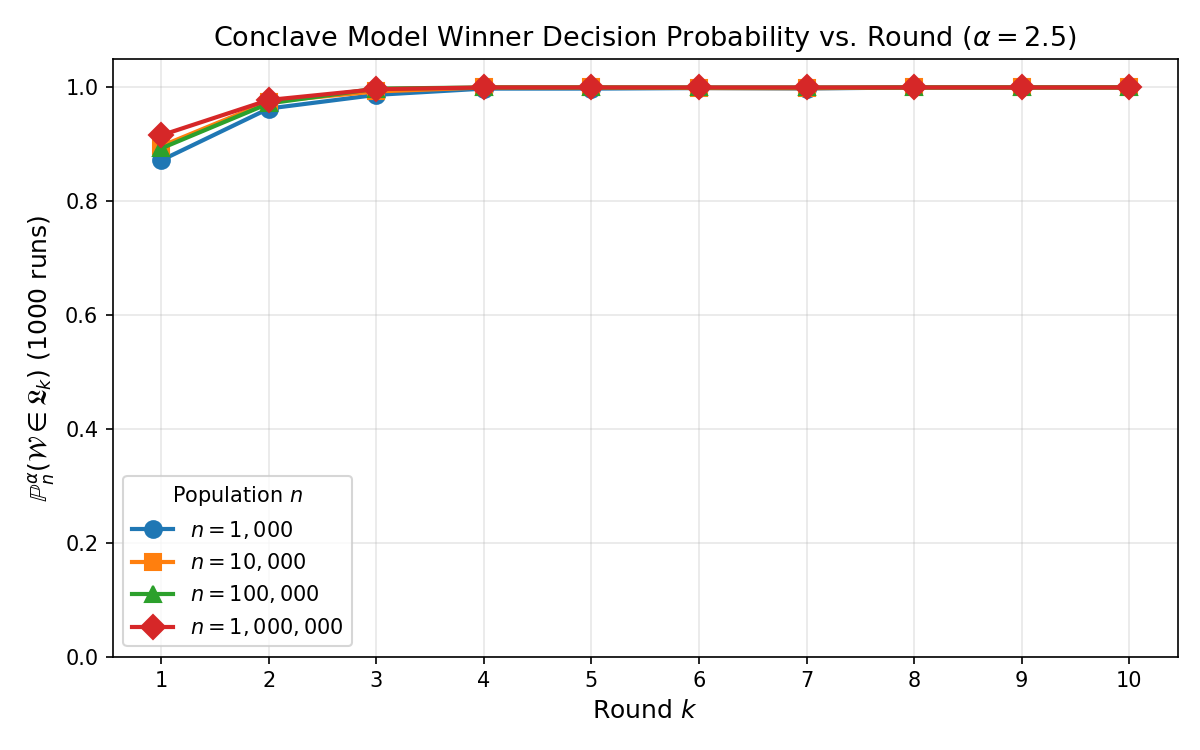} \\
        (b) Plot of $\Pb_{n}^\alpha(\mathcal{W} \in \mathfrak L_k)$ ($\alpha = 1.9$) \quad\quad (c) Plot of $\Pb_{n}^\alpha(\mathcal{W} \in \mathfrak L_k)$  ($\alpha = 2.5$)  \\
    \end{tabular}
    \caption{\centering Simulation results on the conclave model.}
    \label{fig:simulations}
\end{figure}

\subsection{Open problems}

Our work reiterates a number of future research directions as we discuss below.

\underline{The critical case $\alpha=\lambda_k$.} In comparison to Theorem~\ref{thm-rapid-ordering-supercritical}, Theorem~\ref{thm-all-or-nothing-transition} does not include the result when $\alpha = \lambda_k$. Our proof shows that, in this case,
\begin{align*}
    \Pb^{\alpha}_n\left( \mathcal W \in \mathfrak L_{k+1} \right)=1-o(1), \quad \Pb_n^{\alpha}\left( \mathcal W \in \mathfrak L_{k-1} \right)=o(1) \,.
\end{align*}
We conjecture that one also has 
\begin{align*}
    \Pb_n^{\alpha}\left( \mathcal W \in \mathfrak L_{k} \right)=o(1)\,, \quad \Pb_n^{\alpha}\left(| \mathfrak L_{k}| = 1 \right) = 1 - o(1)
\end{align*}
so that the eventual winner is uniquely determined only by the leader at round $k+1$. This conjecture can be proved in the special case $k=1,\alpha=2$. We leave the general critical case $\alpha=2^{1/k}, k \geq 2$ as an intriguing open problem.

\underline{Settling sharp convergence rate.} The convergence rates in Figures (b) and (c) are relatively slow, since the $(\log n)^{-\Omega(1)}$ terms in the remainder converge very slowly when $\alpha$ is close to $\lambda_k$ for some $k \geq 1$. We leave the lower bound on the convergence rate as an interesting open problem.

\subsection{The $\alpha=1$ case}\label{section1.1_Kingman}

We briefly mention the $\alpha=1$ case. Genetic drift, namely the random fluctuation of allele frequencies in a population, is a central topic in evolutionary biology and has been modeled in many different ways. A fundamental model in this direction is the Wright--Fisher model, named after Sewall Wright and Ronald Fisher, which has become one of the most widely studied extensions of Mendelian genetics. In its classical form, the model considers a sequence of non-overlapping generations, each consisting of a fixed number $n/2$ of diploid individuals, or equivalently $n$ alleles (where $n$ is even). If a generation contains $m$ copies of a given allele, then the number of copies of this allele in the next generation has distribution $\operatorname{Bin}(n,\frac{m}{n})$,
that is, the sum of $n$ independent Bernoulli random variables with success probability $m/n$. Note that the evolutions of the counts of different alleles are correlated because the total population size is fixed. Under $\mathbb P_n^1$, the conditional distribution of $\mathbf X_k^{(i)}$ given $\{\mathbf X_{k-1}^{(i)}=m\}$ is precisely of this form. Thus, the conclave model with $\alpha=1$ can be viewed as the Wright-Fisher model started from $n$ distinct alleles.

In \cite{kingman1982genealogy}, Kingman observed that this model admits a particularly simple backward-in-time description. In the forward Wright-Fisher dynamics, each allele in the new generation chooses its parent independently and uniformly from the $n$ alleles in the previous generation. Consequently, when the process is viewed backward in time, each ancestral lineage independently chooses a uniformly random parent lineage in the preceding generation, and any pair of lineages coalesces in one step with probability $1/n$. This backward process is Kingman's $n$-coalescent. A natural question is how many generations are needed for all $n$ lineages to coalesce into a single common ancestor. By the correspondence described above, this coalescence time, denoted by $\mathfrak T_n$, has the same distribution as the absorption time $\mathcal T$ under the law $\mathbb P_n^1$. 

More recently, Kingman's $n$-coalescent has been reformulated in terms of compositions of random functions and studied in a series of works \cite{adler2003coalescing, dalal2002compositions, goh2002iterating, hitczenko2005central, mcsweeney2008expected}. Let $\{f_i\}_{i\in\mathbb N^+}$ be a sequence of i.i.d. random functions, each chosen uniformly from the $n^n$ functions from $[n]$ to $[n]$. For $k\in\mathbb N^+$, define $g_k := f_k\circ f_{k-1}\circ \cdots \circ f_1$,
and let $R_k := |\operatorname{Im}(g_k)|$ be the cardinality of its image. Clearly, $(R_k)_{k\ge 1}$ is non-increasing. Moreover, it is immediate from the backward representation that $R_k$ has the same distribution as the number of ancestral lineages remaining after tracing Kingman's $n$-coalescent back for $k$ generations. In particular, the first time $k$ at which $g_k$ is a constant function, that is, $g_k(i)=g_k(j)$ for all $i,j\in[n]$, coincides in distribution with $\mathfrak T_n$. In fact, the sequence $\{R_k\}_{k\ge 1}$ is a Markov process with explicit transition probabilities, from which the order of $\mathbb{E}[\mathfrak{T}_n]$ follows by a straightforward computation. Precisely, when the range of $g_k$ is sufficiently small, say $R_k=m = o(\sqrt{n})$, the range decreases in the next composition with probability $\Theta(\tbinom{m}{2}\cdot n^{-1})$; moreover, when it does decrease, it typically drops by $1$. In other words, the expected holding time between successive drops is of order $\frac{n}{m^2}$. Thus,
\begin{align*}
    \mathbb{E}[\mathfrak{T}_n] \propto \sum\nolimits_{2\le m \le o(\sqrt{n})} \frac{n}{m^2}\propto n \,.
\end{align*}
A more refined analysis \cite{dalal2002compositions} showed that $\mathbb{E}[\mathfrak{T}_n]\sim 2n$, where $h_1\sim h_2$ means that $\frac{h_1(n)}{h_2(n)}$ converges to $1$ as $n\to \infty$. Furthermore, the asymptotic behavior of the tail probability of $\mathfrak{T}_n$ was established in \cite{goh2002iterating}: for any $a>0$, 
\begin{equation*}
    \lim\limits_{n\to \infty} \mathbb{P}\big( \mathfrak{T}_n \le an \big) = \int_{0}^{a} h(t) \mathrm{d}t,
\end{equation*}  
where 
\(
    h(t)=\sum_{k\ge 2}(-1)^k \binom{k}{2}(2k-1) e^{-\binom{k}{2}t} \,.
\)
Moreover, the central limit theorem for the cardinality of $\{R_{k}\}_{1\le k\le \mathfrak{T}_n}$ was established in \cite{hitczenko2005central}. The generalized version of this iteration model, where $f_i$ is a map $r:[n]\to [n]$ with probability $\prod_{1\le i\le n}p_{r(i)}$ for a predetermined probability vector $(p_1,...,p_n)$ (not necessarily $(\frac{1}{n},...,\frac{1}{n})$), was studied in \cite{adler2003coalescing, mcsweeney2008expected}. 

However, when $\alpha\neq 1$, the backward process of $\{ \mathbf{X}_k \}_{k\in \mathbb{N}}\sim \mathbb{P}_n^{\alpha}$ no longer admits the tractable form as the case $\alpha=1$. For example, suppose that $\mathbf{X}_k^{(i)}=b_i$ for $i\in [n]$; to imitate the backward process, one needs $\sum_{1\le i\le n}b_i^{\alpha}$ ``voters'' in the previous step in a heuristic sense, but this quantity (i.e., the partition function $\mathbf{Z}_k$) is not deterministic. To this end, estimating $\mathcal{T}$ for $\alpha\neq 1$ requires new ideas beyond those used in the $\alpha=1$ case.

\subsection{Notation}{\label{subsec:notation}}

We record some notation conventions in this subsection. We denote by $\mathsf{Ber}(p)$ the Bernoulli distribution with parameter $p$, denote by $\mathsf{Bin}(n,p)$ the binomial distribution with $n$ trials and success probability $p$, and denote by $\mathsf{Mult}(n;p_1,\ldots,p_k)$ the multinomial distribution with $n$ trials and parameters $p_1,\ldots,p_k$ such that $p_1+\ldots+p_k=1$. In addition, we denote by $\mathcal{N}(\mu, \sigma^2)$ the normal distribution with mean $\mu$ and variance $\sigma^2$, and denote by $\mathcal{N}(\mu,\Sigma)$ the multivariate normal distribution with mean $\mu$ and covariance matrix $\Sigma$. Recall Definition~\ref{def-conclave}. Viewing $(\mathbf X_k)_{k=1}^{\infty}$ as a Markov chain, we define $\mathcal X:\Omega_n \times \Omega_n \to \mathbb R$ to be the transition matrix of this Markov chain. For any $\mathbf X\in\Omega_n$, define $\nu_{\mathbf X}:\Omega_n \to [0,1]$ such that 
\begin{equation}{\label{eq-def-mu}}
    \nu_{\mathbf X}(\mathbf X')=\Pb_{n}^{\alpha}( \mathbf X_2=\mathbf X' \mid \mathbf X_1 = \mathbf X ) \,.
\end{equation}

We use standard asymptotic notations: for two sequences $a_n$ and $b_n$ of positive numbers, we write $a_n = O(b_n)$, if $a_n<Cb_n$ for an absolute constant $C$ and for all sufficiently large $n$ (similarly we use the notation $O_h$ if the constant $C$ is not absolute but depends only on $h$); we write $a_n = \Omega(b_n)$, if $b_n = O(a_n)$; we write $a_n = \Theta(b_n)$, if $a_n =O(b_n)$ and $a_n = \Omega(b_n)$; we write $a_n = o(b_n)$ or $b_n = \omega(a_n)$, if $a_n/b_n \to 0$ as $n \to \infty$. Since $\alpha$ is assumed to be fixed given a particular regime we consider, we omit the subscript $\alpha$ for the asymptotic symbols above when there is dependence on $\alpha$. For two real numbers $a$ and $b$, we let $a \vee b = \max \{ a,b \}$ and $a \wedge b = \min \{ a,b \}$. For two sets $A$ and $B$, we define $A\sqcup B$ to be the disjoint union of $A$ and $B$ (so the notation $\sqcup$ only applies when $A, B$ are disjoint). The indicator function of a set $A$ is denoted by $\mathbf{1}_{A}$. In addition, we use both $|A|$ and $\#A$ to denote the cardinality of $A$.

\section{Proof Strategies}{\label{sec:proof-strats}}

In this section, we introduce the strategies for proving Theorems~\ref{thm-rapid-ordering-supercritical}, \ref{thm-slow-ordering-sub-critical}, and \ref{thm-all-or-nothing-transition}. Our proof is based on Poisson approximation for the partition function, extreme value analysis for the multinomial distribution, and a recursive augment on the evolution of the first and second leader.

\subsection{Proof strategies of Theorems~\ref{thm-rapid-ordering-supercritical} and \ref{thm-all-or-nothing-transition}}{\label{subsec:strat-thm-rapid}}

To give a heuristic for the absorption time $\mathcal T$, a useful observation is to simultaneously investigate the evolution of the largest and second-largest vote counts. Without loss of generality, suppose that the first candidate gets the largest number of votes and the second candidate gets the second largest number of votes. Also, suppose that the current vote vector $\mathbf X$ has maximum $\mathbf X^{(1)} =\mathbf M$, second maximum $\mathbf X^{(2)} = (1-\delta)\mathbf M$, and partition function $\mathbf Z$. Then, we evolve $\mathbf X$ into $\mathbf X'$ by our voting dynamic. It can be seen that
\begin{align*}
\Eb[(\mathbf X')^{(1)}] = \frac{n\mathbf M^\alpha}{\mathbf Z}\,, \quad \Eb[(\mathbf X')^{(2)}] = \frac{n\mathbf M^\alpha(1-\delta)^\alpha}{\mathbf Z}\,.
\end{align*}
Therefore, given that $\mathbf M$ and $(1-\delta)\mathbf M$ are sufficiently large, and that $\mathbf Z$ is controlled, we expect $(\mathbf X')^{(1)}$ and $(\mathbf X')^{(2)}$ concentrate near their expected values, and will not be surpassed by other candidates (which will be shown by Proposition~\ref{lem-repeat-used}). Therefore, when $\delta$ is small enough, the new relative gap $\delta'$ between the maximum and the second maximum will approximately satisfy $1-\delta' = (1-\delta)^\alpha$.
Therefore, if we assume that the ``initial'' $\delta$ is $\delta_0$, and the ``resulting'' delta at a time when someone wins $n-\sqrt{n}$ votes (in which case the voting process reaches absorption after constant steps with high probability, as we will show in the proof) is $1-\frac{1}{\sqrt{n}}$, we have
\begin{align*}
\frac{1}{\sqrt{n}} \approx (1-\delta_0)^{\alpha^\mathcal T} \Longleftrightarrow \delta_0 \cdot \alpha^\mathcal T = \Theta(\log n) \Longleftrightarrow \mathcal T = (1+o(1))\cdot\frac{\log(\delta^{-1}) + \log \log n}{\log \alpha}\,.
\end{align*}
Next, we show that after the first few rounds, $\mathbf M$ and $(1-\delta)\mathbf M$ are sufficiently large for our heuristic to apply. The complete proof of the validity of our heuristics is carried out in Section~\ref{sec:determining-T-and-W}.

Later in this section, we call the regime in which the above heuristics apply \textit{the heuristic zone}, see \eqref{eq-assum-lem-2.4} for a rigorous definition. To describe precisely what happens before the voting dynamics enters the heuristic zone (especially determining the order of $\delta_0$, and the time at which the eventual winner is decided), we can split into two cases:

\textbf{Case 1: $\alpha \geq 2$.} By coupling $\mathsf{Mult}(n;\frac{1}{n}, \ldots , \frac{1}{n})$, with $n$ independent Poisson random variables with means $1$, we can show that $\mathbf M_{\mathbf X_1} = O(\frac{\log n}{\log \log n})$ and $\mathbf Z_{\mathbf X_1} = O(n\log \log n)$ with high probability (see Lemma~\ref{lem-first-round-distribution-features}). Without loss of generality assume $\mathbf X_1^{(1)} = \mathbf M_{\mathbf X_1}$. Then, for $i$ such that $\mathbf X_1^{(i)} = \mathbf M_{\mathbf X_1} - 1$ we have
\begin{align*}
\mathbf X_2^{(1)} - \mathbf X_2^{(i)} &= \frac{n}{\mathbf Z_{\mathbf X_1}} \cdot (\mathbf M_{\mathbf X_1}^\alpha - (\mathbf M_{\mathbf X_1} - 1)^\alpha) + O(1) \cdot \sqrt{\tfrac{n\mathbf M_{\mathbf X_1}^\alpha}{\mathbf Z_{\mathbf X_1}}} \\
&= \Theta((\log n)^{\alpha - 1+o(1)}) + O((\log n)^{\tfrac{\alpha}{2} + o(1)})
\end{align*}
with high probability. Since $\alpha - 1 \geq \frac{\alpha}{2}$, the loss of expected value for $\mathbf X_2^{(i)}$ dominates the fluctuation (and strict domination holds when $\alpha > 2$), and therefore determines the resulting relative gap $\delta_*$ between $\mathbf X_2^{(i)}$ and $\mathbf X_2^{(1)}$. Since $\mathbf M_{\mathbf X_2} = \Theta(\frac{n\mathbf M_{\mathbf X_1}^\alpha}{\mathbf Z_{\mathbf X_1}}) = (\log n)^{\alpha+o(1)}$, we have $\delta_* = (\log n)^{\alpha - 1 - \alpha + o(1)} = (\log n)^{-1+o(1)}$\,. Therefore, we have the following subcases:
\begin{itemize}
\item When $E_{\mathsf{unique}}$ holds, no other elements can contribute to the maximum and second maximum (under proper control of the sizes of level sets below $\mathbf M_{\mathbf X_1}$), which gives 
\begin{align*}
\delta_0 = \delta_* = (\log n)^{-1 + o(1)}\Longrightarrow \mathcal T = (1+o(1))\cdot\frac{2\log \log n}{\log \alpha}\,.
\end{align*}

\item When $E_{\mathsf{unique}}^c$ holds, all elements equal to $\mathbf M_{\mathbf X_1}$ will evolve into $\frac{n\mathbf M_{\mathbf X_1}^\alpha}{\mathbf Z_{\mathbf X_1}}+O(\sqrt{ n\mathbf M_{\mathbf X_1}^\alpha/\mathbf Z_{\mathbf X_1} })$ with relative gaps $\delta_{**} = \frac{(\log n)^{o(1)}}{\sqrt{ n\mathbf M_{\mathbf X_1}^\alpha/\mathbf Z_{\mathbf X_1} }} = (\log n)^{-\alpha / 2 + o(1)} \leq \delta_*$. Therefore, the maximum and the second maximum in the second round are contributed solely by elements equal to $\mathbf M_{\mathbf X_1}$ in the first round with high probability (under proper control of the sizes of level sets below $\mathbf M_{\mathbf X_1}$), which gives
\begin{align*}
\delta_0 = \delta_{**} = (\log n)^{-\tfrac{\alpha}{2} + o(1)}\Longrightarrow \mathcal T = (1+o(1))\cdot\frac{(\tfrac{\alpha}{2} + 1)\log \log n}{\log \alpha}\,.
\end{align*}
\end{itemize}

Combining the above subcases, we have proposed intuitive arguments supporting Theorem~\ref{thm-rapid-ordering-supercritical} when $\alpha \geq 2$. Also, the above analysis implies that the winner is determined in the first round, which shows Theorem~\ref{thm-all-or-nothing-transition} when $\alpha > 2$ (note that the domination is not strict when $\alpha = 2$). 

To actually apply the above ideas, however, we need to solve technical barriers such as giving a sufficient growth rate of $\#\{i : \mathbf X_1^{(i)} = \mathbf M_{\mathbf X_1} - j\}$ with respect to $j \geq 1$, dealing with the conditioning on $E_{\mathsf{unique}}$ or $E^c_{\mathsf{unique}}$, or handling the critical case $\alpha = 2$ where $\alpha - 1 = \frac{\alpha}{2}$. In our actual proof, we would first derive Poisson coupling techniques in Lemmas~\ref{lem-Poisson-coupling} and \ref{lem-Poisson-coupling-further}, and then apply the Poisson coupling techniques to prove Lemma~\ref{lem-first-round-distribution-features}, which describes the features of vote distribution after the first round and deals with conditioning. Next, we use these features to control the second round vote distribution conditioned on $E_{\mathsf{unique}}$ or $E^c_{\mathsf{unique}}$ in Lemma~\ref{lem-alpha>2-second-round}. Also, to deal with the $\alpha = 2$ case, we prove Lemma~\ref{lem-Poisson-maximum-1st-minus-2nd}, which deals with the gap between the maximum and the second maximum among inhomogeneous Poisson variables by a mapping method. Finally, starting from the second round, we use union bounds (see Lemma~\ref{lem-prelim-gap-est-time-2-to-200}) to prove that the voting dynamics enters our heuristic zone in $10$ steps in Proposition~\ref{cor-core-lemma-alpha>=2}, which shows Theorems~\ref{thm-rapid-ordering-supercritical} and \ref{thm-all-or-nothing-transition} when $\alpha \geq 2$. The full proof is carried out in Section~\ref{subsec:case-k=1}.

\textbf{Case 2: $1 <\alpha < 2$.} This case is quite different from the $\alpha \geq 2$ case, due to the fact that $(\log n)^{\alpha - 1+o(1)}$ does not dominate $(\log n)^{\frac{\alpha}{2}+o(1)}$, invalidating our previous analysis. Therefore, we need to carefully investigate the contribution to the maximum from each level set. However, we can still show that $\mathbf M_{\mathbf X_t} = (\log n)^{\alpha^{t-1}+o(1)}$ and $\mathbf Z_{\mathbf X_t} = \Theta(n)$ with high probability for constant $t$ by a more delicate analysis in Lemma~\ref{lem-good-event-E-1}. If we define $\mathcal A_{j,\beta} = \{i : \mathbf X_j^{(i)} = \mathbf M_{\mathbf X_j} - (\log n)^{\beta + o(1)}\}$ for $j \in \mathbb Z_+$ (note that this definition is rough and $(\log n)^{\beta + o(1)}$ refers to a positive integer that matches this magnitude), then by Lemma~\ref{lem-first-round-distribution-features} which describes the first round voting distribution, we have $\log |\mathcal A_{1,\beta}| = (\log n)^{\beta + o(1)}$.
Therefore, if we approximate Poisson variables by Gaussian variables we have
\begin{align}
\max_{i \in \mathcal A_{1,\beta}}\mathbf X_2^{(i)} &\approx \frac{n(\mathbf M_{\mathbf X_1} - (\log n)^{\beta + o(1)})^\alpha}{\mathbf Z_{\mathbf X_1}} + \sqrt{2\log |\mathcal A_{1,\beta}| \cdot \frac{n(\mathbf M_{\mathbf X_1} - (\log n)^{\beta + o(1)})^\alpha}{\mathbf Z_{\mathbf X_1}} }\nonumber \\
&= \frac{n\mathbf M_{\mathbf X_1}^\alpha}{\mathbf Z_{\mathbf X_1}} - (\log n)^{\beta + \alpha - 1 + o(1)} + (\log n)^{\tfrac{\beta + \alpha}{2}+o(1)}\,,\label{rough-analysis-max-second-round}
\end{align}
therefore, when $\alpha + \beta < 2$, we have $\tfrac{\beta + \alpha}{2} > \beta + \alpha - 1$, which means the gain of $\max_{i \in \mathcal A_{1,\beta}}\mathbf X_2^{(i)}$ due to enumeration is actually larger than the loss of $\max_{i \in \mathcal A_{1,\beta}}\mathbf X_2^{(i)}$ due to expected value; When $\alpha + \beta > 2$, the converse holds and $\max_{i \in \mathcal A_{1,\beta}}\mathbf X_2^{(i)} < \frac{n\mathbf M_{\mathbf X_1}^\alpha}{\mathbf Z_{\mathbf X_1}}$ with high probability. This suggests that the maximum of $\max_{i \in \mathcal A_{1,\beta}}\mathbf X_2^{(i)}$ is attained when $\beta = 2 - \alpha$. Therefore, with high probability the winner in the second round comes from $\mathcal A_{1,2-\alpha}$.

Based on this fact, we continue our analysis. To analyze what happens in the third step, we first need an estimate on $\log |\mathcal A_{2,\beta}|$. Note that for any $\beta' > 0$, the layer $\mathcal A_{1,\beta'}$ has a disadvantage of $(\log n)^{\beta' + \alpha - 1}$ on the expected value in the second round. Therefore, intuitively we have $|\mathcal A_{2,\beta}| \approx |\mathcal A_{1,\beta + 1 - \alpha}|$ when $\beta \geq \alpha - 1$ and $|\mathcal A_{2,\beta}| = (\log n)^{o(1)}$ when $0 < \beta < \alpha - 1$. This gives
\begin{align*}
\delta \mbox{ at second step }&= \frac{(\log n)^{\alpha - 1 + o(1)}}{\mathbf M_{\mathbf X_2}} = (\log n)^{-1 + o(1)}\,, \\
\log |\mathcal A_{2,\beta}| &= (\log n)^{0\vee(\beta + 1 - \alpha) + o(1)}\mbox{ for }\beta > 0\,,
\end{align*}
Applying an analysis similar to \eqref{rough-analysis-max-second-round}, we have
\begin{align*}
\max_{i \in \mathcal A_{2,\beta}}\mathbf X_3^{(i)} &\approx \frac{n(\mathbf M_{\mathbf X_2} - (\log n)^{\beta + o(1)})^\alpha}{\mathbf Z_{\mathbf X_2}} + \sqrt{2\log |\mathcal A_{2,\beta}| \cdot \frac{n(\mathbf M_{\mathbf X_2} - (\log n)^{\beta + o(1)})^\alpha}{\mathbf Z_{\mathbf X_2}} } \\
&= \frac{n\mathbf M_{\mathbf X_2}^\alpha}{\mathbf Z_{\mathbf X_2}} - (\log n)^{\beta + \alpha^2 - \alpha + o(1)} + (\log n)^{\frac{\alpha^2 + 0\vee (\beta+1-\alpha)}{2}}\,,
\end{align*}
and by comparing the loss and gain terms we get that the critical value $\beta_* = 1 + \alpha - \alpha^2$. When $\alpha^2 < 2$, we have $\beta_* > \alpha - 1$, in which case the winner in the third round comes from $\mathcal A_{2,\beta^*}$. However, when $\alpha^2 > 2$, we actually have $\beta_* < \alpha - 1$. Therefore, $\max_{i \in \mathcal A_{2,\beta}}\mathbf X_3^{(i)}$ is decreasing in $\beta$ for $\beta \geq \alpha - 1$, and has expected values lower than $\max_{i \in \mathcal A_{2,0}}\mathbf X_3^{(i)}$. Thus, when $\sqrt{2} < \alpha < 2$ the second round winner wins the third round, while when $\alpha < \sqrt{2}$ this is not the case. Also, we can similarly update our $\delta$ and sizes of level sets by
\begin{align*}
\delta \mbox{ at third step }&= \frac{(\log n)^{\alpha^2 - 1 + o(1)}}{\mathbf M_{\mathbf X_3}} = (\log n)^{-1 + o(1)}\,, \\
\log |\mathcal A_{3,\beta}| &= (\log n)^{0\vee(\beta + \alpha - \alpha^2 - (\alpha - 1)) + o(1)}\mbox{ for }\beta > 0\,.
\end{align*}
Therefore, if we continue our analysis inductively, we will find that the eventual winner is determined in round $k$ when $\alpha \in (2^{\frac{1}{k}}, 2^{\frac{1}{k-1}}) = (\lambda_k , \lambda_{k-1})$ for $k \geq 2$, which matches our Theorem~\ref{thm-all-or-nothing-transition} when $1 < \alpha < 2$. Also, the initial $\delta_0$ is always equal to $(\log n)^{-1 + o(1)}$ before the voting dynamics enters the heuristic zone. Thus
\begin{align*}
\mathcal T = (1+o(1))\cdot \frac{\log(\delta_0^{-1}) +\log \log n}{\log \alpha} = (1+o(1))\cdot \frac{2 \log \log n}{\log \alpha}\,, \quad 1 < \alpha < 2\,,
\end{align*}
which matches our Theorem~\ref{thm-rapid-ordering-supercritical}.

In our actual proof, we need to carry out a similar analysis that simultaneously tracks the level set sizes $\log |\mathcal A_{t,\beta}|$, the relative gap $\delta$, the orders of $\mathbf M_{\mathbf X_t}$ and $\mathbf Z_{\mathbf X_t}$, and the critical parameter $\beta_*$ that maximizes $\max_{i \in \mathcal A_{t,\beta}}\mathbf X_{t+1}^{(i)}$. This analysis will be carried out by Lemma~\ref{lem-inductive-ecal-h-t}, which is the most technically difficult lemma in our paper. To rigorously prove Lemma~\ref{lem-inductive-ecal-h-t} by our intuitive analysis above, we need Lemma~\ref{lem-maximum-core-level-properties}, which is the core lemma that characterizes the empirical distribution of i.i.d. Poisson variables and describes the sizes of each level set. This lemma will be proved by a precise mass estimate of Poisson variables near their expected values (see Lemma~\ref{lem-Poisson-dedicate}). Also, to deal with critical values of $\alpha$ (i.e. $\alpha = \lambda_k$ for some $k \geq 2$) in Theorem~\ref{thm-rapid-ordering-supercritical}, we need Lemma~\ref{lem-Poisson-maximum-1st-minus-2nd} again that characterizes the relative gap between the maximum and the second maximum for heterogeneous Poisson random variables. This lemma will also be useful in analyzing the relative gaps for non-critical values of $\alpha$. The full proof is carried out in Section~\ref{subsec:case-k-geq-2}.

\subsection{Proof strategies of Theorems~\ref{thm-slow-ordering-sub-critical}}{\label{subsec:strat-thm-slow}}

The proof strategy for Theorem~\ref{thm-slow-ordering-sub-critical} is relatively direct. To show an exponential absorption time, it suffices to observe that when a candidate has at most $\frac{3}{4}n$ votes, the expected number of votes it will receive in the next round is $(\frac{3}{4} - \Omega(1))n$ since $\alpha < 1$. Therefore, there is only exponentially small probability that it gets more than $\frac{3}{4}n$ votes in the next round, which completes our proof. The full proof is carried out in Section~\ref{sec:proof-subcritical-slow}.

\section{Preliminary Lemmas}{\label{sec:prelims}}

In this section, we will introduce preliminary lemmas that are ingredients for proving our main theorems.

\subsection{Delicate properties of Poisson distributions}{\label{subsec:estim-Poissons}}

In this subsection, we give some delicate properties of Poisson distributions. The next lemma gives precise estimates on the probability of a Poisson variable on certain points, along with the ratio of probability on different points. The proof follows from the standard Stirling's formula, so we will omit further details for simplicity.
\begin{lemma}{\label{lem-Poisson-dedicate}}
    Suppose $\lambda=\omega(1)$, $X \sim \mathsf{Pois}(\lambda)$ and $|t|=o(\sqrt\lambda)$. 
    \begin{enumerate}
        \item[(1)] If $\lambda + t\sqrt{\lambda} \in \mathbb Z$, then
        \begin{align*}
             \tfrac{1}{10\sqrt{\lambda}} \exp\Big(- \big( 1+O\big( \tfrac{|t|+1}{\sqrt \lambda} \big) \big) \tfrac{t^2}{2}\Big) \leq \Pb\left( X=\lambda+t\sqrt{\lambda} \right) \leq \tfrac{10}{\sqrt{\lambda}} \exp\Big(- \big( 1+O\big( \tfrac{|t|+1}{\sqrt \lambda} \big) \big) \tfrac{t^2}{2} \Big) \,.
        \end{align*} 
        \item[(2)] If $\omega(1)\leq s\leq t$ and $\lambda + s\sqrt{\lambda},\lambda + t\sqrt{\lambda} \in \mathbb Z$, then
        \begin{align*}
            \exp \left( \left(1-\tfrac{t}{\sqrt\lambda} \right) \tfrac{t^2-s^2}{2}\right) \leq \frac{\Pb(X=\lambda + s \sqrt{\lambda})}{\Pb(X=\lambda + t \sqrt{\lambda})} \leq \exp \left( \left(1+\tfrac{t}{\sqrt\lambda} \right) \tfrac{t^2-s^2}{2} \right) \,.
        \end{align*}
    \end{enumerate}
\end{lemma}

As a corollary, we get a concentration inequality for the maximal value of i.i.d. Poisson variables:
\begin{corollary}{\label{cor-est-maximum-Poisson}}
    Suppose $m=\omega(1) \in \mathbb Z_+$, $m' \in [m]$ and we have $m$ independent random variables $X_1,\ldots,X_{m} \sim \mathsf{Pois(\lambda)}$ where $\lambda = \omega(1)$. Then, for $h > 3$ such that $\frac{\lambda}{h\log \lambda}>\log m$ and $h < \sqrt{\log \lambda}$ we have
    \begin{equation}{\label{eq-rough-Poisson-maximum}}
        \begin{aligned}
            &\Pb\left( \max_{1 \leq i \leq m} \{ X_i \} \leq \lambda + \sqrt{\lambda\log m} \right) \leq \exp(-m^{1/3}) \,; \\
            &\Pb\left( \max_{1 \leq i \leq m'} \{ X_i \} \geq \lambda + \sqrt{\lambda h\log m} \right) \leq \frac{1}{m^{(h-3)/2}} \,.
        \end{aligned}
    \end{equation}
\end{corollary}
\begin{proof}
    It suffices to prove the claims for the case $m' = m$. Using Lemma~\ref{lem-Poisson-dedicate}, we have
    \begin{align*}
        \Pb\left( X_i > \lambda + \sqrt{\lambda \log m} \right) \geq  \Pb\left( X_i = \lfloor\lambda + \sqrt{\lambda \log m}\rfloor + 1 \right) > \frac{1}{m^{0.51}}\,.
    \end{align*}
    Also, by the standard Chernoff bound (see, e.g., \cite[Theorem~2.9 and Equation~(2.10)]{BLM13}) we have
    \begin{align*}
     &\Pb\left( X_i \geq \lambda + \sqrt{\lambda h\log m} \right) \leq \exp\left(-\frac{h\log m}{2(1+\sqrt{\tfrac{h\log m}{\lambda}})}\right)\\
     \leq\ & \exp\left(-\frac{h\log m}{2(1+\sqrt{\tfrac{1}{\log \lambda}})}\right)\leq \exp\left(-\frac{h\log m}{2(1+\tfrac{1}{h-1})}\right) = \frac{1}{m^{(h-1)/2}}\,.
    \end{align*}
    Therefore, we have
    \begin{align*}
        &\Pb\left( \max_{1 \leq i \leq m} \{ X_i \} \leq \lambda + \sqrt{\lambda\log m} \right) \leq \left( 1-\frac{1}{m^{0.51}} \right)^m \leq \exp(-m^{1/3}) \,, \\
        &\Pb\left( \max_{1 \leq i \leq m} \{ X_i \} \geq \lambda + \sqrt{\lambda h\log m} \right) \leq \frac{m}{m^{(h-1)/2}} = \frac{1}{m^{(h-3)/2}} \,.  
    \end{align*}
\end{proof}

The next lemma is a core lemma for estimating the number of i.i.d. Poisson variables on different layers. The proof is based on the concentration inequalities and the probability estimates on Poisson variables that we have derived previously.

\begin{lemma}{\label{lem-maximum-core-level-properties}}
    Suppose $\lambda = \omega(1)$ and $X_1 \,, \ldots \,, X_m \sim \mathsf{Pois(\lambda)}$ are $i.i.d.$ random variables, where $\exp\left(2\exp(\sqrt{\log \lambda})\right) < m < \exp(\lambda^{1-\epsilon})$ for some constant $\epsilon \in (0,1)$. Define
    \begin{align}{\label{eq-def-S-lambda-m}}
        S(\lambda,m) = \frac{ \sqrt{\lambda} \exp(\sqrt{\log\lambda}) }{ \sqrt{\log m} } \,,
    \end{align}
    then we have the following facts:
    \begin{enumerate}
        \item[(1)] There exists $h_*(\lambda , m) \in [\lambda + \sqrt{1.9\lambda\log m}, \lambda + \sqrt{2.1\lambda\log m}]$ such that
        \begin{align}{\label{eq-typical-1st-minus-2nd-Poisson-1}}
            \Pb\left( |\max_{1\leq j \leq m}X_j - h_*(\lambda,m)| > S(\lambda,m) \right) \leq 2\exp\left(-\frac{\exp(\sqrt{\log \lambda})}{2}\right) \,.
        \end{align}
        \item[(2)] For all $v \in \mathbb Z_{\geq 0}$, we have
        \begin{align}
            \Pb\left( \#\{i:X_i=\max_{1\leq j \leq m}X_j - v\} \geq \exp\left( 10\max\left\{ v\sqrt{ \log m/\lambda}, \exp(\sqrt{\log \lambda}) \right\} \right) \right) \nonumber \\
            \leq 2\exp \left(-\frac{\exp(\sqrt{\log \lambda})}{2} \right) \,.  \label{eq-typical-1st-minus-2nd-Poisson-2}
        \end{align}
        \item[(3)] For all $v \in \mathbb Z$ such that $4S(\lambda,m) \leq v \leq \frac{1}{10} \sqrt{\lambda\log m}$, we have
        \begin{align}{\label{eq-typical-1st-minus-2nd-Poisson-3}}
            \Pb\left( \#\{i:X_i=\max_{1\leq j \leq m}X_j - v\} \leq \exp\left(\frac{v\sqrt{\log m}}{4\sqrt{\lambda}}\right) \right) \leq 2\exp\left(-\frac{\exp(\sqrt{\log \lambda})}{2}\right)\,.
        \end{align}
    \end{enumerate}
\end{lemma}
\begin{proof}
We first prove (1). By Item~(1) of Lemma~\ref{lem-Poisson-dedicate}, there exists $s_*,t_*$ such that $|s_*-\sqrt{1.9\lambda\log m}|\leq 1$ and $|t_* -\sqrt{2.1\lambda\log m}|\leq 1$ and 
\begin{align*}
\Pb(X_1 = \lambda + s_*) \geq \tfrac{1}{m^{0.99}}\,,\quad \Pb(X_1 = \lambda + t_*) \leq  \tfrac{1}{m^{1.01}}\,.
\end{align*}
Also, we have $\tfrac{\Pb(X_1 = h+1)}{\Pb(X_1=h)} = 1 + O(\sqrt{\tfrac{\log m}{\lambda}})$ for $h \in [\lambda + s_* , \lambda + t_*]$. Therefore, there exists $h_*:=h_*(\lambda, m) \in [\lambda + s_* , \lambda + t_*]$ such that $\Pb(X_1=h_*) \in( \tfrac{1}{m}, \tfrac{2}{m})$. Next, for each $u \geq \lambda + \sqrt \lambda$, note that
\begin{align}
1 \leq \tfrac{\Pb(X_i \geq u)}{\Pb(X_i = u)} = \sum_{i \geq 0}\tfrac{\Pb(X_i = u+i)}{\Pb(X_i = u)} \leq \sum_{i \geq 0}(1+\tfrac{1}{\sqrt \lambda})^{-i} \leq \sqrt \lambda + 1\,.{\label{eq-simple-geq-equal-bounds}}
\end{align}
Therefore, we have
\begin{align*}
\Pb(\max_{1\leq i\leq m}X_i \leq h_* - S(\lambda , m)) &= (1-\Pb(X_1 \geq h_*+1 - S(\lambda , m)))^m \\
&\leq \left(1-\exp(\tfrac{(h_*-\lambda)^2-(h_*-\lambda+1- S(\lambda , m))^2}{3\lambda })\Pb(X_1=h_*) \right)^m \\
&\leq \left(1-\exp(\tfrac{\exp(\sqrt {\log \lambda})}{3})\cdot\tfrac{1}{m}\right)^m \leq \exp(-\tfrac{\exp(\sqrt {\log \lambda})}{2}) \,,\\
\end{align*}
and
\begin{align*}
\Pb(\max_{i}X_i \geq h_* + S(\lambda , m)) &\leq m\Pb(X_1 \geq h_* + S(\lambda , m)) \\
&\leq m(\sqrt \lambda + 1)\exp(\tfrac{(h_*-\lambda)^2-(h_*-\lambda+S(\lambda , m))^2}{\lambda })\Pb(X_1=h_*) \\
&\leq 2(\sqrt \lambda + 1) \exp(-\exp(\sqrt {\log \lambda})) \leq\exp(-\tfrac{\exp(\sqrt {\log \lambda})}{2}) \,,
\end{align*}
where the inequalities hold by Item (2) of Lemma~\ref{lem-Poisson-dedicate}, which proves (1). 

Next we prove (2) and (3). By Item (2) of Lemma~\ref{lem-Poisson-dedicate} and \eqref{eq-simple-geq-equal-bounds}, for each $w \in \mathbb Z$ such that $|w| \leq S(\lambda , m)$ we have
\begin{align*}
\Pb(X_1 = h_* + w - v) 
\leq \Pb(X_1 = h_*)\exp(\tfrac{(h_*-\lambda)^2-(h_*-\lambda+w-v)^2}{3\lambda })\leq \tfrac{2}{m}\exp(4(v\vee S(\lambda , m))\sqrt{\tfrac{\log m}{\lambda}})\,,
\end{align*}
\begin{align*}
\Pb(X_1 = h_* + w - v)
\geq \Pb(X_1 = h_*)\exp(\tfrac{(h_*-\lambda)^2-(h_*-\lambda+w-v)^2}{\lambda }) \geq \tfrac{1}{m}\exp((v- S(\lambda , m))\sqrt{\tfrac{\log m}{\lambda}})\,,
\end{align*}
Therefore, for each $v \in \mathbb Z$ such that $2S(\lambda, m) \leq v \leq \tfrac{1}{10}\sqrt{\lambda \log m}$ we have
\begin{align}
&\quad\,\,\Pb\left( \#\{i:X_i = \max_j X_j - v\}\leq \exp \Big(\tfrac{v\sqrt{\log m/\lambda}}{4}\Big) \right)\nonumber \\
&\leq \Pb \left( \exists |w| \leq S(\lambda,m), \#\{i:X_i = h_* + w - v\}\leq\exp\Big(\tfrac{v\sqrt{(\log m)/\lambda}}{4}\Big) \right) + \Pb(\max_j X_j > h_* + S(\lambda,m))\nonumber\\
&\leq \exp\Big(-\tfrac{\exp(\tfrac{1}{10}v\sqrt{(\log m)/\lambda})}{10}\Big) + \exp(-\tfrac{\exp(\sqrt{\log \lambda})}{2}) \leq 2\exp(-\tfrac{\exp(\sqrt{\log \lambda})}{2})\,,{\label{eq-number-of-competitive-candidates-lower-bound}}
\end{align}
where the second inequality holds by the standard binomial Chernoff bound.

Similarly, for each $v \in \mathbb Z$ such that $0 \leq v \leq \tfrac{1}{10}\sqrt{\lambda \log m}$, denoting $S_{v,\lambda,m} = v \vee S(\lambda , m)$ we have
\begin{align}
&\quad\,\,\Pb(\#\{i:X_i \geq \max_j X_j - v\}\geq\exp(10S_{v,\lambda,m}\sqrt{(\log m)/\lambda}))\label{eq-level-set-prob-lb-original} \\
&\leq \Pb(\#\{i:X_i \geq h_* - w - v\}\geq\exp(10S_{v,\lambda,m}\sqrt{(\log m)/\lambda})\mbox{ for some }|w| \leq S(\lambda,m)) \nonumber \\
&+ \Pb(\max_j X_j < h_* - S(\lambda,m))\nonumber\\
&\leq \exp(-\tfrac{\exp(S_{v,\lambda,m}\sqrt{(\log m)/\lambda})}{5}) + \exp(-\tfrac{\exp(\sqrt{\log \lambda})}{2}) \leq 2\exp(-\tfrac{\exp(\sqrt{\log \lambda})}{2})\,,{\label{eq-number-of-competitive-candidates-upper-bound}}
\end{align}
where the second inequality holds by the standard binomial Chernoff bound. Also, for each $v \in \mathbb Z$ such that $v >  \tfrac{1}{10}\sqrt{\lambda \log m}$ we have
\begin{align}
\eqref{eq-level-set-prob-lb-original}\leq \Pb(\#\{i:X_i \geq \max_j X_j - v\}> m) = 0\,.\label{eq-number-of-competitive-candidates-upper-bound-2}
\end{align}
Combining \eqref{eq-number-of-competitive-candidates-lower-bound}, \eqref{eq-number-of-competitive-candidates-upper-bound} and \eqref{eq-number-of-competitive-candidates-upper-bound-2} proves (2) and (3).
\end{proof}

The next lemma bounds the typical value of the difference of the maximal value and the second maximal value among inhomogeneous independent Poisson variables.

\begin{lemma}{\label{lem-Poisson-maximum-1st-minus-2nd}}
    Suppose that $\lambda=\omega(1)$ and $m \in \mathbb N$, and $\eta_i \sim \mathsf{Pois}(a_i)$ are independent Poisson variables with $\frac{\lambda}{2} \leq a_1,\ldots,a_m \leq \lambda$ and $\max_{1 \leq i \leq m} a_i = \lambda$. Denote by $\eta_{(1)} \geq \eta_{(2)} \geq \ldots \geq \eta_{(m)}$ the order statistics of $(\eta_1,\eta_2,\ldots,\eta_m)$. Then we have the following estimates:
    \begin{enumerate}
        \item[(1)] Suppose that $2 \leq m\leq \exp(\frac{0.01\lambda}{(\log \lambda)^3})$. Then, we have
        \begin{align}
            \Pb\left( \eta_{(1)} - \eta_{(2)} \leq \frac{\sqrt{\lambda}}{(\log\lambda)\sqrt{\log m}} \right) \leq \frac{1}{\sqrt{\log\lambda}} \,.  \label{eq-Poisson-1st-minus-2nd-lb}
        \end{align}
        \item[(2)] Suppose that $\exp((\log\lambda)^{2+2c}) \leq m \leq \exp(\frac{0.01\lambda}{(\log \lambda)^3})$ for some $c=\Theta(1)>0$. In addition, suppose that there exists $\lambda_* \in [\frac{\lambda}{2}, \lambda]$ such that $\log \#\{j: a_j=\lambda_*\} \geq \frac{\log m}{(\log\lambda)^c}$ and
        \begin{align*}
            \Pb\Big( \mbox{exists }i\in[m]\,, a_i \leq \lambda_* + \sqrt{\lambda_* (\log \lambda)^{(c+2)/2}}\,, \eta_i = \max_{1 \leq j \leq m} \eta_j \Big) &\geq 1 - \frac{1}{\log \lambda}\,. 
        \end{align*}
        Then we have
        \begin{align}
            \Pb\left( \eta_{(1)} - \eta_{(2)} \geq \frac{ (\log \lambda)^{c+1} \sqrt{\lambda} }{ \sqrt{\log m} } \right) \leq \frac{1}{\sqrt{\log\lambda}}\,.\label{eq-Poisson-1st-minus-2nd-ub}
        \end{align}
    \end{enumerate}
\end{lemma}
\begin{proof}
(1) Define the set $\Lambda := \{\eta_{(1)} - \eta_{(2)} \leq \tfrac{1}{\log \lambda}\sqrt{\frac{\lambda}{\log m}}\}$,
and for $u = (u_1 , \ldots , u_m)$ we define
\begin{align*}
k_1(u) = \min \{i: u_i = u_{(1)}\}\,, k_{2}(u) = \min\{i\neq k_1(u): u_i = u_{(2)}\}\,.
\end{align*}
We write $\eta$ as a shorthand for $(\eta_1,\dots,\eta_m)$.
To show that $\Pb(\eta \in \Lambda) \leq \tfrac{1}{\sqrt{\log \lambda}}$, we will map each element in $\Lambda$ to $\Theta(\log \lambda)$ distinct elements in $\Lambda^c$ with similar probabilities. For each $u \in \Lambda$, define
\begin{align*}
\mathtt S(u) &= \{(u_{i}+2\ell\cdot [\tfrac{1}{\log \lambda}\sqrt{\tfrac{\lambda}{\log m}}]\cdot\mathbf{1}_{\{i=k_1(u)\}})_{1 \leq i \leq m}: 1 \leq \ell \leq \lfloor\log \lambda\rfloor\}\,,
\end{align*}
then $\{\mathtt S(u)\}_{u \in \Lambda}$ are disjoint subsets of $\Lambda^c$. Also, for $u\in \Lambda$ and $v \in \mathtt S(u)$ we have
\begin{align}
\tfrac{\Pb(\eta=u)}{\Pb(\eta=v)} \leq \max_{1 \leq i \leq 2\sqrt{\tfrac{\lambda}{\log m}}} \tfrac{\Pb(\eta_{k_1(u)}=u_{k_1(u)})}{\Pb(\eta_{k_1(u)}=u_{k_1(u)}+i)} \leq \Bigg(\tfrac{u_{k_1(u)}+2\sqrt{\tfrac{\lambda}{\log m}}}{a_{k_1(u)}}\Bigg)^{2\sqrt{\tfrac{\lambda}{\log m}}}\,,\label{eq-Poisson-lr}
\end{align}
Next, define $\Ecal_{\text{small}} = \{\frac{\eta_j-a_j}{\sqrt{a_j}} \leq 10\sqrt{\log (m\vee \log \lambda)}, \forall 1 \leq j \leq m\}$, then we have
\begin{align*}
\Pb(\Ecal_{\text{small}}^c)&\leq m\max_{1 \leq i \leq m}\Pb(\mathsf{Pois}(a_i)\geq a_i + 10\sqrt{a_i \log (m\vee \log \lambda)}) \\
&\leq m\exp(-5\log (m\vee \log \lambda)) \leq \tfrac{1}{(\log \lambda)^4}\,,
\end{align*}
Also, given $u\in \Lambda \cap\Ecal_{\text{small}}$ and $v \in \mathtt S(u)$ we have $a_{k_1(u)} \geq u_{k_1(u)} - 10\sqrt{a_{k_1(u)}\log m}$ and therefore
\begin{align*}
\eqref{eq-Poisson-lr} \leq \Big(1+\tfrac{10\sqrt{\log (m\vee \log \lambda)}}{\sqrt{a_{k_1(u)}}} + \tfrac{2}{a_{k_1(u)}}\sqrt{\tfrac{\lambda}{\log m}} \Big)^{2\sqrt{\frac{\lambda}{\log m}}} \leq e^{100\sqrt{\log \log \lambda}}\,.
\end{align*}
Thus we have
\begin{align*}
    \Pb(\eta \in\Lambda) &\leq \Pb(\Ecal_{\text{small}}^c ) + \sum_{u \in \Ecal_{\text{small}} \cap \Lambda}\Pb(\eta=u) \leq \tfrac{1}{(\log \lambda)^4} + \tfrac{e^{100\sqrt{\log \log \lambda}}}{\log \lambda}\sum_{u\in\Ecal_{\text{small}} \cap\Lambda}\Pb(\eta \in \mathtt S(u))\\
    &< \tfrac{1}{(\log \lambda)^4} + \tfrac{1}{2\sqrt{\log \lambda}} \leq \tfrac{1}{\sqrt{\log \lambda}}\,,
\end{align*}
which yields \eqref{eq-Poisson-1st-minus-2nd-lb}. 

(2) By symmetry, we can assume that $\{a_i\}_{1 \leq i \leq m}$ is non-decreasing (so that our subsequent claim on $k_1(u)$ becomes more transparent). Then, define the set
\begin{align*}
\Lambda_*  &= \{\eta_{(1)} - \eta_{(2)} \geq (\log \lambda)^{c+1}\cdot \sqrt{\tfrac{\lambda}{\log m}}\} \,,
\end{align*}
and for each $u \in \Lambda_*$ we define
\begin{align*}
\mathtt s_*(u) &= (u_{i}-[\tfrac{(\log \lambda)^{c+1}}{2}  \cdot\sqrt{\tfrac{\lambda}{\log m}}]\cdot\mathbf{1}_{\{i=k_1(u)\}})_{1 \leq i \leq m}\,,
\end{align*}
Then $\{\mathtt s_*(u)\}_{u \in \Lambda_*}$ are distinct vectors. Next, define the events
\begin{align*}
\Ecal^*_{\text{locate}} = \{\mbox{exists }i, a_i \leq \lambda_* + \sqrt{\lambda_* (\log \lambda)^{c/2 + 1}}, \eta_i = \max_{1 \leq j \leq m} \eta_j\}
\end{align*}
and
\begin{align*}
\Ecal^*_{\text{large}} = \{\max_{j: a_j = \lambda_*}\tfrac{\eta_j-a_j}{\sqrt{a_j}} \geq \sqrt{\tfrac{\log m}{(\log \lambda)^c}} \}\,,
\end{align*}
Then we have $\Pb\big((\Ecal^*_{\text{locate}})^c\big) \leq \tfrac{1}{\log \lambda}$ by our condition, and
\begin{align*}
\Pb\big((\Ecal^*_{\text{large}})^c\big) &\leq \exp(-|\#\{j:a_j = \lambda_*\}|^{1/3}) \leq \exp\big(-\exp((\log \lambda)^{c+1}/3)\big)\,,
\end{align*}
where the second inequality holds by Corollary~\ref{cor-est-maximum-Poisson} and our second condition. Also, given $u\in \Lambda_* \cap \Ecal^*_{\text{locate}} \cap \Ecal^*_{\text{large}}$ we have
\begin{align*}
u_{k_1(u)} &\geq  \lambda_* + \sqrt{\tfrac{\lambda_*\log m}{(\log \lambda)^c} } \geq a_{k_1(u)} + 0.9\sqrt{ \tfrac{\lambda_*\log m}{(\log \lambda)^c}  } \,,
\end{align*}
where the first inequality holds by $a_{k_1(u)} \leq \lambda_* + \sqrt{\lambda_* (\log \lambda)^{c/2 + 1}}$. Therefore, we have
\begin{align*}
\frac{\Pb(\eta = \mathtt s_*(u))}{\Pb(\eta = u)} &= \frac{\Pb(\eta_{k_1(u)}=u_{k_1(u)}-\frac{(\log \lambda)^{c+1}}{2}  \cdot\sqrt{\frac{\lambda}{\log m}})} {\Pb(\eta_{k_1(u)}=u_{k_1(u)})} \\
&\geq \Big( \frac{a_{k_1(u)} + 0.2\sqrt{\frac{\lambda_*\log m}{(\log \lambda)^c}}}{a_{k_1(u)}} \Big)^{\frac{(\log \lambda)^{c+1}}{2}  \cdot\sqrt{\frac{\lambda}{\log m}} } \geq \lambda^{0.01}\,,
\end{align*}
which gives
\begin{align*}
\Pb(\eta \in\Lambda_*) &\leq \Pb\big((\Ecal^*_{\text{locate}} \cap \Ecal^*_{\text{large}})^c\big) + \sum_{u \in \Ecal^*_{\text{locate}} \cap \Ecal^*_{\text{large}}\cap \Lambda_*}\Pb(\eta=u)\\
&\leq \tfrac{2}{\log \lambda} + \tfrac{1}{\lambda^{0.01}}\sum_{u\in\Ecal_*\cap\Lambda_*}\Pb(\eta =\mathtt s_*(u)) \leq \tfrac{2}{\log \lambda} + \tfrac{1}{\lambda^{0.01}} \leq \tfrac{1}{\sqrt{\log \lambda}}\,,
\end{align*}
which yields \eqref{eq-Poisson-1st-minus-2nd-ub}\,.
\end{proof}

\subsection{Coupling between multinomial and Poisson distributions}{\label{subsec:coupling}}

This subsection consists of three lemmas that deal with dependency between components of a multinomial distribution by coupling multinomial variables with independent Poisson variables of the same means. The first lemma ensures uniform bounded-by-constant errors between variables; the second lemma ensures comparable sizes of level sets in the coupling; the last lemma describes the vote distribution after the first round precisely by the coupling derived in the first two lemmas.

\begin{lemma}{\label{lem-Poisson-coupling}}
    Suppose that $(\zeta_1,\ldots,\zeta_n) \sim \mathsf{Mult}(n;\frac{a_1}{n},\ldots,\frac{a_n}{n})$ with $\sum_{i \in [n]} a_i = n$ and $\max\{ a_i \} \leq (\log n)^{A}$ for some constant $A \geq 10$. Sample independent random variables $\eta_1,\ldots,\eta_n$ such that $\eta_i \sim \mathsf{Pois}(a_i)$. Then there exists a coupling $\pi$ between $(\zeta_1,\ldots,\zeta_n)$ and $(\eta_1,\ldots,\eta_n)$ such that
    \begin{align*}
        &\pi\left( |\zeta_i-\eta_i| \leq 4 \mbox{ for all } 1 \leq i \leq n \right) \geq 1-\frac{1}{n} \,;  \\
        &\pi\left( \#\{i:\zeta_i \neq \eta_i\} \leq 2(\log n) \sqrt{n} \right) \geq 1-\frac{1}{n} \,.  
    \end{align*}
\end{lemma}
\begin{proof}
 Set \(p_i=a_i/n\), \(1\le i\le n\). Let \(\Lambda\) be a Poisson point process on \(\mathbb R_+\) with intensity \(n\,\mathrm dx\). Independently of \(\Lambda\), assign to each point \(x\in \Lambda\) a label \(L(x)\in[n]\), with
\[
    \mathbb P(L(x)=i)=p_i=\frac{a_i}{n},
    \qquad i=1,\ldots,n,
\]
independently over all points \(x\in\Lambda\). Let
\(\Lambda_\eta=\Lambda\cap[0,1],\) and define
\[
    \eta_i'=|\{x\in \Lambda_\eta: L(x)=i\}|,
    \qquad i=1,\ldots,n.
\]
By the thinning property of Poisson point processes, the random variables \(\eta_1',\ldots,\eta_n'\) are independent and satisfy $\eta_i'\sim \operatorname{Pois}(a_i)$. Therefore $(\eta_1',\ldots,\eta_n') \stackrel{d}{=}(\eta_1,\ldots,\eta_n)$.
Next, let \(\tau_n\) be the location of the \(n\)-th point of \(\Lambda\), and let $\Lambda_\zeta=\Lambda\cap[0,\tau_n]$ be the set of the first \(n\) points of \(\Lambda\). Using the same labels \(L(x)\), define
\[
    \zeta_i'=|\{x\in \Lambda_\zeta: L(x)=i\}|,
    \qquad i=1,\ldots,n.
\]
Since \(\Lambda_\zeta\) contains exactly \(n\) points and their labels are independent with law \((p_1,\ldots,p_n)\), we have $(\zeta_1',\ldots,\zeta_n') \stackrel{d}{=}(\zeta_1,\ldots,\zeta_n)$.
We define the coupling \(\pi\) to be the joint law of $\bigl((\zeta_1',\ldots,\zeta_n'), (\eta_1',\ldots,\eta_n')\bigr)$. We now claim that the coupling $\pi$ satisfies the requirements of the lemma. To this end, we define 
    \begin{align}\label{eq-def-Bfar}
        \mathcal B_{\text{far}} &= \left\{ \#(\Lambda_\zeta \triangle \Lambda_\eta) \geq 2(\log n)\sqrt{n} \right\} =\left\{ \left||\Lambda_\eta| -n\right| \geq 2(\log n)\sqrt{n} \right\}\nonumber\\
        &= \Big\{ \Big|\Big(\sum_{i=1}^n \eta_i\Big) -n\Big| \geq 2(\log n)\sqrt{n} \Big\}\,,
    \end{align}
    where the second inequality holds by the definitions of $\Lambda_\zeta$ and $\Lambda_\eta$. By standard Poisson Chernoff bound we have
    \begin{align}
        \pi(\mathcal B_{\text{far}}) = \pi\left( |\mathsf{Pois}(n)-n| \geq 2\sqrt{n}\log n \right) \leq 2\exp(-(\log n)^2) \,, \label{eq-poisson-tail-basic-estimate}
    \end{align}
    It is clear that under $\mathcal B_{\text{far}}^c$ we have $\#\{ i:\zeta_i' \neq \eta_i' \} \leq \#(\Lambda_\zeta \triangle \Lambda_\eta) \leq 2\sqrt{n}\log n$.
    Also, we have
    \begin{align}
        &\pi\left( \max_{1\leq i \leq n}\big|\zeta_i' - \eta_i'\big| > 4 \right) \leq \pi(\mathcal B_{\text{far}}) + \pi\left( \max_{1\leq i \leq n} \big|\zeta_i' - \eta_i'\big| > 4 \mid \mathcal B_{\text{far}}^c \right) \nonumber \\
        \overset{\eqref{eq-poisson-tail-basic-estimate}}{\leq}\ & 2\exp(-(\log n)^2) + \sum_{i=1}^{n} \pi\left(\big|\zeta_i' - \eta_i'\big| > 4 \mid \mathcal B_{\text{far}}^c \right) \nonumber \\
        \leq\ & 2\exp(-(\log n)^2) + \sum_{i=1}^{n} \Pb\left( \mathsf{Bin}\left( \#(\Lambda_\zeta \triangle \Lambda_\eta), \tfrac{a_i}{n} \right) > 4 \mid \mathcal B_{\text{far}}^c \right) \nonumber \\
        \leq\ & 2\exp(-(\log n)^2) + n\Pb\left( \mathsf{Bin}\left( 10\sqrt n \log n, \tfrac{(\log n)^{A}}{n} \right) > 4 \right) < \frac{1}{2n}\,.\label{eq-zeta-prime-similar-to-zeta}
    \end{align}
    This shows that $\pi$ is the desired coupling.
\end{proof} 

\begin{lemma}{\label{lem-Poisson-coupling-further}}
    Define $\mathcal K(h)=\{i:\eta_i' = h\}$. Let $\{a_i\}$ satisfy the condition in Lemma~\ref{lem-Poisson-coupling}. Under the coupling in Lemma~\ref{lem-Poisson-coupling} we have
    \begin{align*}
        \pi\left( \#\left( \{i:\zeta_i' = \eta_i' \} \cap \mathcal K(h) \right) \geq \frac{ |\mathcal K(h)| (1+\mathbf{1}_{ \{ |\mathcal K(h)|< n^{1/6} \} }) }{2} \mbox{ for all } h \right) \geq 1 - \frac{1}{n^{1/8}} \,.
    \end{align*}
\end{lemma}
\begin{proof}
   Define $\mathcal I_{\mathsf{eq}} =\{ i \in [n]: \zeta_i' = \eta_i'\}$\,, and define the event
   \begin{align}{\label{eq-def-Blarge}}
       \mathcal B_{\text{large}} = \left\{ \max_{1 \leq i \leq n} \eta_i' > 2(\log n)^{A} \right\} \,.
   \end{align}
   By Poisson Chernoff bound we have
   \begin{align}
       \Pb(\mathcal B_{\text{large}}) \leq \sum_{1 \leq i \leq n} \Pb\left( \eta_i'>a_i+(\log n)^{A} \right) \leq \exp(-(\log n)^{A/3})\,.\label{eq-low-prob-Bcal-large-lemA}
   \end{align}
   Combining this with \eqref{eq-poisson-tail-basic-estimate} (recall the definition of $\mathcal B_{\text{far}}$ in \eqref{eq-def-Bfar}), for all $h \in \mathbb Z_{\geq 0}$ such that $|\mathcal K(h)|>0$ we have
   \begin{align}
       &\pi\left( \#\left(\mathcal K(h) \cap\mathcal I_{\mathsf{eq}}\right) < \frac{|\mathcal K(h)|(1 + \mathbf 1_{\{|\mathcal K(h)| < n^{1/6}\}})}{2} \right)
       \leq \pi(\mathcal B_{\text{far}} \cup \mathcal B_{\text{large}}) \nonumber \\
       +\ & \pi\left( \left\{\#\left(\mathcal K(h) \cap\mathcal I_{\mathsf{eq}} \right) < \frac{|\mathcal K(h)|(1 + \mathbf 1_{\{|\mathcal K(h)| < n^{1/6}\}})}{2}\right\}\cap \mathcal B_{\text{far}}^c \cap \mathcal B_{\text{large}}^c \right) \nonumber \\
       \leq\ & n^{-100} + \pi\left(|\mathcal K(h) \cap\mathcal I_{\mathsf{eq}}| < \frac{|\mathcal K(h)|(1 + \mathbf 1_{\{|\mathcal K(h)| < n^{1/6}\}})}{2}; \mathcal B_{\text{far}}^c \cap \mathcal B_{\text{large}}^c \right) \,. \label{eq-zeta-prime-similar-to-zeta-2}
   \end{align}
   
   Next, define $\eta' = (\eta'_i)_{1\leq i \leq n}$ and $\zeta' = (\zeta'_i)_{1\leq i \leq n}$ for simplicity of notation. Note that by \eqref{eq-def-Bfar} and \eqref{eq-def-Blarge}, whether the events $\mathcal B_{\text{far}}$ and $\mathcal B_{\text{large}}$ happen or not depends solely on $\eta'$. Therefore, we will write $\eta^* \in \mathcal B_{\text{far}}$(or $\eta^* \in \mathcal B_{\text{large}}$) if and only if the event $\mathcal B_{\text{far}}$(or $\mathcal B_{\text{large}}$) happens given $\eta'=\eta^*$.

   Note that $\{|\mathcal K(h)|\}_{h \in \mathbb Z_{\geq 0}}$ is measurable with respect to $\eta'$ (i.e. $\{|\mathcal K(h)|\}_{h \in \mathbb Z_{\geq 0}}$ is fixed given $\eta'$). Therefore, in the following proof we will view $\{|\mathcal K(h)|\}_{h \in \mathbb Z_{\geq 0}}$ as fixed nonnegative integers conditioned on any specific value of $\eta'$. Then, we deal with \eqref{eq-zeta-prime-similar-to-zeta-2} by considering the following two cases separately.

   \noindent{\bf Case 1:} $|\Lambda_\eta| \leq n$. In this case, we have $\zeta' = \eta'+Y$, where $Y$ is independent of $\eta'$ conditioning on $\eta'$ and $Y \sim \mathsf{Mult}\left( n-|\Lambda_\eta|;\frac{a_1}{n},\ldots,\frac{a_n}{n} \right)$
   conditioning on $\eta'$. Thus, for any $h$ such that $|\mathcal K(h)| \geq n^{\frac{1}{3}}$ we have
   \begin{align}
       &\pi\left(|\mathcal K(h) \cap\mathcal I_{\mathsf{eq}}| < \frac{|\mathcal K(h)|}{2}; \mathcal B_{\text{far}}^c \cap \mathcal B_{\text{large}}^c\right) \nonumber \\
       \leq\ &\max_{\eta^*\in\mathcal B_{\text{far}}^c \cap \mathcal B_{\text{large}}^c}\pi\left(|\mathcal K(h) \cap\{i:Y_i \neq 0 \}| \geq \frac{|\mathcal K(h)|}{2}; \eta'=\eta^* \right) \nonumber \\
       \overset{\mathcal B_{\text{far}}}{\leq}\ &\Pb\left( \mathsf{Bin}\left( |\mathcal K(h)|,\frac{4\sqrt{n}(\log n)^{A}}{n} \right) \geq \frac{|\mathcal K(h)|}{2} \right) \leq 2\exp(-0.01n^{1/3}) \,. \label{eq-A.12-case-1-bound-1} 
   \end{align}
   In addition, for any $1\leq |\mathcal K(h)| < n^{\frac{1}{3}}$ we have 
   \begin{align}
       &\pi\left(|\mathcal K(h) \cap\mathcal I_{\mathsf{eq}}| < |\mathcal K(h)|; \mathcal B_{\text{far}}^c \cap \mathcal B_{\text{large}}^c \right) \le \pi\left(|\mathcal K(h) \cap\{i:Y_i \neq 0 \}| \geq |\mathcal K(h)|; \mathcal B_{\text{far}}^c \cap \mathcal B_{\text{large}}^c \right) \nonumber \\
       \overset{\mathcal B_{\text{far}}}{\leq}\ &\Pb\left( \mathsf{Bin}\left( |\mathcal K(h)|,\tfrac{4\sqrt{n}(\log n)^{A}}{n} \right) \geq 1 \right) \leq n^{-\frac{1}{7}} \,. \label{eq-A.12-case-1-bound-2} 
   \end{align}
   Combining \eqref{eq-A.12-case-1-bound-1} with \eqref{eq-A.12-case-1-bound-2} yields that (note that under $\mathcal B_{\text{large}}^c$ we have $\max |\eta_i'|\leq 3(\log n)^A$)
   \begin{align}
       \eqref{eq-zeta-prime-similar-to-zeta-2} \leq n^{-100}+ \sum_{\substack{h \leq 3(\log n)^{A} \\ \mathcal K(h) \geq n^{1/3} }}  2\exp(-0.01n^{1/3}) + \sum_{\substack{h \leq 3(\log n)^{A} \\ \mathcal K(h) < n^{1/3} }} n^{-\frac{1}{7}} \leq n^{-\frac{1}{8}} \,.  \label{eq-K(h)-scatter-case-1}
   \end{align}
    
   \noindent{\bf Case 2:} $|\Lambda_\eta|>n$. In this case, we have $\zeta' = \eta'-Y$, where $Y$ is independent of $\eta'$ conditioning on $\eta'$ and $Y \sim \mathsf{Hypergeometric}\left( |\Lambda_\eta|-n;\eta'_1,\ldots,\eta'_n\right)$ conditioning on $\eta'$. Therefore, given any $i \in [n]$ and any $\eta^* \in \mathcal B_{\text{far}}^c \cap \mathcal B_{\text{large}}^c$ with $\sum_{1\leq j \leq n}\eta^*_j =k$ we have
    \begin{align}
        \pi(Y_i > 0 \,; \eta' = \eta^*) = 1 - \frac{\tbinom{k-\eta^*_i}{k-n}}{\tbinom{k}{k-n}} \leq 1 - (1-\tfrac{\eta_i^*}{n})^{k-n} \leq \tfrac{(k-n)\eta_i^*}{n} < n^{-\frac{1}{3}}\,,\label{eq-thetaj-sum-for-small}
    \end{align}
    also, for any $h \in \mathbb Z_+$ such that $n^{1/6} < |\mathcal K(h)| \leq 4\sqrt n \log n$ we have $h|\mathcal K(h)| < n^{3/5}$. Therefore, given any $i \geq\tfrac{|\mathcal K(h)|}{3}$ and any $\eta^* \in \mathcal B_{\text{far}}^c \cap \mathcal B_{\text{large}}^c$ with $\sum_{1\leq j \leq n}\eta^*_j =k$ we have
    \begin{align}
        &\pi(\sum_{j \in \mathcal K(h)}Y_j = i \,; \eta' = \eta^*) =\frac{\tbinom{h|\mathcal K(h)|}{i} \cdot \tbinom{k-h|\mathcal K(h)|}{k-n-i} }{\tbinom{k}{k-n}} =\frac{\tbinom{h|\mathcal K(h)|}{i} \cdot \tbinom{k-h|\mathcal K(h)|}{k-n-i} }{\tbinom{k}{h|\mathcal K(h)|}}\nonumber \\
        \geq\ & \mathbf 1_{\{i+1\leq k-n\}}\tfrac{(i+1)(n+i+1-h|\mathcal K(h)|)}{(h|\mathcal K(h)|-i)(k-n-i)} \cdot \frac{\tbinom{h|\mathcal K(h)|}{i+1} \cdot \tbinom{k-h|\mathcal K(h)|}{k-n-i-1} }{\tbinom{k}{h|\mathcal K(h)|}}\nonumber\\
        \geq\ & \mathbf 1_{\{i+1\leq k-n\}} \cdot\tfrac{\sqrt{n}}{6h \log n}\cdot\pi(\sum_{j \in \mathcal K(h)}Y_j = i +1\,; \eta' = \eta^*) \ge n^{\frac{1}{3}} \pi(\sum_{j \in \mathcal K(h)}Y_j = i +1\,; \eta' = \eta^*)\,,\nonumber
    \end{align}
which directly gives 
\begin{align}
    &\pi\Big(\sum_{j \in \mathcal K(h)}Y_j \geq \tfrac{|\mathcal K(h)|}{2}\,; \eta' = \eta^*\Big) \leq 2n^{-\frac{|\mathcal K(h)|}{18}} \pi\Big(\sum_{j \in \mathcal K(h)}Y_j = \lfloor\tfrac{|\mathcal K(h)|}{3}\rfloor\,; \eta' = \eta^*\Big) < \exp(-n^{\frac{1}{6}})\,.\label{eq-thetaj-sum-for-large}
\end{align}
Therefore, combining \eqref{eq-thetaj-sum-for-small} and \eqref{eq-thetaj-sum-for-large}, given any $\eta^* \in \mathcal B_{\text{far}}^c \cap \mathcal B_{\text{large}}^c$, if we define $\mathcal I^* = \{h : |\mathcal K(h)| > n^{\frac{1}{6}}\}$ (which is determined by $\eta'$) we have
   \begin{align}
&\pi\Big(\big\{|\mathcal K(h) \cap\mathcal I_{\mathsf{eq}}| < \tfrac{|\mathcal K(h)|}{2}(1 + \mathbf{1}_{\{|\mathcal K(h) | < n^{1/6}\}})\mbox{ for all }h\big\}^c;\eta' = \eta^*\Big)\nonumber\\
\leq\ & \pi\Big(|\mathcal K(h) \cap\{i:Y_i \neq 0 \}| < \tfrac{|\mathcal K(h)|}{2}\mbox{ for some }h \in \mathcal I^*\,; \eta' = \eta^*\Big)\nonumber \\
+\ & \pi\Big(|\mathcal K(h) \cap\{i:Y_i \neq 0 \}| < |\mathcal K(h)| \mbox{ for some }h \not\in  \mathcal I^*\,;\eta' = \eta^*\Big)\nonumber \\
\leq\ & \sum_{\substack{0 \leq h \leq 2(\log n)^{A} \\ \mathcal K(h) > n^{1/6} }} \pi(\sum_{j \in \mathcal K(h)}Y_j \geq \tfrac{|\mathcal K(h)|}{2}\,; \eta' = \eta^*)+\sum_{\substack{0 \leq h \leq 2(\log n)^{A} \\ \mathcal K(h) \leq n^{1/6} }} \sum_{j \in \mathcal K(h)}\pi(Y_j > 0\,;\eta' = \eta^*) \nonumber \\
\leq\ &  3(\log n)^{A}\exp(-n^{\frac{1}{6}}) + 3(\log n)^{A} n^{-\frac{1}{3}+\frac{1}{6}} < n^{-\frac{1}{7}}\,,\nonumber
\end{align}
which gives
   \begin{align}
       \eqref{eq-zeta-prime-similar-to-zeta-2} &\leq 2\exp(-(\log n)^2)\nonumber \\
        &+ \max_{\eta^* \in \mathcal B_{\text{far}}^c \cap \mathcal B_{\text{large}}^c} \pi\Big(\big\{|\mathcal K(h) \cap\mathcal I_{\mathsf{eq}}| \geq \tfrac{|\mathcal K(h)|(1 + \mathbf{1}_{\{|\mathcal K(h) | < n^{1/6}\}})}{2}\mbox{ for all }h\big\}^c;\eta' = \eta^*\Big) \nonumber \\     
       &\leq n^{-100} + n^{-\frac{1}{7}} <  n^{-\frac{1}{8}} \,.  \label{eq-K(h)-scatter-case-2}
   \end{align}
   Combining \eqref{eq-K(h)-scatter-case-1} and \eqref{eq-K(h)-scatter-case-2} we have shown \eqref{eq-zeta-prime-similar-to-zeta-2}.
\end{proof}

\begin{lemma}{\label{lem-level-set-concentration}}
    Suppose that $X =(X_1,\ldots,X_n) \sim \mathsf{Mult}(n;\frac{1}{n}, \ldots,\frac{1}{n})$. Define 
    \begin{align*}
        \mathbf M=\max_{1 \leq i \leq n}X_i\,, \quad Y_j = \#\{i: X_i = j\}\,.
    \end{align*}
    Then we define $\aleph>0$ to be the unique positive integer such that $\aleph! \leq n < (\aleph+1)!$, and define the events
    \begin{align*}
        &\mathcal D_{\mathsf{unique}} = \{ Y_{\mathbf M} = 1 \} \,,
        \quad \mathcal D_{\mathsf{max}} = \left\{ |\mathbf M-\aleph| \leq 2 \right\} \,, \\
        &\mathcal D_{\mathsf{lb}} = \left\{ Y_j\geq \frac{0.4n}{ej!} \mbox{ for all }0 \leq j \leq \aleph-1 \right\} \,,  \\
        &\mathcal D_{\mathsf{ub}} = \left\{ Y_j\leq \frac{10n}{e((j-4)\vee 0)!} \vee 9\aleph^8 \mbox{ for all } 0 \leq j \leq \aleph+2 \right\} \,,\\
    \end{align*}
    Then we have $\Pb(\mathcal D_{\mathsf{unique}}) \wedge \Pb(\mathcal D^c_{\mathsf{unique}}) \geq \frac{1}{30\aleph}$. 
    and
    \begin{align}\label{eq-bound-D-lb-ub-mx}
        \Pb(\mathcal D_{\mathsf{max}} \cap \mathcal D_{\mathsf{lb}} \cap \mathcal D_{\mathsf{ub}} | \mathcal D_{\mathsf{unique}}), \quad \Pb(\mathcal D_{\mathsf{max}} \cap \mathcal D_{\mathsf{lb}} \cap \mathcal D_{\mathsf{ub}} | \mathcal D_{\mathsf{unique}}^c) = 1 - o(1)\,.
    \end{align}
\end{lemma}
\begin{proof}
Suppose $Z_1, \ldots, Z_{n} \overset{\mathsf{i.i.d.}}{\sim} \mathsf{Pois}(1)$, $\mathbf M' = \max_{1 \leq i \leq n}Z_i$ and $W_j = \#\{i : Z_i = j \}$ for all $j \geq 0$. Then, by Lemmas~\ref{lem-Poisson-coupling} and \ref{lem-Poisson-coupling-further}, there exists a coupling of $(X_1 , \ldots, X_{n})$ and $Z_1, \ldots, Z_{n} \overset{\mathsf{i.i.d.}}{\sim} \mathsf{Pois}(1)$ such that 
\begin{align*}
\mathcal G_{\mathsf{couple}} := \Big\{&\left|\{i: X_i = Z_i\}\cap \{i: Z_i = j\}\right| \geq \tfrac{W_j}{2}(1 + \mathbf{1}_{\{W_j < n^{1/6}\}}) \mbox{ for all }j \geq 0,\,\,\\&|X_i - Z_i| \leq 4 \mbox{ for all }1 \leq i \leq n\,,\#\{i : X_i \neq Z_i\} \leq 2(\log n)\sqrt{n}
\Big\}
\end{align*}
holds with probability $1 - O(\tfrac{1}{n^{1/8}})$. Under $\mathcal G_{\mathsf{couple}}$, we have 
\begin{align}
\tfrac{W_j}{2}(1 + \mathbf{1}_{\{W_j < n^{1/6}\}}) \leq Y_j \leq \sum_{i \in [(j-4) \vee 0, j+4]} W_i \mbox{ for all } j \geq 0.\label{eq-relationship-Wj-Yj}
\end{align}
Next, using binomial Chernoff bound, for all $0 \leq j \leq (\aleph - 1)$ we have
\begin{align*}
\Pb(|W_j - \Eb[W_j]| \geq 0.1\Eb[W_j]) = \Pb(|\mathsf{Bin}(n, \tfrac{1}{ej!}) - \tfrac{n}{ej!} | \geq \tfrac{0.1n}{ej!}) \leq\exp(-\tfrac{n}{300ej!}) \leq \exp(-\tfrac{\aleph}{1000})\,,
\end{align*}
therefore, we define 
\begin{align*}
\mathcal G_{\mathsf{conc}} &= \{\tfrac{0.9n}{ej!} \leq W_j \leq \tfrac{1.1n}{ej!} \mbox{ for all }0 \leq j \leq \aleph - 1\}\,,\\
\mathcal G_{\mathsf{precise}} &= \{|\mathbf M' - \aleph| \leq 2\} \cap \{W_{\aleph}\vee W_{\aleph + 1} \vee W_{\aleph + 2} \leq \aleph^8\} \cap\mathcal G_{\mathsf{couple}} \cap \mathcal G_{\mathsf{conc}}\,,
\end{align*}
then we get 
\begin{align}
\Pb(\mathcal G_{\mathsf{conc}}) \geq 1 - \aleph\exp(-\tfrac{\aleph}{1000}) \geq 1 - \exp(-\tfrac{\aleph}{2000})\,. \label{eq-concentrate-Wjs}
\end{align}
Then we have
\begin{align}
&\quad\,\,\Pb(\mathcal{G}^c_{\mathsf{precise}}) \leq \Pb(\mathcal{G}^c_{\mathsf{couple}})+ \Pb(|\mathbf M' - \aleph| > 2) + \tfrac{\Eb[W_{\aleph} + W_{\aleph+1} +W_{\aleph+2} ]}{\aleph^8} + \Pb(\mathcal G^c_{\mathsf{conc}})\nonumber \\
&< \exp(-\tfrac{\aleph}{8}) + \Pb(\mathbf M' \geq \aleph + 3) + \Pb(\mathbf M' \leq \aleph - 2)+ \tfrac{3}{\aleph^7}+ \exp(-\tfrac{\aleph}{2000}) \nonumber \\
&\leq \tfrac{1}{\aleph^2} + (1-(1-\tfrac{1}{(\aleph + 3)!})^n) + \Pb(W_{\aleph - 1} = 0)+ \exp(-\tfrac{\aleph}{2000})\nonumber \\
&\leq \tfrac{1}{\aleph^2} + \tfrac{1}{\aleph^2} + 2\exp(-\tfrac{\aleph}{2000}) \leq \tfrac{4}{\aleph^2}\,.\label{eq-maximum-concentrate-for-multinom}
\end{align}
Next, define $\mathcal G_{\mathsf{unique}} = \{W_{\mathbf M'} =1 \}$,
then by definition of $\mathcal G_{\mathsf{precise}}$, under $\mathcal G_{\mathsf{precise}}$ we have $W_j \leq 1.1\Eb[W_j] < \aleph^8$ for all $\aleph - 6 \leq j \leq \aleph - 1$, and $W_j \leq \aleph^8$ also holds for $\aleph \leq j \leq \aleph + 2$. 

Since $\aleph^8 < n^{1/6}$, by definition of $\mathcal G_{\mathsf{couple}}$ and Lemma~\ref{lem-Poisson-coupling-further}, for all $l \geq \aleph - 6$ we have $\left|\{i: X_i = Z_i\}\cap \{i: Z_i = l\}\right| \geq W_l = \#\{i: Z_i = l\}$, which directly implies $\{i:Z_i = l\}\subset \{i:X_i = l\}$. Therefore, for all $l \geq \aleph - 2 \geq \aleph - 6 + 4$, we have 
\begin{align*}
&\quad\,\,\{i:X_i = l\} \setminus \{i:Z_i = l\} \subset \{i:X_i = l\} \cap \Big(\bigcup_{|h| \leq 4, h \in \mathbb Z \setminus \{0\}}\{i:Z_i = l + h\} \Big)\\
&\subset \{i:X_i = l\} \cap \Big( \bigcup_{|h| \leq 4, h \in \mathbb Z \setminus \{0\}} \{i:X_i = l + h\}\Big) = \emptyset\,.
\end{align*}
Therefore, we have
\begin{align}
\mathcal G_{\mathsf{precise}} \subset\{W_j = Y_j \mbox{ for all }j \geq \mathbf M' \geq \aleph - 2 \}\,. \label{eq-G-precise-equal-levels}
\end{align}
Thus, we have $\mathcal G_{\mathsf{precise}} \cap \mathcal D_{\mathsf{unique}} =\mathcal G_{\mathsf{precise}}\cap \mathcal G_{\mathsf{unique}}$ and $\mathcal G_{\mathsf{precise}} \cap\mathcal D_{\mathsf{unique}}^c = \mathcal G_{\mathsf{precise}} \cap \mathcal G_{\mathsf{unique}}^c$. Also, note that $\Pb(Z_j \geq \aleph + h) = \sum_{j \geq \aleph + h}\tfrac{1}{ej!} < \tfrac{1}{(\aleph + h)!}$ for all $h \geq 0$, which gives
\begin{align}
\Pb(\mathcal G_{\mathsf{unique}}) &\geq \Pb(W_{\aleph + 1} = 1, \sum_{i \geq \aleph + 2}W_{i} = 0) \geq \Pb(W_{\aleph + 1} = 1| \sum_{i \geq \aleph + 2}W_{i} = 1)\Pb( \sum_{i \geq \aleph + 2}W_{i} = 1) \nonumber\\
&\geq \tfrac{\tfrac{1}{e(\aleph + 1)!}}{\tfrac{1}{e(\aleph + 1)!} + \tfrac{1}{(\aleph + 2)!}} \cdot \tfrac{n}{e(\aleph + 1)!}(1 -\tfrac{1}{(\aleph + 1)!})^{n-1} \geq \tfrac{1}{20\aleph}\,,\label{eq-Gcal-unique-bound}
\end{align}
Also, since $f(x) := (1-x)^{n-1}(1+(n-1)x)$ is strictly decreasing on $(0,1)$, 
\begin{align}
    \Pb(\mathcal G_{\mathsf{unique}}^c) &\geq \Pb(W_{\aleph } \geq 2, \sum_{i \geq \aleph + 1}W_{i} = 0) \geq \Pb(W_{\aleph} \geq 2 \mid \sum_{i \geq \aleph + 1}W_{i} = 0)\Pb(\sum_{i \geq \aleph + 1}W_{i} = 0) \nonumber\\
    &\geq \Big(1-f\big(\Pb(Z_1 = \aleph|Z_1 \leq \aleph)\big)\Big) \cdot (1 -\tfrac{1}{(\aleph + 1)!})^n \nonumber\\
    &\geq \tfrac{1}{e}(1-(1+\tfrac{1}{e})e^{-\tfrac{1}{e} + o(1)}) \geq \Omega(1)\,.\label{eq-Gcal-unique-complement-bound}
\end{align}
Therefore, we have
\begin{align}
\Pb(\mathcal D_{\mathsf{unique}}) &\geq \Pb(\mathcal G_{\mathsf{precise}}\cap\mathcal D_{\mathsf{unique}}) = \Pb(\mathcal G_{\mathsf{precise}}\cap\mathcal G_{\mathsf{unique}}) \nonumber\\
&\geq \Pb(\mathcal G_{\mathsf{unique}})- \Pb(\mathcal G^c_{\mathsf{precise}}) > \tfrac{1}{30\aleph}\,,\label{eq-Dcal-unique-bound}\\
\Pb(\mathcal D_{\mathsf{unique}}^c) &\geq \Pb(\mathcal G_{\mathsf{precise}}\cap\mathcal D_{\mathsf{unique}}^c) = \Pb(\mathcal G_{\mathsf{precise}}\cap\mathcal G_{\mathsf{unique}}^c) \nonumber\\
&\geq \Pb(\mathcal G_{\mathsf{unique}}^c)- \Pb(\mathcal G^c_{\mathsf{precise}}) >\tfrac{1}{\aleph}\,.\label{eq-Dcal-unique-complement-bound}
\end{align}

Finally, to prove \eqref{eq-bound-D-lb-ub-mx}, we first show that $\mathcal G_{\mathsf{precise}} \subset\mathcal D_{\mathsf{max}}\cap \mathcal D_{\mathsf{lb}} \cap \mathcal D_{\mathsf{ub}}$. By \eqref{eq-G-precise-equal-levels} we have $\mathcal G_{\mathsf{precise}} \subset\{ \mathbf M' = \mathbf M\}$ and therefore $\mathcal G_{\mathsf{precise}} \subset \mathcal D_{\mathsf{max}}$. Moreover, on the event $\mathcal G_{\mathsf{precise}}$ we must have $W_j \leq \aleph^8$ for $j \geq \aleph$, and $\tfrac{0.9n}{ej!} \leq W_j \leq \tfrac{1.1n}{ej!}$ for all $0 \leq j \leq \aleph - 1$. By \eqref{eq-relationship-Wj-Yj} we immediately have $\mathcal G_{\mathsf{precise}} \subset\mathcal D_{\mathsf{lb}} \cap \mathcal D_{\mathsf{ub}}$. Therefore, we have
\begin{align}
&\quad\,\,\Pb((\mathcal D_{\mathsf{max}} \cap \mathcal D_{\mathsf{lb}} \cap\mathcal D_{\mathsf{ub}} )^c) \leq \Pb(\mathcal G^c_{\mathsf{precise}}) \leq \tfrac{4}{\aleph^2}\,.\label{eq-3events-for-multinom}
\end{align}
Combining \eqref{eq-3events-for-multinom} with \eqref{eq-Dcal-unique-bound} and \eqref{eq-Dcal-unique-complement-bound} we have
\begin{align*}
\Pb((\mathcal D_{\mathsf{max}} \cap \mathcal D_{\mathsf{lb}} \cap \mathcal D_{\mathsf{ub}})^c | \mathcal D_{\mathsf{unique}}) \vee \Pb((\mathcal D_{\mathsf{max}} \cap \mathcal D_{\mathsf{lb}} \cap \mathcal D_{\mathsf{ub}})^c | \mathcal D_{\mathsf{unique}}^c) \leq \frac{\tfrac{4}{\aleph^2}}{\tfrac{1}{30\aleph}} = o(1)\,,
\end{align*}
which proves \eqref{eq-bound-D-lb-ub-mx}, completing the proof of our lemma.
\end{proof}

\subsection{Estimates on maxima and partition functions}{\label{subsec:estim-M-Z}}

In this subsection, we present some estimates on maxima and partition functions during the voting process. The following lemma gives a rough bound on the maxima in the first $O(\log \log n)$ rounds.

\begin{lemma}{\label{lem-event-E-diamond}}
    Define
    \begin{align}{\label{eq-def-E-diamond}}
        \mathcal E_{\ref{lem-event-E-diamond}} = \left\{ \mathbf M_{\mathbf X_k} \leq \left( \frac{ 20\log(n) }{ \log\log(n) } \right)^{\alpha^{k}},\ 1 \leq k \leq \frac{\log\log(n)}{\log(\alpha)} \right\} \,. 
    \end{align}
    Then we have $\Pb_n^{\alpha}(\mathcal E_{\ref{lem-event-E-diamond}})=1-o(\frac{1}{n^{0.9}})$ when $\alpha>1$.
\end{lemma}
\begin{proof}
    Note that $(\mathbf X_1^{(1)},\ldots,\mathbf X_1^{(n)}) \sim \mathsf{Mult}(n;\frac{1}{n},\ldots,\frac{1}{n})$. Thus, by Lemma~\ref{lem-level-set-concentration} we have
    \begin{align*}
        \Pb_n^{\alpha}\left( \mathbf M_{\mathbf X_1} > \frac{10\log(n)}{\log\log(n)} \right) &\leq n \cdot \binom{n}{\tfrac{10\log n}{\log \log n}} \cdot(\frac{1}{n})^{\tfrac{10\log n}{\log \log n}} \\
        &\leq n\cdot((\tfrac{10\log n}{\log \log n})!)^{-1} = o(\frac{1}{n})\,,
    \end{align*}
    In addition, given any $(\mathbf X_k^{(1)},\ldots,\mathbf X_k^{(n)}) \in \Omega_n$ such that $\mathbf M_{\mathbf X_k} \leq M_k$ for some $M_k \geq \tfrac{10\log(n)}{\log\log(n)}$, we have (note that when $\alpha\geq 1$ we always have $\mathbf Z_{\mathbf X_k} \geq n$)
    \begin{align*}
        (\mathbf X_{k+1}^{(1)},\ldots,\mathbf X_{k+1}^{(n)}) \sim \mathsf{Mult}(n;a_1,\ldots,a_n) \mbox{ with } \max\{ a_i \} \leq \frac{M_k^{\alpha}}{n}\,.
    \end{align*}
    Therefore, when $M_k^\alpha < n$, by binomial Chernoff bound we have
    \begin{align*}
        &\Pb_n^{\alpha}\left( \mathbf M\left( \mathbf X_{k+1}^{(1)}, \ldots, \mathbf X_{k+1}^{(n)} \right) > 2M_k^{\alpha} \right) \\
        \leq\ & n\Pb\left( \mathsf{Bin}\left( n,\frac{M_k^{\alpha}}{n} \right)>2M_k^{\alpha} \right) \leq n\exp(-\frac{M_k^\alpha}{5}) < \frac{1}{n}\,,
    \end{align*}
    and when $M_k^\alpha \geq n$ we have $\Pb_n^{\alpha}\left( \mathbf M\left( \mathbf X_{k+1}^{(1)}, \ldots, \mathbf X_{k+1}^{(n)} \right) > 2M_k^{\alpha} \right) = 0$. Thus, using an inductive argument and a union bound, it is clear that $\Pb_n^{\alpha}(\mathcal E_{\ref{lem-event-E-diamond}})=1-o(\frac{1}{n^{0.9}})$, which finishes our proof. 
\end{proof}

The next lemma is a rough control of the partition functions and maxima in the first constant rounds.

\begin{lemma}{\label{lem-good-event-E-1}}
    Suppose $\alpha \in [1, +\infty)$ and $R$ is a fixed constant. Define $\mathcal E_{\ref{lem-good-event-E-1}}=\mathcal E_{\ref{lem-event-E-diamond}} \cap \mathcal E_{\star}$, where 
    \begin{align*}
         \mathcal E_{\star} = \left\{ \mathbf Z_{\mathbf X_k} \leq n(200\alpha)^{2R^2\alpha^R}, 1\leq k \leq R \right\} \,.
    \end{align*}
    We then have $\Pb^{\alpha}_n(\mathcal E_{\ref{lem-good-event-E-1}})=1-o(n^{-\tfrac{1}{2}})$.
\end{lemma}  

\begin{proof}
Note that on the event $\Ecal_{\ref{lem-event-E-diamond}}$ we have $\mathbf M_{\mathbf X_k} \leq (\log n)^{\alpha^k}$ for all $1 \leq k \leq R-1$. Let
\begin{align}
    \Ecal_{\ref{lem-event-E-diamond};\leq k} = \left\{ \mathbf M_{\mathbf X_i} \leq (\log n)^{\alpha^i} \mbox{ for all } 1 \leq i \leq k \right\}, \quad k \geq 1 \,.
\end{align}  
Since we have shown that $\Pb^{\alpha}_n(\mathcal{E}_{\ref{lem-event-E-diamond}}) = 1-o(n^{-0.9})$, it suffices to show that $\Pb^{\alpha}_n(\mathcal E_{\star}^c \cap \Ecal_{\ref{lem-event-E-diamond}; \leq R-1}) = o(n^{-\tfrac{1}{2}})$. Define the events
\begin{align}
    \Ecal_{\star;k} = \left\{ \sum_{i=1}^{n} \left( \mathbf X_k^{(i)} \right)^{\alpha^{R+1-k}} \leq n(200\alpha)^{2kR\alpha^R}  \right\}, \quad k \geq 0 \,,  \label{eq-def-Ecal-star-K}
\end{align}
and define $\Ecal_{\star;\leq k} = \bigcap_{0 \leq i \leq k} \Ecal_{\star;i}$ for any $k\geq 0$. Using H\"{o}lder's inequality, on $\Ecal_{\star;k}$ we have
\begin{align*}
    \mathbf Z_{\mathbf X_k} &\leq \left( \sum_{i=1}^n \left( X_k^{(i)} \right)^{\alpha^{R+1-k}} \right)^{\alpha^{k-R}} \cdot \left( \sum_{i=1}^n 1 \right)^{1-\alpha^{k-R}} \leq n (200\alpha)^{2kR\alpha^R}  \,, 
\end{align*}
and therefore $\Ecal_{\star;\leq R}\subset\Ecal_{\star}$. Let $\Ecal_{\mathsf{all};\leq k} = \Ecal_{\ref{lem-event-E-diamond};\leq k} \cap \Ecal_{\star;\leq k}$.
Then it suffices to prove that
\begin{align}
    &\Pb_n^{\alpha}\left( \Ecal_{\star;0}^c  \right) = 0 \,,  \label{eq-prob-E-star-0}  \\
    &\Pb^{\alpha}_n\left( \Ecal_{\star;k+1}^c \cap \Ecal_{\mathsf{all};\leq k} \right) = o(n^{-\tfrac{1}{2}}), 0 \leq k \leq R-1\,,\label{eq-induction-Ecal-star}\\
    &\Pb^{\alpha}_n\left( \Ecal_{\ref{lem-event-E-diamond};\leq k+1}^c \cap \Ecal_{\mathsf{all};\leq k} \right) = o(n^{-\tfrac{1}{2}}), 0 \leq k \leq R-1\,.\label{eq-induction-Ecal-trivial}
\end{align}
Note that \eqref{eq-prob-E-star-0} is straightforward by definition of $\Ecal_{\star ; 0}$. Also, \eqref{eq-induction-Ecal-trivial} automatically holds since $\Pb^{\alpha}_n(\Ecal_{\ref{lem-event-E-diamond};\leq k+1}^c \cap \Ecal_{\mathsf{all};\leq k}) \leq \Pb^{\alpha}_n(\mathcal{E}^c_{\ref{lem-event-E-diamond}}) = o(n^{-0.9})$. Thus, it suffices to prove \eqref{eq-induction-Ecal-star}. By an induction using \eqref{eq-prob-E-star-0} and \eqref{eq-induction-Ecal-trivial}, it suffices to prove that
\begin{align}
\Pb^{\alpha}_n\left( \Ecal_{\star;k+1}^c ;\Ecal_{\mathsf{all};\leq k} \right) = o(n^{-\tfrac{1}{2}}), 0 \leq k \leq R-1\,.\label{eq-induction-Ecal-star-conditional}
\end{align}
Define independent Poisson random variables
\begin{align*}
    \eta^{(i)}_{k+1} \sim \mathsf{Pois}\left( \frac{n(\mathbf X_k^{(i)})^\alpha}{\mathbf Z_{\mathbf X_k}} \right),\ \zeta_{k+1}^{(i)} \sim \mathsf{Pois}\left( (\mathbf X_k^{(i)})^\alpha \right), \quad 1 \leq i \leq n \text{ and }\mathbf{X}_k^{(i)}>0\,.
\end{align*}
Since $\mathbf Z_{\mathbf X_k} \geq n$ when $\alpha\geq 1$, $\zeta_{k+1}^{(i)}$ stochastically dominates $\eta^{(i)}_{k+1}$. Also, by Lemma~\ref{lem-Poisson-coupling} with $A = 10 \vee \alpha^{R+1}$, given $(\mathbf X_1,\ldots,\mathbf X_k)$ such that $\Ecal_{\mathsf{all};\leq k}$ holds, there exists a coupling $\pi_{k+1}$ between $(\mathbf X_{k+1}^{(i)})_{1 \leq i \leq n}$ and $(\eta_{k+1}^{(i)})_{1 \leq i \leq n}$ such that
\begin{align*}
    \pi_{k+1}\left( \max_{1\leq i \leq n}|\mathbf X^{(i)}_{k+1} - \eta^{(i)}_{k+1}| \leq 4 \right) \geq 1-\frac{1}{n} \,.
\end{align*}
Therefore, given $\Pb^{\alpha}_n(\Ecal_{\mathsf{all};\leq k} ) = 1-o(1)$ we have
\begin{align}
    & \Pb^{\alpha}_n\left( \Ecal_{\star;k+1}^c ;\Ecal_{\mathsf{all};\leq k} \right) \nonumber \\
    \leq\ & \frac{1}{n} + \Pb^\alpha_n\left( \sum_{i=1}^n \left( \eta^{(i)}_{k+1}+4 \right)^{\alpha^{R-k}} > (200\alpha)^{2(k+1)R\alpha^R} n; \Ecal_{\mathsf{all};\leq k} \right) \nonumber \\
    \leq\ & o(\frac{1}{\sqrt{n}}) + \Pb^\alpha_n\left( \sum_{i=1}^n \left( \zeta^{(i)}_{k+1}+4 \right)^{\alpha^{R-k}} > (200\alpha)^{2(k+1)R\alpha^R} n; \Ecal_{\mathsf{all};\leq k} \right) \,. \label{eq-Ecal-star-k+1-bound-part-1-step-0} 
\end{align}
Note that on $\Ecal_{\star,\leq k}$ we have
\begin{align*}
    &\sum_{i=1}^n\left( \zeta^{(i)}_{k+1}+4 \right)^{\alpha^{R-k}} \leq 3^{\alpha^{R-k}} \sum_{i=1}^n\left( \left( \zeta^{(i)}_{k+1} - (\mathbf X_{k}^{(i)})^{\alpha} \right)^{\alpha^{R-k}}\mathbf 1_{\{\mathbf X_{k}^{(i)} > 0\}} + (\mathbf X_{k}^{(i)})^{\alpha^{R-k+1}} + 4^{\alpha^{R-k}} \right) \,, \\
    &\sum_{i=1}^n (\mathbf X_{k}^{(i)})^{\alpha^{R-k+1}} \leq (200\alpha)^{2kR\alpha^R}n \,.
\end{align*} 
Thus
\begin{align*}
    &\sum_{i=1}^n \left( \zeta^{(i)}_{k+1}+4 \right)^{\alpha^{R-k}} > (200\alpha)^{2(k+1)R\alpha^R} n \\
    \Longrightarrow & \sum_{i=1}^n\left| \zeta^{(i)}_{k+1}- (\mathbf X_{k}^{(i)})^{\alpha} \right|^{\alpha^{R-k}} \mathbf 1_{\{\mathbf X_{k}^{(i)} > 0\}}> \left( 3^{-\alpha^{R-k}}(200\alpha)^{2(k+1)R\alpha^R} - 4^{\alpha^{R-k}} - (200\alpha)^{2kR\alpha^R} \right)n \\
    \Longrightarrow & \sum_{i=1}^n\left(\left| \zeta^{(i)}_{k+1}- (\mathbf X_{k}^{(i)})^{\alpha} \right|^{\alpha^{R-k}}\cdot \mathbf 1_{\{\mathbf X_{k}^{(i)} > 0\}}\right) > (200\alpha)^{2kR\alpha^R} n \,.
\end{align*}
Thus, we have \eqref{eq-Ecal-star-k+1-bound-part-1-step-0} is bounded by $o(\frac{1}{n})$ plus
\begin{align}
    \Pb^\alpha_n\left( \sum_{i=1}^n\left(\left| \zeta_{k+1}^{(i)} - (\mathbf X_{k}^{(i)})^{\alpha} \right|^{\alpha^{R-k}}\cdot \mathbf 1_{\{\mathbf X_{k}^{(i)} > 0\}}\right) > (200\alpha)^{(2k+1)R\alpha^R} n; \Ecal_{\mathsf{all};\leq k} \right) \,.  \label{eq-Ecal-star-k+1-bound-part-1}
\end{align}
Define $\iota_{k+1}^{(i)}=\frac{ |\zeta^{(i)}_{k+1}-(\mathbf X_{k}^{(i)})^\alpha |}{ (\mathbf X_{k}^{(i)})^{\frac{\alpha}{2}} }\mathbf 1_{\{\mathbf X_{k}^{(i)} > 0\}}$. By binomial Chernoff bound, for all $t>0$ and all $i$ such that $\mathbf X^{(i)}_{k}>0$ we have $\Pb( \iota_{k+1}^{(i)} \geq t ) \leq 2\exp(-(t\wedge t^2)/3)$. Thus, for all $s\geq 1$ we have
\begin{align}
    \Eb\left[ \left( \iota_{k+1}^{(i)} \right)^s \right] &\leq 1 + \int_{1}^{\infty} \Pb\left( \left(\iota_{k+1}^{(i)} \right)^s > u \right) \mathrm du = 1 + s\cdot 3^{s}\int_{1/3}^{\infty} \Pb\left( \iota_{k+1}^{(i)} > 3t \right)t^{s-1}dt\nonumber\\
    &\leq 1 + 2s\cdot 3^{s} \int_{0}^{\infty} e^{-t}t^{s-1}dt \leq 4 \cdot 3^{s} \Gamma(s+1)\,.
\end{align}
Therefore, by Cauchy-Schwarz inequality, we have
\begin{align}
    \eqref{eq-Ecal-star-k+1-bound-part-1} &\leq \Pb\left( \Big(\sum_{i=1}^n (\iota^{(i)}_{k+1})^{2\alpha^{R-k}}\Big)\Big(\sum_{i=1}^n (\mathbf X_k^{(i)})^{\alpha^{R+1-k}}\Big) \geq (200\alpha)^{(4k+2)R\alpha^R}n^2\right) \nonumber\\
    &\leq \Pb\left( \sum_{i=1}^n (\iota^{(i)}_{k+1})^{2\alpha^{R-k}} \geq (200\alpha)^{(2k+2)R\alpha^R}n \right)\nonumber\\
    &\leq \Pb\left( \sum_{i=1}^n \Big((\iota^{(i)}_{k+1})^{2\alpha^{R-k}} -\Eb[(\iota^{(i)}_{k+1})^{2\alpha^{R-k}}]\Big)\geq 4 \cdot 3^{2\alpha^{R-k}} \Gamma(2\alpha^{R-k}+1)n \right)\,,\label{eq-Ecal-star-k+1-bound-temp1}
\end{align}
where the third inequality holds by the fact that $(200\alpha)^{(2k+2)R\alpha^R}n \geq 8 \cdot 3^{2\alpha^{R-k}} \Gamma(2\alpha^{R-k}+1)n$. Thus, by Chebyshev's inequality we have
\begin{align*}
    \eqref{eq-Ecal-star-k+1-bound-temp1} &\leq \left( 4\cdot 3^{2\alpha^{R-k}} \Gamma(2\alpha^{R-k}+1)n \right)^{-2} \sum_{i=1}^{n}\mathbb{E}\left[ \left( \iota^{(i)}_{k+1} \right)^{4\alpha^{R-k}} \right] \\
    &\leq \frac{8\cdot 3^{4\alpha^{R-k}}\Gamma(4\alpha^{R-k}+1)n }{ (4\cdot 3^{2\alpha^{R-k}}\Gamma(2\alpha^{R-k}+1)n)^2 } = o(n^{-\tfrac{1}{2}})\,, 
\end{align*}
which yields \eqref{eq-induction-Ecal-star-conditional} and therefore proves \eqref{eq-induction-Ecal-star}. Combining \eqref{eq-prob-E-star-0}, \eqref{eq-induction-Ecal-star} and \eqref{eq-induction-Ecal-trivial} yields our desired result.

\end{proof}

\section{Determining the absorption time and the winner}{\label{sec:determining-T-and-W}}

In this section, we establish a key probabilistic argument that will be used repeatedly throughout the paper. Informally, it states that once a ``leader'' accumulates a sufficiently large gap comparing to other candidates, it will survive with high probability. Moreover, the remaining absorption time will be determined solely by the size of the gap. To formalize this, we first introduce the necessary notation. 
\begin{definition}{\label{eq-def-M-admissible}}
    For $M \in \mathbb Z_+$, we say $(\mathbf X_1,\ldots,\mathbf X_n) \in \Omega_n$ is an \emph{$M$-admissible} tuple, if 
    \begin{align*}
        \mathbf M_{\mathbf X}= \max_{i \in [n]}\{ \mathbf X_i \}=M \,.
    \end{align*}
    Denote by $\Omega_n(M)$ the set of all $M$-admissible tuples.
\end{definition}
\begin{definition}{\label{eq-def-(M,1-delta)-admissible}}
    For $M \in \mathbb Z_+$ and $0 \leq \delta \leq 1$, we say $(\mathbf X_1,\ldots,\mathbf X_n) \in \Omega_n$ is an \emph{$(M,\delta)$-admissible} tuple, if there exists $j \in [n]$ such that 
    \begin{align*}
        \mathbf X_j=M \mbox{ and } \mathbf X_i \leq (1-\delta)M \mbox{ for all } i \neq j \,.
    \end{align*}
    Denote by $\Omega_n(M,\delta)$ the set of all $(M,\delta)$-admissible tuples, and denote by $\overline{\Omega}_n (M,\delta)$ the set of tuples $\mathbf X=(\mathbf X_1,\ldots,\mathbf X_n)$ such that
    \begin{align}\label{eq-def-Omega-bar}
        (\mathbf X_1,\ldots,\mathbf X_n) \in \Omega_n (M,\delta), \quad \mathbf X_i = \lfloor(1-\delta)M\rfloor \mbox{ for some } i \in [n] \,.
    \end{align}
\end{definition}
\begin{definition}{\label{eq-def-(M,1-delta,Z)-admissible}}
    For $M \in \mathbb Z_+$, $0 \leq \delta \leq 1$ and $Z \in \mathbb R_+$, we say $(\mathbf X_1,\ldots,\mathbf X_n) \in \Omega_n$ is an \emph{$(M,\delta,Z)$-admissible} tuple, if 
    \begin{align*}
        (\mathbf X_1,\ldots,\mathbf X_n) \in \Omega_n(M,\delta) \mbox{ and } \mathbf Z_\mathbf X = \sum_{i \in [n]}\mathbf X_i^\alpha \leq Z \,.
    \end{align*}
    Denote by $\Omega_n(M,\delta,Z)$ the set of all $(M,\delta,Z)$-admissible tuples.
\end{definition}

The main goal of this section is to prove the following result:
\begin{proposition}{\label{lem-repeat-used}}
    Recall that we use $\mathcal W$ to denote the eventual winner of the election, and we use $\mathcal L(\mathbf X)$ to denote the set of leaders in a tuple $\mathbf X$. Suppose that $\alpha > 1$ and
    \begin{equation}{\label{eq-assum-lem-2.4}}
        M \geq (\log n)^{101}, \quad \delta \geq M^{-\frac{1}{2}} (\log n)^3, \quad (1-\delta)M \geq (\log n)^{20}\,, \quad 1-\delta >n^{-\tfrac{1}{100(\alpha + 1)}} \,.
    \end{equation}
    Define $\kappa=\kappa(n,\delta,\alpha)$ such that
    \begin{equation}{\label{eq-def-kappa}}
        \kappa = \frac{ \log\log(n) + \log(1/\delta) }{ \log(\alpha) } \,.
    \end{equation}
    Then, given any $(\mathbf X_k^{(1)},\ldots,\mathbf X_k^{(n)}) \in \overline{\Omega}_n(M,\delta)$, we have
    \begin{align*}
        \Pb_n^{\alpha}\left( \mathcal W \in \mathcal L(\mathbf X_k), \mathcal T=k+(1-o(1))\kappa \right)=1-o(1) \,.
    \end{align*}
\end{proposition}

\subsection{Key estimates during the evolution}

The key to proving Proposition~\ref{lem-repeat-used} is to step by step show a series of results that can be applied to different phases of $(M, \delta)$, starting from the initial conditions of $(M, \delta)$ given by \eqref{eq-assum-lem-2.4}. We start with the following lemma that brings $1-\delta$ to a ``relatively small'' phase:
\begin{lemma}{\label{lem-prelim-gap-est-time-after-2R}}
    Suppose that $\alpha > 1$ and
    \begin{equation}{\label{eq-assum-lem-2.5}}
        M \geq (\log n)^{100},\quad \delta \geq M^{-\frac{1}{2}}(\log n)^3, \quad (1-\delta)M \geq (\log n)^{12}, \quad 1-\delta > n^{-\frac{1}{\alpha+1}} \,.
    \end{equation}
    Define $\mathfrak F(M,\delta)$ to be the collection of $(M',\delta')$ such that
    \begin{align}
        &(M')^{-\frac{1}{2}}(\log n)^3 \leq \delta', \quad \left( 1-(\log n)^{-1} \right) M \leq M' \,; \label{eq-def-F_2-on-general}  \\
        &\min\left\{ \left( 1-(\log n)^{-3} \right) M(1-\delta), (\log n)^{20} \right\} \leq M'(1-\delta') \,; \label{eq-def-F_2-on-M}  \\
        & \left( 1-\frac{\delta}{\sqrt{\log n}} \right) (1-\delta)^\alpha \leq (1-\delta') \leq \left( 1+\frac{\delta}{\sqrt{\log n}} \right) (1-\delta)^{\alpha} \,.  \label{eq-def-F_2-on-delta}
    \end{align}
    Given any $(\mathbf X_1,\ldots,\mathbf X_n) \in \overline{\Omega}_n(M,\delta)$, we have 
    \begin{equation*}
        \Pb^\alpha_{n,\mathbf X}\left( (\mathbf X_1',\ldots,\mathbf X_n') \in \cup_{ (M',\delta') \in \mathfrak F(M,\delta) } \overline{\Omega}_n(M',\delta'), \mathcal L(\mathbf X')=\mathcal L(\mathbf X) \right) \geq 1-\frac{1}{\log n} \,.
    \end{equation*}
\end{lemma}
\begin{proof}
Without loss of generality, we may assume that
\begin{align*}
    \mathbf X_1=M<n, \quad \mathbf X_2 = \max_{2 \leq i \leq n} \mathbf X_i= (1-\delta)M \,.
\end{align*}
We first give some preliminary estimates. Note that we have
\begin{align}\label{eq-standard-upper-bound-Z}
    \mathbf Z_{\mathbf X} \leq \mathbf X_1^\alpha + \left( \sum_{i=2}^n \mathbf X_i \right) \mathbf X_2^{\alpha-1} \leq M^{\alpha}+(n-M)(1-\delta)^{\alpha-1}M^{\alpha-1} \,. 
\end{align}
Therefore, letting $p=\frac{M^\alpha}{\mathbf Z_{\mathbf X}}$, we have
\begin{align}
    np=\frac{nM^\alpha}{\mathbf Z_{\mathbf X}} \geq \frac{nM^\alpha}{M^\alpha+(n-M)(1-\delta)^{\alpha-1} M^{\alpha-1}} \geq \frac{nM^\alpha}{M^\alpha+(n-M)M^{\alpha-1}} = M \,.   \label{eq-np-lower-bound}
\end{align}
Moreover, we have
\begin{align}
    np(1-\delta)^\alpha &\geq \frac{nM^\alpha(1-\delta)^\alpha}{M^\alpha + (n-M)(1-\delta)^{\alpha-1}M^{\alpha-1}} = \frac{n(1-\delta)M}{(1-\delta)^{1-\alpha}M + n - M} \nonumber\\
    &\geq \frac{\frac{1}{2} nM \mathbf 1_{\{\delta \leq \frac{1}{2}\}}}{(2^{\alpha - 1}-1)M + n}  \geq \frac{M}{2^\alpha} \mathbf 1_{\{\delta \leq \frac{1}{2}\}}\,,  \label{eq-np-lower-bound-crude}
\end{align}
In addition, we have
\begin{align}
    np(1-\delta)^\alpha &\geq \frac{nM^\alpha(1-\delta)^\alpha}{M^\alpha + (n-M)(1-\delta)^{\alpha-1}M^{\alpha-1}} = \frac{n(1-\delta)M}{(1-\delta)^{1-\alpha}M + n - M} \nonumber\\
    &\geq \min\left\{ \left( 1-n^{-\frac{1}{2(\alpha + 1)}} \right)(1-\delta)M, n^{\frac{1}{3(\alpha + 1)}} \right\} \,,  \label{eq-np-lower-bound-crude-2}
\end{align}
where the last inequality follows from
\begin{align*}
    \frac{n}{(1-\delta)^{1-\alpha}M+n-M} & \geq \frac{1}{1+\frac{(1-\delta)M}{n} \cdot (1-\delta)^{-\alpha}} \overset{\eqref{eq-assum-lem-2.5}}{\geq} 1-\frac{(1-\delta)M}{n} \cdot n^{\frac{\alpha}{\alpha+1}} \geq 1-n^{-\frac{1}{2(\alpha + 1)}}
\end{align*}
when $(1-\delta)M \leq n^{\frac{1}{2(\alpha + 1)}}$, and
\begin{align*}
    \frac{n(1-\delta)M}{(1-\delta)^{1-\alpha}M+n-M} \geq \frac{n(1-\delta)M}{n^{\frac{\alpha}{\alpha+1}}(1-\delta)M+n} > n^{\frac{1}{3(\alpha + 1)}}
\end{align*}
when $(1-\delta)M > n^{\frac{1}{2(\alpha + 1)}}$. 

Now we return to the proof of Lemma~\ref{lem-prelim-gap-est-time-after-2R}. Define
\begin{align}\label{def-H-M,n,alpha}
H(M, n, \alpha) := \min\left\{ \left( 1-n^{-\frac{1}{2(\alpha + 1)}} \right)(1-\delta)M, n^{\frac{1}{3(\alpha + 1)}} \right\} \vee \frac{M}{2^\alpha} \mathbf 1_{\{\delta \leq \frac{1}{2}\}}\,,
\end{align}
then, by \eqref{eq-np-lower-bound-crude} and \eqref{eq-np-lower-bound-crude-2} we have
\begin{align}
np(1-\delta)^\alpha \geq H(M, n, \alpha)\,. \label{eq-np-lower-bound-delicate}
\end{align}
Also, define the bad events 
\begin{align*}
    \mathcal B_1 &= \left\{ \mathbf X'_1 \geq \left( 1+\frac{2(\log n)^2}{\sqrt{M}} \right) np \right\} \cup \left\{ \max_{2\leq i \leq n} \mathbf X'_i \leq \left( 1-\frac{3(\log n)^2}{\sqrt{H(M,n,\alpha)}} \right) np(1-\delta)^{\alpha} \right\} \,, \\
    \mathcal B_2 &= \left\{ \mathbf X'_1 \leq \left( 1-\frac{2(\log n)^2}{\sqrt{M}} \right)np \right\}  \cup \left\{ \max_{2\leq i \leq n} \mathbf X'_i \geq \left( 1+\frac{3(\log n)^2}{\sqrt{H(M,n,\alpha)}} \right) np(1-\delta)^{\alpha} \right\}\,, \\
    \mathcal B_3 &= \left\{ \max_{2\leq i \leq n} \mathbf X'_i \leq \min\left\{ \left( 1-\frac{1}{(\log n)^{3}} \right)M(1-\delta), (\log n)^{20} \right\} \right\}\,,
\end{align*}
We first argue that on $(\mathcal B_1 \cup \mathcal B_2 \cup \mathcal B_3)^c$ we have $\mathbf X' \in \overline{\Omega}_n(M',\delta')$ for some $(M',\delta') \in \mathfrak F(M,\delta)$. In fact, on $(\mathcal B_1 \cup \mathcal B_2)^c$ we have
\begin{align*}
    \frac{ \max_{2\leq i \leq n} \mathbf X'_i }{ \mathbf X_1' } \geq \frac{ (1-\frac{3(\log n)^2}{\sqrt{H(M,n,\alpha)}})np(1-\delta)^{\alpha} }{ (1+\frac{2(\log n)^2}{\sqrt{M}})np } \geq \left( 1-\frac{10(\log n)^2}{\sqrt{H(M,n,\alpha)}} \right) (1-\delta)^{\alpha}
\end{align*}
by the fact that $(\log n)^{11} \leq H(M, n, \alpha) \leq M$, and similarly
\begin{align*}
    \frac{ \max_{2\leq i \leq n} \mathbf X'_i }{ \mathbf X_1' } &\leq \frac{ (1+\frac{3(\log n)^2}{\sqrt{H(M,n,\alpha)}})np(1-\delta)^{\alpha} }{ (1-\frac{2(\log n)^2}{\sqrt{M}})np } \leq \left( 1+\frac{10(\log n)^2}{\sqrt{H(M,n,\alpha)}} \right) (1-\delta)^{\alpha} \,.
\end{align*}
This also implies that $\mathbf X_1'=\max_{1 \leq i \leq n}\mathbf X_i'$. Thus, $(\mathbf X_1',\ldots,\mathbf X_n') \in \overline{\Omega}_n (M',\delta')$ for a unique $(M',\delta')$. Combined with $\mathcal B_3^c$, we see that $(M',\delta')$ satisfies \eqref{eq-def-F_2-on-M}. Also, from $\mathcal B_1^c$ and $\mathcal B_2^c$ we have
\begin{align*}
    M' \geq \left( 1-(\log n)^{-2} \right) np \geq \left( 1-(\log n)^{-2} \right) M \,.
\end{align*}
Next, when $\delta > \frac{1}{2}$, we have $\frac{10(\log n)^2}{\sqrt{H(M,n,\alpha)}}  \leq \frac{20(\log n)^2}{\sqrt{(\log n)^{11}}} < \frac{\delta}{\sqrt{\log n}}$, which gives
\begin{align*}
    \left|\frac{1-\delta'}{(1-\delta)^\alpha} - 1\right| \leq \frac{\delta}{\sqrt{\log n}}, \quad \delta' \geq 1-\left( 1+\frac{\delta}{\sqrt{\log n}} \right) (\frac{1}{2})^{\alpha} \geq 1-2^{-\frac{\alpha}{2}} \,,
\end{align*}
and when $\delta \leq \frac{1}{2}$, by $\delta \geq \frac{(\log n)^3}{\sqrt{M}}$ we have $\frac{10(\log n)^2}{\sqrt{H(M,n,\alpha)}}  \leq \frac{10(\log n)^2}{\sqrt{M\cdot 2^{-\alpha}}} < \frac{\delta}{\sqrt{\log n}}$. Also, $(1-\delta)^\alpha < 1 - C_\alpha \delta$ for some constant $C_\alpha > 1$ in this case, which gives
\begin{align*}
    \left|\frac{1-\delta'}{(1-\delta)^\alpha} - 1\right| \leq \frac{\delta}{\sqrt{\log n}}, \quad \delta' \geq 1-\left( 1+\frac{\delta}{\sqrt{\log n}} \right) (1-C_\alpha\delta) &\geq \frac{C_\alpha + 1}{2}\delta \,.
\end{align*}
Hence, $(M',\delta')$ satisfies \eqref{eq-def-F_2-on-delta}, and moreover
\begin{align*}
    M'(\delta')^2 &\geq \min\{ (1-(\log n)^{-2})\cdot (\frac{C_\alpha + 1}{2})^2 \cdot M\delta^2 \,, M(1-2^{-\frac{\alpha}{2}})^2 \} \\
    &\geq \min\{M\delta^2 \,, (\log n)^{99}\} \geq (\log n)^6\,,
\end{align*}
and thus $(M',\delta')$ satisfies \eqref{eq-def-F_2-on-general}. Combining the above argument, we have $\mathbf X' \in \overline{\Omega}(M',\delta')$ for some $(M',\delta') \in \mathfrak F(M,\delta)$.

To this end, it suffices to prove that
\begin{equation}
    \Pb\left( \mathcal B_1 \cup \mathcal B_2 \cup \mathcal B_3 \right) \leq \frac{1}{\log n} \,. \label{eq-lem-2.4-goal}
\end{equation}
Note that 
\begin{align}
&\left( 1-n^{-\frac{1}{2(\alpha + 1)}} \right)(1-\delta)M - (\log n)^2 \sqrt{\left( 1-n^{-\frac{1}{2(\alpha + 1)}} \right)(1-\delta)M}\label{eq-np-lb-de-1}\\
\geq \ & \left( 1-n^{-\frac{1}{2(\alpha + 1)}} \right)(1-\delta)M - \frac{1}{2}(\log n)^{-3}(1-\delta) M > (1-\frac{1}{(\log n)^3})(1-\delta)M\,,\nonumber
\end{align}
and
\begin{align}
n^{\frac{1}{3(\alpha + 1)}} - (\log n)^2 \sqrt{n^{\frac{1}{3(\alpha + 1)}}}  > (\log n)^{20}\,,\label{eq-np-lb-de-2}
\end{align}
Therefore, if we define $f(x) = x - (\log n)^2\sqrt{x}$ which is increasing on $(\frac{(\log n)^{4}}{4} \,, +\infty) \supset \left(\eqref{eq-np-lower-bound-crude-2}\,, +\infty\right)$, we have 
\begin{align*}
    f(\eqref{eq-np-lower-bound-delicate}) &\geq f(\eqref{eq-np-lower-bound-crude-2}) = \min\{\eqref{eq-np-lb-de-1}\,, \eqref{eq-np-lb-de-2}\} \geq \min\{ (1-\tfrac{1}{(\log n)^3})(1-\delta)M\,, (\log n)^{20} \}\,.
\end{align*}
Hence, we have
\begin{align*}
    \Pb\left( \mathcal B_3 \right) &\leq \Pb\left( \max_{2 \leq i \leq n} \{ \mathbf X_i' \} < \eqref{eq-np-lower-bound-delicate} - (\log n)^2 \sqrt{\eqref{eq-np-lower-bound-delicate}} \right) \leq 2n\exp(-(\log n)^{1.9}) \,.
\end{align*}
In addition, we have
\begin{align}
    & \Pb(\mathcal B_1\cup \mathcal B_2)
    \leq \Pb\left( |\mathsf{Bin}(n,p)-np| \geq (\log n)^2 \sqrt{np} \right) \nonumber\\
    &+(n-1)\Pb\left( |\mathsf{Bin}(n,p(1-\delta)^\alpha)-np(1-\delta)^\alpha| \geq (\log n)^{2} \sqrt{np(1-\delta)^{\alpha}} \right) \leq 2n\exp\big(-(\log n)^{2}\big)\,, \nonumber
\end{align}
leading to \eqref{eq-lem-2.4-goal} from a direct union bound.
\end{proof}

Then, we use the following lemma to bring $1-\delta$ from the ``relatively small'' phase to a ``very small'' phase:

\begin{lemma}{\label{lem-prelim-gap-est-final-constant-rounds}}
    Suppose that $\alpha > 1$ and
    \begin{equation}{\label{eq-assum-lem-2.6}}
        M \geq (\log n)^{99}, \quad 1-\delta \leq n^{-\frac{1}{100(\alpha+1)}} \,.
    \end{equation}
    Define $\mathfrak F(M,\delta)$ to be the collection of $(M',\delta')$ such that 
    \begin{equation}{\label{eq-def-F_3-on-M-delta}}
        (1-(\log n)^{-2})M \leq M', \quad (1-\delta') \leq \max\left\{ \frac{\log n}{n}, (\log n)^2(1-\delta)^{\alpha} \right\}  \,.
    \end{equation}
    Given any $(\mathbf X_1,\ldots,\mathbf X_n) \in \overline{\Omega}_n(M,\delta)$, we have 
    \begin{equation*}
        \Pb^\alpha_{n,\mathbf X}\left( (\mathbf X_1',\ldots,\mathbf X_n') \in \cup_{ (M',\delta') \in \mathfrak F(M,\delta) } \overline{\Omega}_n(M',\delta'), \mathcal L(\mathbf X')=\mathcal L(\mathbf X) \right) \geq 1-\frac{1}{\log n} \,.
    \end{equation*}
\end{lemma}
\begin{proof}
Without loss of generality, we may assume that
\begin{align*}
    \mathbf X_1=M<n, \quad \mathbf X_2 = \max_{2 \leq i \leq n} \mathbf X_i= (1-\delta)M \,.
\end{align*}
We first show that $\Pb(\mathbf X_1'\geq M)\geq 1-\frac{1}{\log n}$. When $(\log n)^{99}\leq M \leq \frac{2n}{3}$, using \eqref{eq-standard-upper-bound-Z} we have (recall that $(1-\delta)^{\alpha-1}=o(1)$)
\begin{align*}
    \mathbf Z_{\mathbf X} \leq M^{\alpha}+(n-M)(1-\delta)^{\alpha-1} M^{\alpha-1} \leq \frac{3n}{4} \cdot M^{\alpha-1} \,.
\end{align*}
Therefore, we have
\begin{align*}
    &\Pb(\mathbf X'_1<M) \leq \Pb\left( \mathsf{Bin}\left( n,\tfrac{4M}{3n} \right)<M \right) \\
    \leq\ & \Pb\left( \mathsf{Bin}\left( n,\tfrac{4M}{3n} \right)< n\cdot \tfrac{4M}{3n} - (\log n) \sqrt{n\cdot\tfrac{4M}{3n}\left( 1-\tfrac{4M}{3n} \right)} \right) \leq \exp\Big( -\frac{(\log n)^2}{4+\log n} \Big) \leq \frac{1}{\sqrt n}\,,
\end{align*}
where the second inequality follows from
\begin{align*}
    n \cdot \frac{4M}{3n} - M = \frac{M}{3} \overset{\eqref{eq-assum-lem-2.6}}{\geq} (\log n) \sqrt{n \cdot \frac{4M}{3n}\left( 1-\frac{4M}{3n} \right)} \,.
\end{align*}
When $M \geq\frac{2n}{3}$, we have
\begin{align*}
    \mathbf Z(\mathbf X_1,\ldots,\mathbf X_n)= M^{\alpha}+ \sum_{i=2}^{n}\mathbf X_i^\alpha \leq M^{\alpha} + (n-M)^{\alpha} \,.
\end{align*}
Therefore, defining $p=\frac{M^{\alpha}}{M^{\alpha}+(n-M)^{\alpha}} \geq \frac{1}{2}$, we have
\begin{align*}
    np-M &= \frac{(n-M)M(1-(\frac{n-M}{M})^{\alpha-1})}{M+(\frac{n-M}{M})^{\alpha-1}(n-M)} \geq \left( 1-\frac{1}{2^{\alpha-1}} \right) \cdot \frac{(n-M)M}{M+ (\frac{n-M}{M})^{\alpha-1}(n-M)}\\
    &\geq \left( 1-\frac{1}{2^{\alpha-1}} \right) (\log n)^2 \sqrt{ n\cdot\frac{M^\alpha }{ M^\alpha+(n-M)^\alpha}\left( 1-\frac{M^\alpha}{M^\alpha+(n-M)^\alpha} \right) } \\
    &= \left( 1-\frac{1}{2^{\alpha-1}} \right) (\log n)^2 \sqrt{np(1-p)} \,,
\end{align*}
where the second inequality follows from 
\begin{align*}
    (n-M)^2 M^{2\alpha} \geq n(\log n)^4 M^\alpha (n-M)^{\alpha} \mbox{ when } \frac{2n}{3} \leq M \leq n-1 \mbox{ and } \alpha>1 \,.
\end{align*}
When $p \leq \frac{n-1}{n}$, we have 
\begin{align}\label{eq-lower-bound-above-avg-3}
    \Pb(\mathbf X_1'<M) &\leq \Pb\left(\mathsf{Bin}(n,p)<M \right) \leq \Pb\left( \mathsf{Bin}(n,p) < np - \left( 1-\tfrac{1}{2^{\alpha-1}} \right)(\log n)^2 \sqrt{np(1-p)} \right) \nonumber\\
    &\leq \exp\Big( -\left(1-\tfrac{1}{2^{\alpha-1}} \right)^2 \tfrac{(\log n)^4}{4+(\log n)^2} \Big)\,,
\end{align}
where the last inequality holds by binomial Chernoff bound. When $p>\tfrac{n-1}{n}$ (which implies that $np(1-p)\leq 1$) and $n-M>(\log n)^2$, again by binomial Chernoff bound,
\begin{align}\label{eq-lower-bound-above-avg-4}
    \Pb(\mathbf X'_1 < M) \leq \Pb\left( \mathsf{Bin}(n,1-p)>(\log n)^2 \right) \leq \exp(-(\log n)^{2})\,.
\end{align}
When $n-M \leq (\log n)^2$, we have $M \geq n - (\log n)^2$ and therefore
\begin{align}\label{eq-lower-bound-above-avg-5}
    \Pb(\mathbf X'_1<M) &\leq \Pb\left( \mathsf{Bin}\left( n,\frac{M^\alpha}{M^\alpha+(n-M)^\alpha} \right) < n \right)\nonumber\\
    &= \Pb\left( \mathsf{Bin}\left( n,\frac{(n-M)^\alpha}{M^\alpha + (n-M)^\alpha} \right) > 0 \right) < n^{\frac{1-\alpha}{2}} \,.
\end{align}
Combining \eqref{eq-lower-bound-above-avg-3}, \eqref{eq-lower-bound-above-avg-4} and \eqref{eq-lower-bound-above-avg-5} we have $\Pb(\mathbf X'_1 < M) \leq n^{\frac{1-\alpha}{2}}$.

Next, we consider the second item of \eqref{eq-def-F_3-on-M-delta}. Define $p_M = \frac{M^\alpha}{\mathbf Z_{\mathbf X}}$ and $q_M = \frac{(1-\delta)^\alpha M^\alpha}{\mathbf Z_{\mathbf X}}$. 
Applying the binomial Chernoff bound, we have
\begin{align*}
    \Pb\left( |\mathbf X_1'-np_M| \geq 0.1np_M \right) \leq \exp\Big(-\frac{0.01np_M}{4(1-p_M)+0.1}\Big) \leq \exp(-0.001np_M)\,.
\end{align*}
Combined with \eqref{eq-np-lower-bound} and the fact that $M \geq (\log n)^{100}$, we have
\begin{align}
    \Pb\left( |\mathbf X_1'-np_M| \geq 0.1np_M \right) \leq \exp(-(\log n)^{99}) \,.   \label{eq-ratio-pM-qM-bound-1}
\end{align}
Next, if $(1-\delta)^\alpha>\frac{1}{n}$ we have 
\begin{align*}
    nq_M \geq \frac{n(1-\delta)^\alpha M^\alpha}{M^{\alpha} + (n-M)(1-\delta)^{\alpha-1}M^{\alpha-1}} \geq \frac{n}{(1-\delta)^{-\alpha}+n-1} \geq \frac{1}{2} \,.
\end{align*}
Thus, applying the binomial Chernoff bound, we have
\begin{align}
    \Pb\left( \max_{2 \leq i \leq n} \{ \mathbf X_i' \} > (\log n)^{1.5} nq_M \right) &\leq n\Pb\left( \mathsf{Bin}(n,q_M)>(\log n)^{1.5}nq_M \right) \nonumber \\
    &\leq n\exp\Big(-\frac{(\log n)^3 nq_M}{4+(\log n)^{1.5}}\Big) \leq \exp(-(\log n)^{0.4})  \,.    \label{eq-ratio-pM-qM-bound-2}
\end{align}
Combining \eqref{eq-ratio-pM-qM-bound-1} and \eqref{eq-ratio-pM-qM-bound-2}, we have in this case
\begin{align}\label{eq-ratio-pM-qM-bound-2-Gaussian}
    &\Pb\left( \max_{2 \leq i \leq n} \{ \mathbf X_i' \} > (\log n)^2 (1-\delta)^\alpha \mathbf X_1' \right)\nonumber\\
    \leq\ &\Pb\left( \mathbf X_1' \leq \tfrac{9}{10}np_M \right) + \Pb\left( \max_{2 \leq i \leq n} \{ \mathbf X_i' \} > (\log n)^{1.5} nq_M \right) \nonumber \\
    \leq\ &\exp(-(\log n)^{99}) + \exp(-(\log n)^{0.4}) \,.
\end{align}
For $(1-\delta)^\alpha \leq\frac{1}{n}$ we have
\begin{align*}
     q_M \leq (1-\delta)^{\alpha}p_M \leq \frac{1}{n}, \quad \Eb[\mathbf X_1'] = np_M > \frac{n}{2} \,.
\end{align*}
Therefore, we have
\begin{align}\label{eq-ratio-pM-qM-bound-Poisson}
    &\Pb\left( \max_{2 \leq i \leq n} \{ \mathbf X_i' \} > \frac{\log n}{n} \mathbf X'_1 \right) \leq \Pb\left( \max_{2 \leq i \leq n} \mathbf X_i' > \frac{9\log n}{20}  \right) +\Pb\left(\mathbf X_1' < \frac{9n}{20} \right)\nonumber \\
    \leq\ & n\Pb\left( \mathsf{Bin}(n,\tfrac{1}{n})>\frac{9\log n}{20} \right) + \Pb\left( \mathsf{Bin}(n,\tfrac{1}{2})<\frac{9n}{20} \right) \leq \frac{n}{(9\log n/20)!} + \exp(-(\log n)^{99}) \,.
\end{align}
Combining \eqref{eq-ratio-pM-qM-bound-2-Gaussian} and \eqref{eq-ratio-pM-qM-bound-Poisson}, the second item of \eqref{eq-def-F_3-on-M-delta} is satisfied with probability $1-\frac{1}{2\log n}$.
\end{proof}

Next, we use the following lemma to show that the voting process ends within a bounded number of additional rounds once $1-\delta$ enters the ``very small'' phase:
\begin{lemma}{\label{lem-prelim-gap-est-very-final-rounds}}
    Suppose that $\alpha > 1$, $M \in \mathbb Z_{+}$ and $\delta \in (0, 1]$ such that
    \begin{equation}{\label{eq-assum-lem-2.7}}
         1-\delta \leq \frac{\log n}{n} \,.
    \end{equation}
    Define $N_\alpha = \lceil\log_\alpha (\frac{2}{\alpha - 1})\rceil + 1$. Then, given that $(\mathbf X^{(1)}_k,\ldots,\mathbf X^{(n)}_k) \in \overline{\Omega}_n(M,\delta)$ we have 
    \begin{align*}
        \Pb_n^{\alpha}\left( \mathcal T \leq k+N_\alpha + 1, \mathfrak L(\mathbf X_k) = \{\mathcal W\} \right) \geq 1-\frac{1}{\log n} \,.
    \end{align*} 
\end{lemma}
\begin{proof}
Without loss of generality, we may assume that
\begin{align*}
    \mathbf X_k^{(1)}=M<n, \quad \mathbf X_k^{(2)} = \max_{2 \leq i \leq n}\mathbf X_k^{(i)}= (1-\delta)M \,.
\end{align*}
Since $(1-\delta)M \geq 1$, we have $M \geq(1-\delta)^{-1}\geq \frac{n}{\log n}$. Also, for $K \in [0,\frac{n}{3}]$ we define
\begin{align*}
\mathcal S_{n,\alpha}(K) = \{(x_1 \,, \ldots \,, x_n) \in \Omega_n : \sum_{2\leq i\leq n}x_i =n-x_1\leq K  \,, \max_{2 \leq i \leq n} x_i \leq \frac{2\alpha}{\alpha - 1}\}\,,
\end{align*}
then by definition of $\mathcal S_{n,\alpha}(\cdot)$ we have $\mathcal S_{n,\alpha}(K) \subset \bigcup_{M \in \mathbb Z_+ , \delta \in (0,1-\frac{\log n}{n}]} \overline{\Omega}_n(M,\delta)$ for each $K \in [0,\frac{n}{3}]$. Next, note that for $i \geq 2$ we have
\begin{align*}
    \frac{ (\mathbf X_k^{(i)})^{\alpha} }{ \mathbf Z_{\mathbf X_k} } \leq \frac{ (\mathbf X_k^{(i)})^{\alpha} }{ M^{\alpha} } \leq (1-\delta)^{\alpha} < \left( \frac{\log n}{n} \right)^{\alpha} \,.
\end{align*}
Thus, 
\begin{align*}
    \Pb_n^\alpha\left( \sum_{2 \leq i \leq n}\mathbf X_{k+1}^{(i)} > n^{2-\alpha}(\log n)^{2\alpha}\,;  \mathbf X_k\in \overline{\Omega}_n(M,\delta)\right) \leq \frac{\sum_{2\leq i \leq n} \Eb[\mathbf X_{k+1}^{(i)}]}{n^{2-\alpha}(\log n)^{2\alpha}} 
    < (\log n)^{-\alpha} \,,
\end{align*}
and
\begin{align}
    &\Pb_n^\alpha \left( \max_{2 \leq i \leq n}\mathbf X_{k+1}^{(i)} \geq \frac{2\alpha}{\alpha -1}\,;  \mathbf X_k\in \overline{\Omega}_n(M,\delta) \right) \nonumber \\
    \leq\ & \sum_{2 \leq i \leq n}\Pb\left(\mathsf{Bin}\left(n, \left(\frac{\log n}{n}\right)^\alpha\right) \geq \frac{2\alpha}{\alpha -1}\right) < n \cdot \binom{n}{\lceil\frac{2\alpha}{\alpha - 1}\rceil} \left(\frac{\log n}{n}\right)^{\alpha \cdot \lceil\frac{2\alpha}{\alpha - 1}\rceil} < n^{1-2\alpha +o(1)}\,.\label{eq-veryfinal-ub-individual}
\end{align}
Hence, we have
\begin{align}
\Pb_n^\alpha \left( \mathbf X_{k+1} \in \mathcal S_{n,\alpha}(n^{2-\alpha}(\log n)^{2\alpha}) \,; \mathbf X_k\in \overline{\Omega}_n(M,\delta) \right) = 1-O(\frac{1}{(\log n)^{\alpha}})\,.\label{eq-very-final-initial-step}
\end{align}
Next, we prove that for any $i \in \mathbb Z_+$ we have
\begin{align}
\Pb_n^\alpha \left( \mathbf X_{k+i+1} \in \mathcal S_{n,\alpha}(n^{1+\alpha^{i}-\alpha^{i+1}}(\log n)^{(2\alpha)^{i+1}}) \,; \mathbf X_{k+i}\in \mathcal S_{n,\alpha}(n^{1+\alpha^{i-1}-\alpha^i}(\log n)^{(2\alpha)^i}) \right) \nonumber\\
= 1-O(\frac{1}{(\log n)^{\alpha}})\,. \label{eq-very-final-induction}
\end{align}
Given that $\mathbf X_{k+i}\in \mathcal S_{n,\alpha}(n^{1+\alpha^{i-1}-\alpha^i}(\log n)^{(2\alpha)^i})$, we have $\mathbf X_{k+i} \in \overline{\Omega}_n(M,\delta)$ for some $M \in \mathbb Z_+$ and $\delta \in (0,1-\frac{\log n}{n}]$. Also, we have $\mathbf X_{k+i}^{(1)} = \max_{j \in [n]} \mathbf X^{(j)}_{k+i}$ by definition of $\mathcal S(\cdot)$. Therefore, similar to \eqref{eq-veryfinal-ub-individual} we have
\begin{align*}
\Pb_n^\alpha \left( \max_{2 \leq j \leq n}\mathbf X^{(j)}_{k+i+1} \geq \frac{2\alpha}{\alpha - 1}\,; \mathbf X_{k+i}\in \mathcal S_{n,\alpha}(n^{1+\alpha^{i-1}-\alpha^i}(\log n)^{(2\alpha)^i}) \right) < n^{1-2\alpha + o(1)}\,.
\end{align*}
Next, note that
\begin{align*}
&\sum_{2\leq j \leq n} \Eb[\mathbf X_{k+i+1}^{(j)}] = \frac{n\sum_{2\leq j \leq n}(\mathbf X_{k+i}^{(j)})^\alpha}{(\mathbf X_{k+i}^{(1)})^\alpha + \sum_{2\leq j \leq n}(\mathbf X_{k+i}^{(j)})^\alpha}\\
\leq\ & \frac{n(\sum_{2\leq j \leq n}\mathbf X_{k+i}^{(j)})^\alpha}{(\mathbf X_{k+i}^{(1)})^\alpha + (\sum_{2\leq j \leq n}\mathbf X_{k+i}^{(j)})^\alpha} \leq \frac{n(n^{1+\alpha^{i-1}-\alpha^i}(\log n)^{(2\alpha)^i})^\alpha}{(\mathbf X_{k+i}^{(1)})^\alpha+(n^{1+\alpha^{i-1}-\alpha^i}(\log n)^{(2\alpha)^i})^\alpha}\\
&\leq \frac{n(n^{1+\alpha^{i-1}-\alpha^i}(\log n)^{(2\alpha)^i})^\alpha}{\frac{2}{3}n^\alpha+(n^{1+\alpha^{i-1}-\alpha^i}(\log n)^{(2\alpha)^i})^\alpha} \leq 2n^{1+\alpha^i-\alpha^{i+1}} \cdot (\log n)^{(2\alpha)^i}\,,
\end{align*}
which gives
\begin{align*}
    &\Pb_n^\alpha\left( \sum_{2 \leq j \leq n}\mathbf X_{k+i+1}^{(j)} > n^{1+\alpha^i-\alpha^{i+1}}(\log n)^{(2\alpha)^{i+1}}\,;  \mathbf X_{k+i}\in \mathcal S(n^{1+\alpha^{i-1}-\alpha^{i}}(\log n)^{(2\alpha)^{i}})\right) \\ 
    \leq\ & \frac{\sum_{2\leq j \leq n} \Eb[\mathbf X_{k+i+1}^{(j)}]}{n^{1+\alpha^i-\alpha^{i+1}}(\log n)^{(2\alpha)^{i+1}}} 
    < (\log n)^{-2\alpha} \,.
\end{align*}
Hence, we have proved \eqref{eq-very-final-induction}. Combining \eqref{eq-very-final-induction} with \eqref{eq-very-final-initial-step}, we have shown
\begin{align*}
\Pb_n^\alpha\left(\mathbf X_{k+N_\alpha + 1} \in \mathcal S(n^{-\Omega(1)})\,;  \mathbf X_{k}\in \overline{\Omega}_n(M,\delta) \mbox{ with }\mathbf X_k^{(1)} = \max_{i \in [n]}\mathbf X_k^{(i)} \right) \geq 1-\frac{1}{\log n}\,.
\end{align*}
Therefore, by symmetry we have
\begin{align*}
    \Pb_n^{\alpha}\left( \mathcal T \leq k + N_\alpha +1, \mathfrak L(\mathbf X_k) = \{\mathcal W\} \right) \geq 1-\frac{1}{\log n} \,,
\end{align*} 
which is exactly our desired result.
\end{proof}

Now we are ready to prove Proposition~\ref{lem-repeat-used}.
\begin{proof}[Proof of Proposition~\ref{lem-repeat-used}]
    Without loss of generality, we may assume that $k=0$. Suppose that $(\mathbf X_t^{(1)},\ldots,\mathbf X_t^{(n)}) \in \overline{\Omega}_n(M_t,\delta_t)$ and define
    \begin{align*}
        \mathcal T'&= \inf_{ t \geq 0 } \left\{ 1-\delta_t < n^{-\frac{1}{100(\alpha+1)}} \right\} \wedge \inf_{ t \geq 0 } \left\{ M_t(1-\delta_t) < (\log n)^{12} \right\} \\
        &\quad \wedge \inf_{ t \geq 0 } \left\{ M_t < (\log n)^{100} \right\} \wedge \lfloor\sqrt{\log n}\rfloor \,; \\
        \mathcal T''&= \inf_{ t \geq \mathcal T' } \left\{ 1-\delta_t < \frac{\log n}{n} \right\} \wedge \lfloor\sqrt{\log n}\rfloor \,.
    \end{align*}
    Since we assume that $(\mathbf X_0^{(1)},\ldots,\mathbf X_0^{(n)}) \in \overline{\Omega}_n(M,\delta)$ for some $(M,\delta)$ satisfying \eqref{eq-assum-lem-2.4}, it is clear that $\mathcal T'>0$. Also, it is clear that $\mathcal T' \leq\sqrt{\log n}$ by definition of $\mathcal T'$. In addition, define the event
    \begin{align*}
        \mathcal G=\ & \cap_{t \leq \mathcal T'-1} \left\{ \min\left\{ \left( 1-\frac{1}{(\log n)^{3}} \right) M_t(1-\delta_t), (\log n)^{20} \right\} \leq M_{t+1}(1-\delta_{t+1}) \right\} \\
        &\cap \cap_{t \leq \mathcal T'-1} \left\{ \left( 1- \delta_t(\log n)^{-\frac{1}{2}} \right) (1-\delta_t)^{\alpha} \leq (1-\delta_{t+1}) \right\} \\
        &\cap \cap_{t \leq \mathcal T'-1} \left\{ (1-\delta_{t+1}) \leq \left( 1+ \delta_t(\log n)^{-\frac{1}{2}} \right) (1-\delta_t)^{\alpha} \right\} \\
        &\cap \cap_{t \leq \mathcal T'-1} \left\{\delta_{t+1} \geq M_{t+1}^{-\frac{1}{2}}(\log n)^3, M_{t+1} \geq \left( 1-\frac{1}{\log n} \right)M_t \right\}\\
        &\cap  \cap_{t \leq \mathcal T'-1}\left\{\mathcal L(\mathbf X_{t+1}) = \mathcal L(\mathbf X_t)\right\}\,.
    \end{align*}
    From Lemma~\ref{lem-prelim-gap-est-time-after-2R} and a union bound we then have $\Pb(\mathcal G)=1-o(1)$. In addition, using $M_0=M\geq (\log n)^{101}$ and $M_0(1-\delta_0)=M(1-\delta)\geq (\log n)^{20}$, it is straightforward to check by induction that on the event $\mathcal G$ we have
    \begin{align*}
        & M_{t+1}(1-\delta_{t+1}) \geq (\log n)^{19}, \ M_t \geq (\log n)^{100} \,; \\
        & \left( 1-\frac{\delta_t}{\sqrt{\log n}} \right)(1-\delta_t)^{\alpha} \leq (1-\delta_{t+1}) \leq \left( 1+\frac{\delta_t}{\sqrt{\log n}} \right)(1-\delta_t)^{\alpha} 
    \end{align*}
    for all $t \leq \mathcal T'-1$. This implies that
    \begin{align*}
        M_{\mathcal T'} &\geq (\log n)^{100},\quad  M_{\mathcal T'}(1-\delta_{\mathcal T'}) \geq (\log n)^{18}\,,\\
        (1-\delta_{\mathcal T'})&\in\left[(1-\delta)^{(\alpha-\frac{1}{\sqrt{\log n}})^{\mathcal T'}} \,,(1-\delta)^{(\alpha+\frac{1}{\sqrt{\log n}})^{\mathcal T'}} \right]\,.
    \end{align*}
    Thus, under the event $\mathcal G$ we have
    \begin{align*}
        (\alpha + \frac{1}{\sqrt{\log n}})^{\mathcal T'-1}\log( (1-\delta_0)^{-1}) &> \frac{\log n}{200\alpha(\alpha + 1)}\Longrightarrow \mathcal T' = o(\sqrt{\log n}) \\
         \Longrightarrow(\alpha - \frac{1}{\sqrt{\log n}})^{\mathcal T'}\log( (1-\delta_0)^{-1}) &\leq \frac{\log n}{100(\alpha + 1)} 
        \overset{\eqref{eq-def-kappa}}{\Longrightarrow} \mathcal T'=(1+o(1)) \kappa \,.
    \end{align*}
    Thus, we have (note that $\mathcal T' \leq \mathcal T$ since $1-\delta_{\mathcal T}=0$)
    \begin{align}
        \Pb_{n,M,\delta}^{\alpha}\Big( \mathcal T \geq (1-o(1))\kappa \Big) \geq \Pb_{n,M,\delta}^{\alpha}\Big( \mathcal T' \geq (1-o(1))\kappa \Big)\geq \Pb_{n,M,\delta}^{\alpha}(\mathcal G) =1-o(1) \,. \label{eq-lem-2.4-lower-bound}
    \end{align}
    In addition, under the event $\mathcal G$ we have
    \begin{align*}
        1-\delta_{\mathcal T'} \leq n^{-\frac{1}{100(\alpha+1)}}, M_{\mathcal T'} \geq n^{\frac{1}{100(\alpha+1)}} \,. 
    \end{align*}
   Then we define
    \begin{align*}
        \mathcal G'&= \cap_{\mathcal T' \leq t \leq \mathcal T'+\lceil\log_\alpha(100(\alpha + 1))\rceil}\left\{ 1-\delta_{t+1} \leq \max\left\{ \frac{\log n}{n}, (\log n)^2 (1-\delta_t)^{\alpha} \right\}\right\} \\
        &\cap  \cap_{\mathcal T' \leq t \leq \mathcal T'+\lceil\log_\alpha(100(\alpha + 1))\rceil}\left\{\mathcal L(\mathbf X_{t+1}) = \mathcal L(\mathbf X_t)\right\}
    \end{align*}
    Using Lemma~\ref{lem-prelim-gap-est-final-constant-rounds}, we have $\Pb(\mathcal G')=1-o(1)$. Thus, on the event $\mathcal G \cap \mathcal G'$ we have
    \begin{align*}
        (1-\delta_{\mathcal T'+k}) &\leq \max\left\{ (\log n)^{\alpha^k+\ldots+\alpha+1} (1-\delta_{\mathcal T'})^{\alpha^k}, \frac{\log n}{n} \right\} \\
        &\leq \max\left\{ \frac{\log n}{n}, n^{ -\frac{\alpha^k}{10(\alpha+1)} } \right\} \mbox{ for all } k \leq C \,.
    \end{align*}
    By choosing $\frac{\alpha^C}{10(\alpha+1)}>1$ we have on the event $\mathcal G \cap \mathcal G'$
    \begin{align*}
        1-\delta_{\mathcal T'+C} \leq \frac{\log n}{n} \Longrightarrow \mathcal T'' \leq \mathcal T'+C \,.
    \end{align*}
    Finally, define
    \begin{align*}
        \mathcal G''=\left\{ \mathcal T \leq \mathcal T'' + \lceil\log_\alpha (\frac{2}{\alpha - 1})\rceil + 2 \,, \mathcal L(\mathbf X_{\mathcal T''}) = \{\mathcal W\}\right\} \,.
    \end{align*}
    By Lemma~\ref{lem-prelim-gap-est-very-final-rounds} we have $\Pb_{n,M,\delta}^{\alpha}(\mathcal G'') = 1-o(1)$. Thus, we have
    \begin{align}
        \Pb_{n,M,\delta}^{\alpha}\Big( \mathcal T \leq (1+o(1))\kappa \Big) \geq \Pb_{n,M,\delta}^{\alpha}(\mathcal G \cap \mathcal G' \cap \mathcal G'') =1-o(1) \,. \label{eq-lem-2.4-upper-bound}
    \end{align}
    Combining \eqref{eq-lem-2.4-lower-bound} and \eqref{eq-lem-2.4-upper-bound} yields the desired result.
\end{proof}

\section{Proofs of Theorems~\ref{thm-rapid-ordering-supercritical} and \ref{thm-all-or-nothing-transition}}{\label{sec:proof-all-or-nothing}}

\subsection{The case $\alpha\geq 2$}{\label{subsec:case-k=1}}

We first prove the case $\alpha\geq 2$, which needs to be treated separately due to Remark~\ref{rmk-uniqueness-leader}. Define $\aleph=\aleph(n)$ such that 
\begin{equation}{\label{eq-choice-aleph}}
    \aleph! \leq n < (\aleph+1)!
\end{equation}
(note that $\aleph=\Theta(\frac{\log n}{\log\log n})$). The starting point of our proof is the following characterization of the voting dynamics in the first round.
\begin{lemma}\label{lem-first-round-distribution-features}
    Fix $\alpha\geq 1$. Define $A_j = \#\{ 1 \leq i \leq n: \mathbf X^{(i)}_1 = j\}$ for $0 \leq j \leq \aleph + 2$. In addition, define
    \begin{align*}
        \mathtt N = (\log \log n)^2 <\exp(40\alpha \sqrt{\log \log n}), \quad \mathtt N' = \frac{\log n}{(\log \log n)^{30}} \,.
    \end{align*}
    Define $\mathcal E_{\ref{lem-first-round-distribution-features}}=\mathcal E_{\ref{lem-first-round-distribution-features},1} \cap \mathcal E_{\ref{lem-first-round-distribution-features},2} \cap \Ecal_{\ref{lem-first-round-distribution-features},3}$, where 
    \begin{align}
        & \mathcal E_{\ref{lem-first-round-distribution-features},1} = \left\{ |\mathbf M_1-\aleph| \leq 2, n \leq \mathbf Z_1 \leq n\log \log n \right\} \,; \\
        & \mathcal E_{\ref{lem-first-round-distribution-features},2} = \cap_{j \geq 0} \left\{ A_{\mathbf M_1-j} \leq \exp(\log\log n \cdot \max\{ \mathtt N, j \}) \right\} \,; \\
        & \mathcal E_{\ref{lem-first-round-distribution-features},3} = \cap_{10 \leq j \leq \mathtt N'} \left\{ A_{\mathbf M_1-j} \geq \exp( (\log\log n)^{-1} \cdot j ) \right\}\,.
    \end{align}
    We then have: 
    \begin{enumerate}
        \item[(1)] $\Pb^{\alpha}_n(E_{\mathsf{unique}}), \Pb^{\alpha}_n(E^c_{\mathsf{unique}}) \geq \frac{1}{\log n}$.
        \item[(2)] $\Pb^{\alpha}_n(\mathcal E_{\ref{lem-first-round-distribution-features}} \mid E_{\mathsf{unique}}), \Pb^{\alpha}_n(\mathcal E_{\ref{lem-first-round-distribution-features}} \mid E^c_{\mathsf{unique}})=1-o(1)$.
    \end{enumerate}
\end{lemma}
\begin{proof}
    Note that $(\mathbf X^{(1)}_1,\ldots,\mathbf X^{(n)}_1) \sim \mathsf{Mult}(n;\frac{1}{n},\ldots,\frac{1}{n})$. By Lemma~\ref{lem-level-set-concentration}, we have 
    \begin{align*}
        \Pb^{\alpha}_n(E_{\mathsf{unique}}), \Pb^{\alpha}_n(E^c_{\mathsf{unique}}) \geq \frac{1}{30\aleph} \geq \frac{1}{\log n} \,.
    \end{align*}
    In addition, by Lemma~\ref{lem-level-set-concentration}, with probability $1-o(1)$ the event
    \begin{align*}
        \Ecal_{\mathsf{crit}} :=\{ |\mathbf M_1-\aleph| \leq 2 \} &\cap \left\{ A_j\geq \frac{0.4n}{ej!} \mbox{ for all } 0 \leq j \leq \aleph-1 \right\} \\
        &\cap\left\{ A_j\leq \frac{10n}{e((j-4)\vee 0)!} \vee 9\aleph^8 \mbox{ for all } 0 \leq j \leq \aleph + 2 \right\}
    \end{align*}
    holds, both conditioned on $E_{\mathsf{unique}}$ and $E^c_{\mathsf{unique}}$. Therefore, it suffices to show $\Ecal_{\mathsf{crit}} \subset \Ecal_{\ref{lem-first-round-distribution-features}}$. Note that under $\Ecal_{\mathsf{crit}}$ we have $|\mathbf M_1-\aleph|\leq 2$ and (recall $\alpha\geq 1$)
    \begin{align*}
        n & \leq \mathbf Z_1 \leq 100\aleph^{8+\alpha} +n\sum_{0 \leq j \leq \aleph - 5}\tfrac{10j^\alpha}{e((j-4)\vee 0)!} \\
        &\leq 50n + n\Eb_{x \sim \mathsf{Pois}(1)}[(x + 4)^{\alpha}] < n \log \log n \,.
    \end{align*}
    Therefore, we have $\Ecal_{\mathsf{crit}} \subset\Ecal_{\ref{lem-first-round-distribution-features},1}$. Next, under $\Ecal_{\mathsf{crit}}$, for all $0 \leq j \leq \mathbf M_1$ we have
    \begin{align*}
        A_{\mathbf M_1-j} &\leq \max\left\{ \frac{10n}{e((\mathbf M_1-j-4)\vee 0)!}, 9\aleph^8 \right\} \leq \max\left\{ \frac{ 10n(\mathbf M_1+3)^{j+7} }{ e(\mathbf M_1+3)! }, \exp(9\log\log n) \right\}  \\
        &\leq \max\left\{ \left( \frac{2\log n}{\log\log n} \right)^{j+7}, \exp(9\log\log n) \right\} \leq \exp(\log\log n \cdot (\mathtt N \vee j))\,,
    \end{align*}
    which implies that $\Ecal_{\mathsf{crit}} \subset\Ecal_{\ref{lem-first-round-distribution-features},2}$. Finally, under $\Ecal_{\mathsf{crit}}$, for all $10 \leq j \leq \mathtt N'$ we have
    \begin{align*}
        A_{\mathbf M_1 - j} &\geq \frac{0.4n}{e(\mathbf M_1 - j)!} \geq \frac{0.1n}{(\mathbf M_1-2)!}\cdot (\mathbf M_1-j)^{j-2} \\
        &\geq 0.1 \cdot \left( \sqrt{\log n} \right)^{j-2} \geq (\log n)^{j/3} > \exp((\log\log n)^{-1}j)\,,
    \end{align*}
    which implies that $\Ecal_{\mathsf{crit}} \subset\Ecal_{\ref{lem-first-round-distribution-features},3}$. Combining the above arguments, we have $\Ecal_{\mathsf{crit}} \subset\Ecal_{\ref{lem-first-round-distribution-features}}$ and thus we have finished our proof. 
\end{proof}

Based on Lemma~\ref{lem-first-round-distribution-features}, the following lemma characterizes the voting dynamics in the second round.
\begin{lemma}{\label{lem-alpha>2-second-round}}
    When $\alpha\geq 2$, define
    \begin{align*}
        \mathfrak M_{\alpha} = \Big\{ & (M,\delta): \frac{(\log n)^{\alpha}}{(\log\log n)^{2\alpha}} \leq M \leq (\log n)^{\alpha}(\log \log n)^\alpha,  \\
        & \frac{(\log \log n)^{-300-\alpha}}{\log n} \leq \delta \leq \frac{(\log \log n)^{300+\alpha}}{\log n} \Big\} \,.  
    \end{align*}
    Also define
    \begin{align*}
        \mathfrak M_{\alpha}' = \Big\{ & (M,\delta): \frac{(\log n)^{\alpha}}{(\log \log n)^{2\alpha}} \leq M \leq (\log n)^{\alpha}(\log \log n)^\alpha, \\
        &\frac{(\log \log n)^{-300-\alpha}}{(\log n)^{\frac{\alpha}{2}}} \leq \delta \leq \frac{(\log \log n)^{300+\alpha}}{(\log n)^{\frac\alpha 2}}  \Big\} \,.
    \end{align*}
    We have
    \begin{align*}
        &\Pb_n^{\alpha}\left( (\mathbf X^{(1)}_2,\ldots,\mathbf X^{(n)}_2) \in \cup_{(M,\delta) \in \mathfrak M_\alpha} \overline{\Omega}_n(M,\delta) \mid E_{\mathsf{unique}} \right) = 1-o(1) \,; \\
        &\Pb_n^{\alpha}\left( (\mathbf X^{(1)}_2,\ldots,\mathbf X^{(n)}_2) \in \cup_{(M,\delta) \in \mathfrak M_\alpha'} \overline{\Omega}_n(M,\delta) \mid E_{\mathsf{unique}}^c \right) = 1-o(1) \,.
    \end{align*}
    Moreover, for $\alpha > 2$ we have
    \begin{align*}
        &\Pb_n^{\alpha}\Big( \mathcal L(\mathbf X_2) = \mathcal L(\mathbf X_1)  \mid E_{\mathsf{unique}} \Big) = 1-o(1) \,; \\
        &\Pb_n^{\alpha}\Big( \mathcal L(\mathbf X_2) \subsetneq \mathcal L(\mathbf X_1), |\mathcal L(\mathbf X_2)| = 1  \mid E_{\mathsf{unique}}^c \Big) = 1-o(1) \,.
    \end{align*}
\end{lemma}
\begin{proof}
Define $\mathcal A_*(j) = \{ i: \mathbf X_1^{(i)} = \mathbf M_1-j\}$.
First, by Lemma~\ref{lem-first-round-distribution-features}, we can work on the event $\Ecal_{\ref{lem-first-round-distribution-features}}$ in the following proof. Define $\Pb^*(\cdot) = \Pb(\cdot\mid  E_{\mathsf{unique}})$ and $\Pb^{**}(\cdot) = \Pb(\cdot\mid  E^c_{\mathsf{unique}})$. When the probability measure can be either we use the notation $\Pb$. Then, we split into two cases:

\textbf{Case 1:} $\alpha = 2$. In this case, $\frac{n\mathbf M_1^\alpha}{\mathbf Z_1} \in (\frac{(\log n)^2}{(\log\log n)^4}, (\log n)^2)$. Define $\mathtt H = (\log \log n)^{30}$, and define $\mathcal A_{\diamond} = \{ i: \mathbf X_1^{(i)} \geq \mathbf M_1 - \mathtt H\}$. In addition, define ($\overset{(2)}{\max}$ denotes the second maximum)
\begin{align*}
    \mathcal A_{\mathsf{highest}} &= \{(\log \log n)^{-100} \log n \leq\max_{k \in \mathcal A_{\diamond}} \{ \mathbf X_2^{(k)} \} - \max^{(2)}_{k \in \mathcal A_{\diamond}} \{ \mathbf X_2^{(k)} \} \leq (\log \log n)^{200} \log n\}\,.
\end{align*}
Note that by Lemma~\ref{lem-Poisson-coupling}, given any $\mathbf X_1 \in \Omega_n$ there exists a coupling of $\mathbf X_2$ and $(\mathbf Y_2^{(i)} :\sim \mathsf{Pois}(\frac{(\mathbf X_1^{(i)})^2}{\mathbf Z_{\mathbf X_1}}))_{i\in[n]}$ such that $\Pb(\max_{i \in [n]} |\mathbf X_2^{(i)} - \mathbf Y_2^{(i)}| \leq 4) = 1 - o(1)$.
Under $\Ecal_{\ref{lem-first-round-distribution-features}}$ we have $\mathbf M_1 \in [\aleph-2,\aleph+2]$, and
\begin{align}
    \exp(\frac{\mathtt H}{\log\log n}) \leq |\mathcal A_{\diamond}| \leq \exp(2(\log\log n)\mathtt H)\,,\label{eq-size-A-diamond}
\end{align}
which satisfy the requirements of Item (1) and Item (2) of Lemma~\ref{lem-Poisson-maximum-1st-minus-2nd} with $\lambda = \frac{n\mathbf M_1^2}{\mathbf Z_1}, \lambda_* = \frac{n(\mathbf M_1 - \mathtt H)^2}{\mathbf Z_1}$ and $c = 128$, since we have
\begin{align*}
    \lambda_* + \sqrt{\lambda_* (\log \lambda)^{1+c/2}}\geq \frac{n(\mathbf M_1 - \mathtt H)^2}{\mathbf Z_1} + \frac{\sqrt{n}\cdot(\mathbf M_1 - \mathtt H)}{\sqrt{\mathbf Z_1}} \cdot (\log\log n)^{32} > \frac{n\mathbf M_1^2}{\mathbf Z_1}\,.
\end{align*}
Therefore, by Items (1) and (2) of Lemma~\ref{lem-Poisson-maximum-1st-minus-2nd} with $\lambda = \frac{n\mathbf M_1^2}{\mathbf Z_1}$, we have
\begin{align*}
    \Pb\Big( \frac{ (\log\log n)^{-1} \sqrt{n\mathbf M_1^2} }{ \sqrt{\mathbf Z_1 \log |\mathcal A_{\diamond}|} } \leq \max_{k \in \mathcal A_{\diamond}} \{ \mathbf Y_2^{(k)} \} - \max^{(2)}_{k \in \mathcal A_{\diamond}} \{ \mathbf Y_2^{(k)} \} \leq \frac{ (\log\log n)^{129} \sqrt{n\mathbf M_1^2} }{ \sqrt{\mathbf Z_1 \log |\mathcal A_{\diamond}|} } \,; \mathcal E_{\ref{lem-first-round-distribution-features}}\Big) = 1-o(1)\,.
\end{align*}
This implies that (by \eqref{eq-size-A-diamond})
\begin{align}
    \Pb(\mathcal A_{\mathsf{highest}} \cap \mathcal E_{\ref{lem-first-round-distribution-features}}) = 1-o(1)-\Pb(\max_{i \in [n]} |\mathbf X_2^{(i)} - \mathbf Y_2^{(i)}| > 4) =1-o(1)\,. \label{eq-A-diamond-gap-1st-2nd-alpha=2}
\end{align}
Next, note that for $h\geq\mathtt H$ we have
\begin{align*}
    & \frac{n\mathbf M_1^2}{\mathbf Z_1} - \frac{n(\mathbf M_1-h)^{2}}{\mathbf Z_1} \geq \frac{n\mathbf M_1 h}{\mathbf Z_1} \\
    >\ & (\log n)(\log \log n)^{2} + \sqrt{\frac{n\mathbf M_1^2}{\mathbf Z_1}\cdot (\log\log n)^5} + \sqrt{\frac{n(\mathbf M_1-h)^2}{\mathbf Z_1} \cdot h(\log \log n)^5}\,.
\end{align*}
Also, we have 
\begin{align*}
    h(\log \log n)^5 \leq \frac{\mathbf M_1(\log \log n)^5}{2} < \frac{n\mathbf M_1^2}{4\mathbf Z_1} \,.
\end{align*}
Therefore, define
\begin{align*}
    \mathcal A_{\mathsf{order}} = \Big\{ \max_{k \in \mathcal A_*(j)} \{ \mathbf X_2^{(k)} \} \leq \max_{k \in \mathcal A_{\diamond}} \{ \mathbf X_2^{(k)} \} - (\log\log n)^2 \log n \mbox{ for all } j \geq \mathtt H \Big\} \,.
\end{align*}
We then have
\begin{align}
    &\Pb(\mathcal A_{\mathsf{order}}^c\cap \Ecal_{\ref{lem-first-round-distribution-features}}) \nonumber  \\
    \leq\ & \Pb\left( \max_{k \in \mathcal A_*(j)} \{ \mathbf X_{2}^{(k)} \} > \max_{k \in \mathcal A_*(0)} \{ \mathbf X_{2}^{(k)} \} - (\log \log n)^2 \log n \mbox{ for some } \mathtt H \leq j \leq \frac{\mathbf M_1}{2} \right)\nonumber \\
    &+ \Pb\left( \max_{k \in \mathcal A_*(j)} \{ \mathbf X_{2}^{(k)} \} > \max_{k \in \mathcal A_*(0)} \{ \mathbf X_{2}^{(k)} \} - (\log \log n)^2 \log n \mbox{ for some } j >\frac{\mathbf M_1}{2} \right) \nonumber \\
    \leq\ & \Pb\left( \max_{k \in \mathcal A_*(0)} \{ \mathbf X_{2}^{(k)} \} \leq \tfrac{n\mathbf M_1^2}{\mathbf Z_1} - \sqrt{\tfrac{n\mathbf M_1^2}{\mathbf Z_1}\cdot (\log \log n)^5} \right) \nonumber \\
    &+\sum_{\mathtt H \leq j \leq \mathbf M_1/2} \Pb\left( \max_{k \in \mathcal A_*(j)} \{ \mathbf X_{2}^{(k)} \} \geq \tfrac{n(\mathbf M_1 - j)^2}{\mathbf Z_1} + \sqrt{\tfrac{n(\mathbf M_1-j)^2}{\mathbf Z_1}\cdot j(\log \log n)^5} \right) \nonumber \\
    &+ n\cdot \Pb\left( \mathsf{Bin}\left( n,\tfrac{\mathbf M_1^2}{4\mathbf Z_1} \right) \geq \tfrac{n\mathbf M_1^2}{2\mathbf Z_1} \right) \nonumber \\
    \leq\ & 2\exp(-\tfrac{(\log \log n)^5}{5}) + \sum_{\mathtt H \leq j \leq \mathbf M_1/2} 2\exp(-\tfrac{j(\log \log n)^5}{10}) + n\exp(-(\log n)^{2-o(1)}) = o(1)\,. \label{eq-bound-away-from-highests-alpha=2}
\end{align}
Combining \eqref{eq-A-diamond-gap-1st-2nd-alpha=2} and \eqref{eq-bound-away-from-highests-alpha=2}, we have
\begin{align*}
    &\Pb\left( (\log \log n)^{-100}(\log n) \leq \max_{k \in [n]} \{ \mathbf X_2^{(k)} \} - \max_{k \in [n]}^{(2)} \{ \mathbf X_2^{(k)} \} \leq (\log \log n)^{200}(\log n) \right) \\
    \geq\ & \Pb(\mathcal A_{\mathsf{highest}} \cap \mathcal A_{\mathsf{order}} \cap \Ecal_{\ref{lem-first-round-distribution-features}}) = 1-o(1) \,.
\end{align*}
Also, we have 
\begin{align}
    &\Pb\Big( \mathbf M_2 \not\in \left[ \tfrac{n\mathbf M_1^2}{2\mathbf Z_1}, \tfrac{3n\mathbf M_1^2}{2\mathbf Z_1} \right] \Big) \leq \Pb(\Ecal_{\ref{lem-first-round-distribution-features}}^c) + \Pb\Big( \mathsf{Bin}\Big(n, \tfrac{\mathbf M_1^2}{\mathbf Z_1} \Big) < \tfrac{\mathbf M_1^2}{2\mathbf Z_1} \Big) + n\Pb\Big( \mathsf{Bin}\Big( n,\tfrac{\mathbf M_1^2}{\mathbf Z_1} \Big) > \tfrac{3\mathbf M_1^2}{2\mathbf Z_1} \Big) \nonumber \\
    &\leq o(1)+\exp(-(\log n)^{2-o(1)})+n\exp(-(\log n)^{2-o(1)})=o(1) \,, \label{control-second-step-case-1}
\end{align}
which gives our desired result.

\noindent\textbf{Case 2: $\alpha>2$.} In this case, $\frac{n\mathbf M_1^\alpha}{\mathbf Z_1} \in (\frac{(\log n)^\alpha}{(\log\log n)^{\alpha + 2}}, (\log n)^\alpha)$. Define
\begin{align*}
    &\mathcal A_{\mathsf{mult}} = \Big\{ \tfrac{ (\log n)^{\frac{\alpha}{2}} }{ (\log\log n)^{\frac{\alpha}{2}+2.9} } \leq \max_{k \in \mathcal A_*(0)} \{ \mathbf X_{2}^{(k)} \} - \max_{k \in \mathcal A_*(0)}^{(2)} \{ \mathbf X_{2}^{(k)} \} \leq (\log \log n)^{\frac{\alpha}{2}+2.9} (\log n)^{\frac{\alpha}{2}} \Big\}\,,\\
    &\mathcal A_{\mathsf{order};1} = \cap_{j \geq 1} \Big\{ \max_{k \in \mathcal A_*(j)} \{ \mathbf X_{2}^{(k)} \} \leq \max_{k \in \mathcal A_*(0)} \{ \mathbf X_{2}^{(k)} \} - (\log \log n)^{-\alpha} (\log n)^{\alpha-1} \Big\} \,, \\
    &\mathcal A_{\mathsf{order};2} = \Big\{ \max_{k \in \mathcal A_*(10)} \{ \mathbf X_{2}^{(k)} \} \geq \max_{k \in \mathcal A_*(0)} \{ \mathbf X_{2}^{(k)} \} - (\log \log n)^{\alpha} (\log n)^{\alpha-1} \Big\} \,.
\end{align*}
It is clear that under the event $\mathcal A_{\mathsf{order};1}$ we have $\mathcal L(\mathbf X_2) \subset \mathcal L(\mathbf X_1)$.
Note that for all $j \geq 1$ we have
\begin{align*}
    &\tfrac{n\mathbf M_1^\alpha}{\mathbf Z_1} - \tfrac{n(\mathbf M_1 -j)^{\alpha}}{\mathbf Z_1} > \tfrac{n}{\mathbf Z_1}\cdot \mathbf M_1^{\alpha - 1} j \\
    >\ & (\log n)^{\alpha-1}(\log\log n)^{-\alpha}+ \sqrt{\tfrac{n\mathbf M_1^\alpha}{\mathbf Z_1}\cdot (\log\log n)^5} + \sqrt{\tfrac{n(\mathbf M_1 - j)^\alpha}{\mathbf Z_1} \cdot j(\log \log n)^5} \,.
\end{align*}
Thus, we have 
\begin{align}
    &\Pb(\mathcal A_{\mathsf{order};1}^c \cap \Ecal_{\ref{lem-first-round-distribution-features}} ) \nonumber \\
    \leq\ & \Pb\left( \max_{k \in \mathcal A_*(j)} \{ \mathbf X_{2}^{(k)} \} > \max_{k \in \mathcal A_*(0)} \{ \mathbf X_{2}^{(k)} \} -  (\log \log n)^{-\alpha} (\log n)^{\alpha-1} \mbox{ for some } 1 \leq j \leq \frac{\mathbf M_1}{2} \right) \nonumber \\
    &+ \Pb\left( \max_{k \in \mathcal A_*(j)} \{ \mathbf X_{2}^{(k)} \} > \max_{k \in \mathcal A_*(0)} \{ \mathbf X_{2}^{(k)} \} - (\log\log n)^{-\alpha} (\log n)^{\alpha-1} \mbox{ for some } j >\frac{\mathbf M_1}{2} \right)\nonumber \\
    \leq\ & \Pb\left( \max_{k \in \mathcal A_*(0)} \{ \mathbf X_{2}^{(k)} \} \leq \tfrac{n\mathbf M_1^\alpha}{\mathbf Z_1} - \sqrt{\tfrac{n\mathbf M_1^\alpha}{\mathbf Z_1}\cdot (\log\log n)^5} \right) \nonumber \\
    &+\sum_{1 \leq j \leq \mathbf M_1/2} \Pb\left( \max_{k \in \mathcal A_*(j)} \{ \mathbf X_{2}^{(k)} \} \geq \tfrac{n(\mathbf M_1-j)^\alpha}{\mathbf Z_1} + \sqrt{\tfrac{n(\mathbf M_1-j)^\alpha}{\mathbf Z_1}\cdot j(\log\log n)^5} \right) \nonumber \\
    &+ n\Pb\left( \mathsf{Bin}\left( n,\tfrac{\mathbf M_1^\alpha}{4\mathbf Z_1} \right) \geq \tfrac{n\mathbf M_1^\alpha}{2\mathbf Z_1} \right) \nonumber \\
    \leq\ & 2\exp(-\tfrac{(\log \log n)^5}{5}) + \sum_{1 \leq j \leq \mathbf M_1/2} 2\exp(-\tfrac{j(\log \log n)^5}{10}) + n\exp(-(\log n)^{\alpha-o(1)}) = o(1)\,. \label{eq-bound-away-from-highests-alpha>2}
\end{align}
Similarly, using
\begin{align*}
    &\frac{n\mathbf M_1^\alpha}{\mathbf Z_1} - \frac{n(\mathbf M_1 -10)^{\alpha}}{\mathbf Z_1}< \frac{20\alpha n}{\mathbf Z_1}\cdot \mathbf M_1^{\alpha - 1}  \\
    <\ & (\log n)^{\alpha-1}(\log \log n)- \sqrt{\tfrac{n\mathbf M_1^\alpha}{\mathbf Z_1}\cdot (\log \log n)^5} - \sqrt{\tfrac{n(\mathbf M_1 - 10)^\alpha}{\mathbf Z_1} \cdot (\log \log n)^5} \,,
\end{align*}
we have
\begin{align}
    &\Pb( \mathcal A_{\mathsf{order};2}^c \cap \Ecal_{\ref{lem-first-round-distribution-features}} ) \nonumber \\
    \leq\ & \Pb\left( \max_{k \in \mathcal A_*(10)} \{ \mathbf X_{2}^{(k)} \} > \max_{k \in \mathcal A_*(0)} \{ \mathbf X_{2}^{(k)} \} -  (\log \log n)^{\alpha} (\log n)^{\alpha-1} \right) \nonumber \\
    \leq\ & \Pb\left( \max_{k \in \mathcal A_*(10)} \{ \mathbf X_{2}^{(k)} \} \geq \tfrac{n(\mathbf M_1-10)^\alpha}{\mathbf Z_1} + \sqrt{\tfrac{n(\mathbf M_1-10)^\alpha}{\mathbf Z_1} \cdot 10(\log\log n)^5} \right) \nonumber \\
    &+ \Pb\left( \max_{k \in \mathcal A_*(0)} \{ \mathbf X_{2}^{(k)} \} \leq \tfrac{n\mathbf M_1^\alpha}{\mathbf Z_1} - \sqrt{\tfrac{n\mathbf M_1^\alpha}{\mathbf Z_1}\cdot (\log\log n)^5} \right) \nonumber \\
    \leq\ & 2\exp(-\frac{(\log \log n)^5}{5}) + 2\exp(-\frac{10(\log \log n)^5}{10}) = o(1) \,.  \label{eq-bound-away-from-highests-alpha>2-another}
\end{align}
Combining \eqref{eq-bound-away-from-highests-alpha>2} and \eqref{eq-bound-away-from-highests-alpha>2-another}, we have 
\begin{align}
    \Pb^*\Big( \mathcal L(\mathbf X_2) \subset \mathcal L(\mathbf X_1), \tfrac{ (\log n)^{\alpha-1} }{ (\log\log n)^{\alpha} } \leq \max_{k \in [n]} \{ \mathbf X_2^{(k)} \} - \max^{(2)}_{k \in [n]} \{ \mathbf X_2^{(k)} \} \leq \tfrac{ (\log n)^{\alpha-1} }{ (\log\log n)^{-\alpha} } \Big) \nonumber \\
    \geq \Pb^*( \Ecal_{\ref{lem-first-round-distribution-features}} \cap \mathcal A_{\mathsf{order};1} \cap \mathcal A_{\mathsf{order};2}) = 1 - o(1)\,,\label{control-second-step-case-2.1}
\end{align}
Also, by Lemma~\ref{lem-Poisson-coupling}, given any $\mathbf X_1 \in \Omega_n$ there exists a coupling of $\mathbf X_2$ and $(\mathbf Y_2^{(i)} :\sim \mathsf{Pois}(\frac{(\mathbf X_1^{(i)})^\alpha}{\mathbf Z_{\mathbf X_1}}))_{i\in[n]}$ such that $\Pb(\max_{i \in [n]} |\mathbf X_2^{(i)} - \mathbf Y_2^{(i)}| \leq 4\,;\mathbf X_1) = 1 - o(1)$. Note that $\log |\mathcal A_*(0)| \leq (\log \log n)^3$ under the event $\Ecal_{\ref{lem-first-round-distribution-features}}$. Therefore, using Item~(1) of Lemma~\ref{lem-Poisson-maximum-1st-minus-2nd} with $(\lambda , m) = ( \frac{\mathbf M_{\mathbf X_1}^\alpha}{\mathbf Z_{\mathbf X_1}},|\mathcal A_*(0)|)$ we have
\begin{align*}
\Pb^{**}\Big(\tfrac{ (\log n)^{\frac{\alpha}{2}} }{ (\log\log n)^{\frac{\alpha}{2}+2.9} } \leq \max_{k \in \mathcal A_*(0)} \{ \mathbf Y_{2}^{(k)} \} - \max_{k \in \mathcal A_*(0)}^{(2)} \{ \mathbf Y_{2}^{(k)} \} \leq (\log \log n)^{2.9} (\log n)^{\frac{\alpha}{2}}\,; \Ecal_{\ref{lem-first-round-distribution-features}}\Big) = 1-o(1)\,,
\end{align*}
which gives
\begin{align*}
    \Pb^{**}( \mathcal A_{\mathsf{mult}}^c \cap \Ecal_{\ref{lem-first-round-distribution-features}} ) \leq o(1) + \Pb^{**}(\max_{i \in [n]} |\mathbf X_2^{(i)} - \mathbf Y_2^{(i)}| >4)= o(1)\,.
\end{align*}
Therefore, we have
\begin{align}
    \Pb^{**}\Big( \mathcal L(\mathbf X_2) \subset \mathcal L(\mathbf X_1), \tfrac{ (\log n)^{\frac{\alpha}{2}} }{ (\log\log n)^{\alpha} } \leq \max_{k \in [n]} \{ \mathbf X_2^{(k)} \} - \max^{(2)}_{k \in [n]} \{ \mathbf X_2^{(k)} \} \leq \tfrac{ (\log n)^{\frac{\alpha}{2}} }{ (\log\log n)^{-\alpha} } \Big) \nonumber \\
    \geq \Pb^{**}(\Ecal_{\ref{lem-first-round-distribution-features}} \cap \mathcal A_{\mathsf{mult}} \cap \mathcal A_{\mathsf{order};1} \cap \mathcal A_{\mathsf{order};2}) = 1 - o(1)\,. \label{control-second-step-case-2.2}
\end{align}
Combining \eqref{control-second-step-case-2.1} and \eqref{control-second-step-case-2.2} leads to the desired result.
\end{proof}

Then, before we present the core lemma, we introduce a lemma that describes the transitions of $(M,\delta)$ in the first constant number of steps (recall Definition~\ref{eq-def-(M,1-delta)-admissible}).

\begin{lemma}{\label{lem-prelim-gap-est-time-2-to-200}}
   Suppose that 
   \begin{equation}
       M \geq \frac{(\log n)^{\alpha}}{(\log \log n)^{2\alpha}}, \quad \frac{1}{ (\log n)^{0.1+\frac{\alpha}{2}} } \leq\delta\leq \frac{1}{(\log n)^{0.9}}, \quad Z \leq n\log\log n \,.
   \end{equation}
   Define $\mathfrak T(M,\delta,Z)$ to be the collection of $(M',\delta')$ such that 
   \begin{equation}{\label{eq-def-F_1-on-M-and-delta}}
       \frac{M^{\alpha}n}{2Z} \leq M' \leq \frac{2M^{\alpha}n}{Z}\,,\quad \frac{\alpha\delta}{2} < \delta' < 2\alpha\delta\,,
   \end{equation}
   Given any $(\mathbf X_1,\ldots,\mathbf X_n) \in \overline{\Omega}_n(M,\delta,Z)$, we have
   \begin{equation}\label{eq-after-second-round-estimate}
       \Pb\left( \mathcal L(\mathbf X') = \mathcal L(\mathbf X)\,, (\mathbf X_1',\ldots,\mathbf X_n') \in \cup_{ (M',\delta') \in \mathfrak T(M,\delta,Z) } \overline{\Omega}_n(M',\delta') \right) \geq 1-o\left( \frac{1}{\log n} \right) \,.
   \end{equation}
\end{lemma}
\begin{proof}
    Without loss of generality we assume $\mathbf X_1 = \max_{1 \leq i \leq n} \mathbf X_i$. Then, 
    \begin{align}
        \Pb\left( \left| \mathbf X_1' - \tfrac{nM^\alpha}{Z} \right| \geq (\log \log n) \sqrt{\tfrac{nM^\alpha}{Z}} \right) \leq 2\exp(-\tfrac{(\log \log n)^2}{5}) \,.  \label{eq-2-to-200-maximum}
    \end{align}
    In addition, we have
    \begin{align}
        \Pb\left( \left| \max_{2 \leq i \leq n} \mathbf X_i' - \tfrac{nM^\alpha(1-\delta)^\alpha}{Z} \right| \geq (\log\log n) \sqrt{\tfrac{(n\log n) M^\alpha(1-\delta)^{\alpha}}{Z}} \right) \nonumber \\
        \leq n\Pb\left( \left| \mathsf{Bin}(n,\tfrac{M^\alpha(1-\delta)^{\alpha}}{Z})-\tfrac{nM^\alpha(1-\delta)^\alpha}{Z} \right| \geq (\log\log n)\sqrt{\tfrac{(n\log n)\cdot M^\alpha(1-\delta)^{\alpha}}{Z}} \right) \nonumber \\
        \leq 2n\exp(-\tfrac{(\log\log n)^2(\log n)}{5}) < \exp(-(\log \log n)^{1.9})\,,  \label{eq-2-to-200-second-maximum}
    \end{align}
    Note that
    \begin{align*}
        \frac{ \frac{nM^\alpha(1-\delta)^\alpha}{Z} + (\log \log n)\sqrt{\frac{(n\log n)\cdot M^\alpha(1-\delta)^{\alpha}}{Z}}}{\frac{nM^\alpha}{Z} - (\log \log n)\sqrt{\frac{nM^\alpha}{Z}}}  \leq 1-\frac{\alpha\delta}{2}
    \end{align*}
    and similarly we have
    \begin{align*}
         \frac{ \frac{nM^\alpha(1-\delta)^\alpha}{Z} - (\log \log n) \sqrt{\frac{(n\log n)\cdot M^\alpha(1-\delta)^{\alpha}}{Z}}}{\frac{nM^\alpha}{Z} + (\log \log n)\sqrt{\frac{nM^\alpha}{Z}}} \geq 1-2\alpha\delta \,,
    \end{align*}
    it is clear that \eqref{eq-after-second-round-estimate} follows from \eqref{eq-2-to-200-maximum} and \eqref{eq-2-to-200-second-maximum}.
\end{proof}

Then, we present the core lemma in the $\alpha\geq 2$ case, which shows that we can bring $(M,\delta)$ to the initial phase in \eqref{eq-assum-lem-2.4} in constant steps. Moreover, the order of the resulting $\delta$ will be affected solely by the event $E_{\mathsf{unique}}$, which affects the absorption time by Proposition~\ref{lem-repeat-used}.

\begin{proposition}{\label{cor-core-lemma-alpha>=2}}
    Given $\alpha \geq 2$, define 
    \begin{align*}
        &\mathfrak T_{\mathsf{crit}}= \left\{ (M,\delta):(M,\delta) \mbox{ satisfying } \eqref{eq-assum-lem-2.4} \right\} \,; \\
        &\mathfrak T_{\mathsf{unique}} = \Big\{ (M,\delta) \in\mathfrak T_{\mathsf{crit}}: (\log\log n)^{-300-2\alpha} < \tfrac{\delta^{-1}}{\log n} < (\log\log n)^{300+2\alpha} \Big\}\,;  \\
        &\mathfrak T_{\mathsf{mult}} = \Big\{ (M,\delta) \in\mathfrak T_{\mathsf{crit}}: (\log\log n)^{-300-2\alpha} < \tfrac{\delta^{-1}}{(\log n)^{\frac{\alpha}{2}}} < (\log\log n)^{300+2\alpha} \Big\} \,.
    \end{align*}
    In addition, define the event
    \begin{align*}
        \Ecal_{\mathsf{lead}} &= \{\mathcal L(\mathbf X_{2})=\mathcal L(\mathbf X_{3})=\ldots=\mathcal L(\mathbf X_{10})\} \,.
    \end{align*}
    Then we have 
    \begin{align*}
        &\Pb\left( \Ecal_{\mathsf{lead}}\cap\{\mathbf X_{10} \in \cup_{(M,\delta) \in \mathfrak T_{\mathsf{unique}}}\overline{\Omega}_n(M,\delta)\}\mid E_{\mathsf{unique}} \right) = 1-o(1)\,, \\ 
        &\Pb\left( \Ecal_{\mathsf{lead}}\cap\{\mathbf X_{10} \in \cup_{(M,\delta) \in \mathfrak T_{\mathsf{mult}}}\overline{\Omega}_n(M,\delta)\}\mid E^c_{\mathsf{unique}} \right) = 1-o(1)\,.
    \end{align*}
\end{proposition}
\begin{proof}
 For $k \in [10]$, define
    \begin{align*}
        \Ecal_{\mathsf{boundZ};k} = \cap_{1 \leq i \leq k} \{\mathbf Z_i < n\log \log n \} \,.
    \end{align*}
    Note that by Item (1) of Lemma~\ref{lem-first-round-distribution-features} we have $\Pb(E_{\mathsf{unique}}), \Pb(E^c_{\mathsf{unique}}) \geq \frac{1}{\log n}$. Also, by Lemma~\ref{lem-good-event-E-1}, for all $k \in [10]$ we have
    \begin{align*}
        \Pb(\Ecal^c_{\mathsf{boundZ};k}) \leq \Pb(\mathcal E^c_{\ref{lem-good-event-E-1}}) \leq \frac{1}{\sqrt{n}} \,.
    \end{align*}
    Thus, for all $k \in [10]$ we have $\Pb(\Ecal^c_{\mathsf{boundZ};k}\mid E_{\mathsf{unique}}), \Pb(\Ecal_{\mathsf{boundZ};k}^c|E^c_{\mathsf{unique}}) = o(1)$. Define
    \begin{align*}
        \mathfrak T^{(i)}_{\mathsf{unique}} = \Big\{ & (M, \delta) \in\mathfrak T_{\mathsf{crit}}: \left( \tfrac{(\log n)^\alpha}{2(\log\log n)^{2\alpha}} \right)^{\alpha^{i-2}} \leq M \leq \left( 2(\log n)^\alpha(\log \log n)^{2\alpha} \right)^{\alpha^{i-2}}, \\
        & (\log \log n)^{-300-\alpha}(2\alpha)^{2-i} \leq \frac{\delta^{-1}}{\log n} \leq (\log \log n)^{300+\alpha}(2\alpha)^{i-2} \Big\} 
    \end{align*}
    and 
    \begin{align*}
        \mathfrak T^{(i)}_{\mathsf{mult}} &= \Big\{ (M,\delta) \in\mathfrak T_{\mathsf{crit}}: \left( \tfrac{(\log n)^\alpha}{2(\log \log n)^{2\alpha}} \right)^{\alpha^{i-2}} \leq M \leq \left( 2(\log n)^\alpha(\log \log n)^{2\alpha} \right)^{\alpha^{i-2}}, \\
        & (\log\log n)^{-300-\alpha}(2\alpha)^{2-i} \leq \tfrac{\delta^{-1}}{(\log n)^{\frac{\alpha}{2}}} \leq (\log \log n)^{300+\alpha}(2\alpha)^{i-2} \Big\} \,.
    \end{align*}
    By Lemma~\ref{lem-alpha>2-second-round} we have
    \begin{align*}
        \Pb\Big( \Ecal_{\mathsf{boundZ};2} \cap \Big\{ \mathbf X_2 \in \cup_{(M,\delta) \in \mathfrak T^{(2)}_{\mathsf{unique}}} \overline{\Omega}_n(M,\delta) \Big\} \mid E_{\mathsf{unique}} \Big) = 1-o(1)
    \end{align*}
    and 
    \begin{align*}
        \Pb\Big( \Ecal_{\mathsf{boundZ};2} \cap \Big\{ \mathbf X_2 \in \cup_{(M,\delta) \in \mathfrak T^{(2)}_{\mathsf{mult}}} \overline{\Omega}_n(M,\delta) \Big\} \mid E^c_{\mathsf{unique}} \Big) = 1-o(1) \,.
    \end{align*}
    Thus, for $3 \leq k \leq 10$, by Lemma~\ref{lem-prelim-gap-est-time-2-to-200} we can inductively prove for $E\in\{ E_{\mathsf{unique}}, E^c_{\mathsf{unique}} \}$
    \begin{align*}
        &\Pb\Big( \Ecal_{\mathsf{boundZ};k}\cap \cap_{3\leq h \leq k}\Big\{\{ \mathcal L(\mathbf X_{h-1})=\mathcal L(\mathbf X_{h})\}\cap \big\{\mathbf X_h \in \cup_{(M,\delta) \in \mathfrak T^{(h)}_{\mathsf{unique}}} \overline{\Omega}_n(M, \delta) \big\} \Big\}\mid E \Big) = 1-o(1) \,.
    \end{align*}
    Note that $\mathfrak T^{(10)}_{\mathsf{unique}} \subset\mathfrak T_{\mathsf{unique}}$ and $\mathfrak T^{(10)}_{\mathsf{mult}} \subset\mathfrak T_{\mathsf{mult}}$. The proof is completed.
\end{proof}

Finally, based on Proposition~\ref{cor-core-lemma-alpha>=2}, we can complete the proof of Theorems~\ref{thm-rapid-ordering-supercritical} and \ref{thm-all-or-nothing-transition} when $\alpha\geq 2$.
\begin{proof}[Proof of Theorems~\ref{thm-rapid-ordering-supercritical} and \ref{thm-all-or-nothing-transition} for $\alpha \geq 2$]
    We first prove Theorem~\ref{thm-rapid-ordering-supercritical}. For $\alpha \geq 2$, using Lemma~\ref{lem-alpha>2-second-round}, Proposition~\ref{cor-core-lemma-alpha>=2} and Proposition~\ref{lem-repeat-used}, we see that (write $\tau=\frac{2\log\log(n)}{\log(\alpha)}$ and $\tau'=\frac{(1+\tfrac{\alpha}{2})\log\log(n)}{\log(\alpha)}$)
    \begin{align*}
        \Pb_n^{\alpha}\Big( \frac{ \mathcal T-10 }{ \tau } = 1+o(1) \mid E_{\mathsf{unique}} \Big)&=1-o(1) \,,\\
        \Pb_n^{\alpha}\Big( \frac{ \mathcal T-10 }{ \tau' } = 1+o(1) \mid E^c_{\mathsf{unique}} \Big)&=1-o(1) \,,
    \end{align*}
    thus we have shown Theorem~\ref{thm-rapid-ordering-supercritical} for $\alpha \geq 2$. 

    Then, we prove Theorem~\ref{thm-all-or-nothing-transition}. Based on Lemma~\ref{lem-alpha>2-second-round}, Proposition~\ref{cor-core-lemma-alpha>=2} and Proposition~\ref{lem-repeat-used}, we see that 
    \begin{align*}
        \Pb_n^{\alpha}\left(\mathcal L(\mathbf X_2)\subseteq \mathcal L(\mathbf X_1),\mathcal L(\mathbf X_2) = \ldots =\mathcal L(\mathbf X_{10})=\{\mathcal W\} \right)=1-o(1) \,.
    \end{align*}
    Thus we have shown Theorem~\ref{thm-all-or-nothing-transition} for $\alpha > 2$.
\end{proof}

\subsection{The case $1<\alpha<2$}{\label{subsec:case-k-geq-2}}

Now we focus on the case $1 < \alpha < 2$, where there exists $k \geq 2$ such that $\alpha\in(\lambda_k,\lambda_{k-1}]$. In this case, it is clear that $\alpha^{k-1}\leq 2<\alpha^k$. The key to our argument is to show the following lemma.
\begin{proposition}{\label{lem-k-geq-2-very-first-rounds}}
    Suppose $\alpha^{k-1}\leq 2<\alpha^k$. Define 
    \begin{equation}
        m_0 = \left\lceil \frac{\log 101}{\log \alpha} \right\rceil \geq k \,.
    \end{equation}
    Then the following hold with probability $1-o(1)$:
    \begin{enumerate}
        \item[(1)] $\mathcal L(\mathbf X_t)$ is a singleton set for each $2 \leq t \leq k$. In addition, all $\{ \mathcal L(\mathbf X_t):\alpha^{t-1}<2 \}$ are distinct.
        \item[(2)] $\mathcal L(\mathbf X_t) = \mathcal L(\mathbf X_{t+1})$ for $k \leq t \leq m_0$.
        \item[(3)] $\mathbf X_{m_0+1} \in \overline{\Omega}_n (M,\delta)$ for some 
        \begin{align*}
            M = (\log n)^{\Theta(1)} \geq (\log n)^{101} \mbox{ and } \delta = (\log n)^{-1+o(1)} \,,
        \end{align*}
        which satisfies the conditions of Proposition~\ref{lem-repeat-used}.
    \end{enumerate}
    In particular, the $k$-th round is the exact round that determines the eventual winner by Proposition~\ref{lem-repeat-used}.
\end{proposition}
Based on Proposition~\ref{lem-k-geq-2-very-first-rounds}, we can first finish the proof of Theorems~\ref{thm-all-or-nothing-transition} and \ref{thm-rapid-ordering-supercritical} when $1<\alpha<2$. 
\begin{proof}[Proof of Theorems~\ref{thm-rapid-ordering-supercritical} and \ref{thm-all-or-nothing-transition} for $1<\alpha<2$]
    We first prove Theorem~\ref{thm-rapid-ordering-supercritical}. Using Item~(3) in Proposition~\ref{lem-k-geq-2-very-first-rounds} and Proposition~\ref{lem-repeat-used}, we see that (write $\tau=\frac{2\log\log(n)}{\log(\alpha)}$)
    \begin{align*}
        \Pb_n^{\alpha}\Big( \frac{ \mathcal T-m_0 }{ \tau } = 1+o(1) \Big)=1-o(1) \,.
    \end{align*}
    Thus, we have shown Theorem~\ref{thm-rapid-ordering-supercritical}.

    Then, we prove Theorem~\ref{thm-all-or-nothing-transition}. Suppose $\alpha\in(\lambda_k,\lambda_{k-1})$, combining Proposition~\ref{lem-k-geq-2-very-first-rounds} and Proposition~\ref{lem-repeat-used}, we see that 
    \begin{align*}
        \Pb_n^{\alpha}\left( \mathcal L(\mathbf X_k) \neq \mathcal L(\mathbf X_i) \mbox{ for all }1 \leq i \leq k-1\,, \mathcal L(\mathbf X_{k}) =\mathcal L(\mathbf X_{m_0+1})=\{\mathcal W\} \right)=1-o(1) \,.
    \end{align*}
    Thus, we have shown Theorem~\ref{thm-all-or-nothing-transition}. 
\end{proof}

The rest of this subsection is devoted to the proof of Proposition~\ref{lem-k-geq-2-very-first-rounds}. For notational convenience, define $\mathbf M_t = \mathbf M_{\mathbf X_t}$ and $\mathbf Z_t = \mathbf Z_{\mathbf X_t}$. Define the events 
\begin{align*}
    \mathcal E_{1,t} = \left\{ \mathbf Z_t \leq n\log \log n \right\}, \ \mathcal E_{2,t} = \Big\{ \tfrac{(\log n)^{\alpha^{t-1}}}{(2\log \log n)^{(\alpha^{t}-1)/(\alpha-1)}} \leq \mathbf M_t \leq (2\log n)^{\alpha^{t-1}} \Big\} \,.  
\end{align*}
In addition, define the level sets at time $t$ by
\begin{align}
    &\mathcal A_{t}(j) = \left\{ k \in [n]: \mathbf X^{(k)}_{t} = \mathbf M_{\mathbf X_t} - j \right\}, \quad \mathcal A_t[j_1,j_2] = \cup_{j_1 \leq j \leq j_2} \mathcal A_t(j) \,.  \label{eq-def-mathcal-A-t-j} 
\end{align}
Define
\begin{equation}{\label{eq-def-mathtt-N-t-puls-minus}}
    \begin{aligned}
        &\mathtt N_t^+ = \lfloor(\log n)^{1+\alpha^{t-1}-\alpha^{t}}(\log \log n)^{10\cdot 30^{t-1}}\rfloor \,, \\
        &\mathtt N_t^- = \lfloor(\log n)^{1+\alpha^{t-1}-\alpha^{t}}(\log \log n)^{-10\cdot 30^{t-1}}\rfloor \,.
    \end{aligned}
\end{equation}
Also define 
\begin{align}{\label{eq-def-mathtt-n-t-plus-minus}}
    \mathtt n_t^- = \lfloor(\log n)^{\alpha^{t-1}-1} \exp( 0.1(40\alpha)^{t}\sqrt{\log\log n} )\rfloor \,, \quad \mathtt n_t^+ =  \lfloor\frac{\log n}{(\log\log n)^{30^t}}\rfloor\,.
\end{align}
Define $\Ecal_{3,t}$ to be the event that
\begin{align*}
    \begin{cases}
    \mathcal L(\mathbf X_{t}) \subseteq [n], & t=1 \,;\\
        \mathcal L(\mathbf X_{t}) \subseteq \mathcal A_{t-1}[\mathtt N_{t-1}^-,\mathtt N_{t-1}^+], & t\geq 2\,,\alpha^{t-1}<2  \,; \\
        \mathcal L(\mathbf X_t) \subseteq \mathcal A_{t-1}[0,\mathtt n_{t-1}^-], & \alpha^{t-1}=2 \,; \\
        \mathcal L(\mathbf X_{t}) \subseteq \mathcal A_{t-1}(0), & \alpha^{t-1}>2  \,.
    \end{cases} 
\end{align*}
Next, for each $\alpha^{t-1}<2$ we define $\Ecal_{4,t}$ to be the event on which all of the following bounds hold:
\begin{equation}\label{eq-At-ub-lb}
    |\mathcal A_t(j)| 
    \begin{cases}
        \leq \exp((\log\log n)\exp((40\alpha)^{t}\sqrt{\log\log n})), \quad & 1 \leq j \leq \mathtt n_t^- \,; \\
        \leq \exp((\log\log n)^{30^{t-1}}(\log n)^{1-\alpha^{t-1}}j), \quad & j > \mathtt n_t^- \,;  \\
        \geq \exp((\log\log n)^{-30^{t-1}}(\log n)^{1-\alpha^{t-1}}j), \quad & \mathtt n_t^- \leq j \leq \mathtt n_t^+ \,.
    \end{cases}
\end{equation}
For each $\alpha^{t-1}\geq 2$ we define $\Ecal_{4,t}$ by
\begin{align*}
    |\mathcal A_t(j)| &\leq \exp((\log\log n)\exp((40\alpha)^{t} \sqrt{\log\log n})) \\
    &+ \exp((\log\log n)^{30^{t-1}}(\log n)^{1-\alpha^{t-1}}j) \mbox{ for all } j \geq 1 \,.
\end{align*}
Next, for $t=1$, define the event $\Ecal_{5,t}=\emptyset^c$; for $t \geq 2$ such that $\alpha^{t-1} < 2$, define $\Ecal_{5,t}$ to be the event that
\begin{align*}
    \frac{(\log n)^{\alpha^{t-1}-1}}{(4\log\log n)^{30^{t-1}}} \leq \max_{k \in [n]} \{ \mathbf X_t^{(k)} \} - \max_{k \in [n]}^{(2)} \{ \mathbf X_t^{(k)} \} 
    \leq (\log n)^{\alpha^{t-1}-1} (4\log\log n)^{30^{t-1}}\,;
\end{align*}
for $t \geq 2$ such that $\alpha^{t-1} \geq 2$, define $\Ecal_{5,t}$ to be the event that
\begin{align*}
    \frac{(\log n)^{\alpha^{t-1}-1}}{(4\log\log n)^{3\cdot30^{t-1}+\sqrt{\log \log n}}} &\leq \max_{k \in [n]} \{ \mathbf X_t^{(k)} \} - \max_{k \in [n]}^{(2)} \{ \mathbf X_t^{(k)} \} \\
    &\leq (\log n)^{\alpha^{t-1}-1} (4\log\log n)^{3\cdot30^{t-1}+\sqrt{\log \log n}}\,;
\end{align*}
Finally, define
\begin{equation}
    \Ecal_{6,t} =   
    \begin{cases}
        \{ |\mathcal L(\mathbf X_t)| \leq \exp((\log \log n)^3) \}, & t=1; \\
        \left\{ \mathcal L(\mathbf X_t) \mbox{ is singleton}, \mathcal L(\mathbf X_t) \not\in\{ \mathcal L(\mathbf X_{t-1}), \ldots,\mathcal L(\mathbf X_{1}) \}\right\}, & t\geq 2\,,\alpha^{t-1}<2; \\
        \{ \mathcal L(\mathbf X_t) \mbox{ is singleton} \}, & \alpha^{t-1}=2 \,; \\
        \{ \mathcal L(\mathbf X_t) \mbox{ is singleton}, \mathcal L(\mathbf X_t) = \mathcal L(\mathbf X_{t-1}) \}, & \alpha^{t-1}>2 \,.
    \end{cases}
\end{equation}
By construction of these events, Proposition~\ref{lem-k-geq-2-very-first-rounds} can be proved if 
\begin{align}\label{eq-cruicial-alpha-medium-lemma-reduction}
    \Pb_n^{\alpha}\Big( \cap_{1 \leq h \leq 6, 1 \leq t \leq m_0+1} \mathcal E_{h,t} \Big) = 1-o(1) \,,
\end{align}
By Lemma~\ref{lem-first-round-distribution-features}, $\cap_{1 \leq h \leq 6} \mathcal E_{h,1}$ holds with probability $1-o(1)$. It remains to prove \eqref{eq-cruicial-alpha-medium-lemma-reduction} via the following inductive argument.
\begin{lemma}{\label{lem-inductive-ecal-h-t}}
    We have the following estimates for $1 \leq t \leq m_0$:
    \begin{align}
        &\Pb_n^{\alpha}\left( \mathcal E_{1,t+1}^c; \cap_{1 \leq h \leq 6, 1 \leq s \leq t} \mathcal E_{h,s} \right) = o(1) \,; \label{eq-prob-E-1-fail} \\
        &\Pb_n^{\alpha}\left( \mathcal E_{2,t+1}^c; \cap_{1 \leq h \leq 6, 1 \leq s \leq t} \mathcal E_{h,s} \right) = o(1) \,; \label{eq-prob-E-2-fail} \\
        &\Pb_n^{\alpha}\left( \mathcal E_{3,t+1}^c; \cap_{1 \leq h \leq 6, 1 \leq s \leq t} \mathcal E_{h,s} \right) = o(1) \,; \label{eq-prob-E-3-fail} \\
        &\Pb_n^{\alpha}\left( \mathcal E_{4,t+1}^c; \cap_{1 \leq h \leq 6, 1 \leq s \leq t} \mathcal E_{h,s}  \right) = o(1) \,; \label{eq-prob-E-4-fail} \\
        &\Pb_n^{\alpha}\left( \mathcal E_{5,t+1}^c; \cap_{1 \leq h \leq 6, 1 \leq s \leq t} \mathcal E_{h,s}  \right) = o(1) \,. \label{eq-prob-E-5-fail} \\
        &\Pb_n^{\alpha}\left( \mathcal E_{6,t+1}^c; \cap_{1 \leq h \leq 6, 1 \leq s \leq t} \mathcal E_{h,s} \right) = o(1) \,. \label{eq-prob-E-6-fail}
    \end{align}
\end{lemma}
\begin{proof}
For notational convenience, we define $\mathtt P(\cdot) = \mathbb P_n^{\alpha}\left( \cdot; \cap_{1 \leq h \leq 6, 1 \leq s \leq t} \mathcal E_{h,s} \right)$.
We will prove \eqref{eq-prob-E-1-fail}--\eqref{eq-prob-E-6-fail} separately.

\subsubsection{Proof of (\ref{eq-prob-E-1-fail})}
By Lemma~\ref{lem-good-event-E-1}, we have $\Pb_{n}^{\alpha}(\mathcal E_{1,t+1}^c) = o(1)$,
and therefore \eqref{eq-prob-E-1-fail} holds. 

\subsubsection{Proof of (\ref{eq-prob-E-2-fail})}
Note that using Lemmas~\ref{lem-Poisson-coupling} and \ref{lem-Poisson-coupling-further}, given $\mathbf X_{t}$, we can couple $\mathbf X_{t+1}^{(1)},\ldots,\mathbf X_{t+1}^{(n)}$ with independent Poisson random variables 
\begin{align*}
    \eta^{(i)} \sim \mathsf{Pois}\left( \frac{n(\mathbf X_{t}^{(i)})^\alpha }{ \mathbf Z_{\mathbf X_t}} \right) \mbox{ where } 1 \leq i \leq n \,,
\end{align*}
such that
\begin{align}
    &\mathtt P\left( |\eta^{(i)} - \mathbf X_{t+1}^{(i)}| \leq 4 \mbox{ for all } 1 \leq i \leq n \right) = 1-o(1) \,; \label{eq-couple-error-at-most-4}\\
    &\mathtt P\left( \#\{i:\eta^{(i)} = \mathbf X_{t+1}^{(i)} = h \}\geq \frac{1}{2} \#\{i:\eta^{(i)} = h \} \mbox{ for all } h \right) = 1-o(1) \,.\label{eq-couple-identical-at-least-half}
\end{align}
Thus, it is straightforward to check that
\begin{align*}
    \mathtt P\left( \Ecal_{2,t+1}^c \right) &\leq \mathtt P\left( \mathbf M_{t+1} < \tfrac{n\mathbf M_t^\alpha}{2\mathbf Z_t} \mbox{ or } \mathbf M_{t+1}>\tfrac{2n\mathbf M_t^\alpha}{\mathbf Z_t} \right)  \\
    &\leq \mathtt P\left( \max\{ \eta^{(i)} \} < \tfrac{n\mathbf M_t^\alpha}{2\mathbf Z_t}+4 \mbox{ or } \max\{ \eta^{(i)} \} > \tfrac{2n\mathbf M_t^\alpha}{\mathbf Z_t}-4 \right) \\
    &\leq \exp(-\tfrac{n\mathbf M_t^\alpha}{4\mathbf Z_t}) \leq \exp(-\tfrac{(\log n)^{\alpha^{t}}}{(\log\log n)^2}) \,,
\end{align*}
and therefore \eqref{eq-prob-E-2-fail} holds.

\subsubsection{Proof of (\ref{eq-prob-E-3-fail})}
Recall \eqref{eq-couple-error-at-most-4} and \eqref{eq-couple-identical-at-least-half}, we will consider the following two cases separately: 

\noindent{\bf Case 1:} $\alpha^{t}\leq 2$. We will prove that $\mathcal L(\mathbf X_{t+1}) \in \mathcal A_t[\mathtt N_t^-,\mathtt N_t^+]$ with high probability when $\alpha^t < 2$, and $\mathcal L(\mathbf X_{t+1}) \in \mathcal A_t[0,\mathtt n_t^-]$ with high probability when $\alpha^t = 2$. Define
\begin{align}{\label{eq-def-varphi(lambda)}}
    \varphi(\lambda) = 100\log n + \tfrac{\lambda (\log n)^{\alpha^t - \alpha^{t-1}}}{(\log \log n)^{3 \cdot 30^{t-1}}} \,.
\end{align}
We will prove that
\begin{align}
    &\mathtt P\Big( \max_{k \in \mathcal A_t(0)} \{ \eta^{(k)} \} \leq \max_{\lambda \in [\max\{\mathtt N_t^+, \mathtt n_t^-\} +1,\mathbf M_t]} \big\{ \max_{k \in \mathcal A_t(\lambda)} \{ \eta^{(k)} \} + \varphi(\lambda) \big\} \Big) = o(1),  &\alpha^t \leq 2\,, \label{eq-mathbf-Bt-small-prob} \\
    &\mathtt P\Big( \max_{k \in \mathcal A_t[\mathtt N_t^-,\mathtt N_t^+]} \{ \eta^{(k)} \} < \max_{k \in \mathcal A_t[0,\mathtt N_t^- -1]} \{ \eta^{(k)} \}+10 \Big) = o(1),  &\alpha^t < 2\,.  \label{eq-toolarge-layers-lead-small-prob}
\end{align}
In particular, combining \eqref{eq-mathbf-Bt-small-prob}, \eqref{eq-toolarge-layers-lead-small-prob} and \eqref{eq-couple-error-at-most-4} will imply that
\begin{align}
    &\mathtt P\left( \mathcal L(\mathbf X_{t+1}) \not\in \mathcal A_t[\mathtt N_t^-,\mathtt N_t^+] \right) = o(1) \,,  &\alpha^t < 2\,, {\label{eq-nextmax-in-mathtt-Nts}} \\
    &\mathtt P\left( \mathcal L(\mathbf X_{t+1}) \not\in \mathcal A_t[0,\mathtt n_t^-] \right) = o(1) \,,  &\alpha^t = 2\,. {\label{eq-nextmax-in-mathtt-Nts-critical}}
\end{align}
We first show \eqref{eq-mathbf-Bt-small-prob} for $\alpha^t \leq 2$. Define
\begin{align}{\label{eq-def-mathbf-Ht-lambda}}
    \mathbf H_t(\lambda) = \frac{n(\mathbf M_t-\lambda)^\alpha}{\mathbf Z_t}
\end{align}
for notational convenience. For $\mathtt N_t^+ \leq \lambda \leq \frac{\mathbf M_t}{2}$, we have 
\begin{align}
    &\mathbf H_t(0) - \mathbf H_t(\lambda) =\frac{ n\mathbf M_t^{\alpha} - n(\mathbf M_t-\lambda)^{\alpha} }{\mathbf Z_t} \geq \frac{ n\alpha \lambda(\mathbf M_t- \lambda)^{\alpha-1} }{\mathbf Z_t} \nonumber\\
    &\geq (\tfrac{\lambda}{4} + \tfrac{\lambda}{4} +  \tfrac{\lambda}{4} )\cdot \tfrac{1}{\sqrt{\log \log n}}\cdot \sqrt{\tfrac{n}{\mathbf Z_t}} \cdot \mathbf M_t^{\alpha-1} \nonumber\\
    &\geq \sqrt{\lambda(\log n)^{1-\alpha^{t-1}}}(\log \log n)^{4\cdot 30^{t-1}} \sqrt{\tfrac{n}{\mathbf Z_t}} \cdot \mathbf M_t^{\alpha / 2} + 200\log n +\tfrac{\lambda (\log n)^{\alpha^t - \alpha^{t-1}}}{(\log \log n)^{2 \cdot 30^{t-1}}}  \nonumber\\
    &\geq (\log \log n)^{3\cdot 30^{t-1}}\sqrt{\tfrac{n(\mathbf M_t-\lambda)^{\alpha}}{\mathbf Z_t} \cdot (\log \log n)^{2\cdot30^{t-1}} \lambda(\log n)^{1-\alpha^{t-1}} } + (200\log n + \tfrac{\lambda (\log n)^{\alpha^t - \alpha^{t-1}}}{(\log \log n)^{2 \cdot 30^{t-1}}}  ) \nonumber\\
    &\geq (\log \log n)^{3\cdot 30^{t-1}}\sqrt{\mathbf H_t(\lambda)\cdot ((\log \log n)^{10\cdot 30^{t-1}} \vee \log |\mathcal A_t(\lambda)|)} \nonumber\\
    &+\sqrt{\mathbf H_t(0) \cdot (\log \log n)^{10}} +2\varphi(\lambda) \,.\label{eq-phi-lambda-bound-temp-1}
\end{align}
Next, note that $|\mathcal A_t(0)| \leq \exp((\log \log n)^{3})$ under the measure $\mathtt P$ by the definition of $\Ecal_{6,t}$. Thus, by the Poisson Chernoff bound we have 
\begin{align}
    &\mathtt P\Big( \max_{k \in \mathcal A_t(0)} \{ \eta^{(k)} \} \leq \mathbf H_t(0) - \sqrt{\mathbf H_t(0)(\log \log n)^4} \Big) \\
    \leq\ & \exp\Big(-\tfrac{(\log \log n)^4}{2(1+\frac{(\log \log n)^2}{\sqrt{\mathbf H_t(0)}})}\Big) < \exp(-\tfrac{(\log \log n)^4}{3})\,.\label{eq-phi-lambda-bound-temp-2}
\end{align}
Combining \eqref{eq-phi-lambda-bound-temp-1} with \eqref{eq-phi-lambda-bound-temp-2}, we have
\begin{align*}
    &\sum_{\max\{\mathtt N_t^+, \mathtt n_t^-\} \leq \lambda \leq \mathbf M_t / 2}\mathtt P\Big( \max_{ k \in \mathcal A_t(\lambda) } \{ \eta^{(k)} \} \geq \max_{\mathcal A_t[\mathtt N_t^-,\mathtt N_t^+]} \{ \eta^{(k)} \} - 2\varphi(\lambda) \Big) \\
    \leq\ & \sum_{\substack{\max\{\mathtt N_t^+, \mathtt n_t^-\} \leq \lambda \leq \mathbf M_t / 2 \\ \log|\mathcal A_t(\lambda)| > (\log \log n)^{10}} }\mathtt P\left( \max_{k \in \mathcal A_t(\lambda)} \{ \eta^{(k)} \} \geq \mathbf H_t(\lambda)+(\log\log n)^{3\cdot 30^{t-1}} \sqrt{\mathbf H_t(\lambda) \cdot \log |\mathcal A_t(\lambda)|} \right) \\
    +\ & \sum_{\substack{\max\{\mathtt N_t^+, \mathtt n_t^-\} \leq \lambda \leq \mathbf M_t / 2 \\ \log|\mathcal A_t(\lambda)| \leq (\log \log n)^{10}} }\mathtt P\left( \max_{k \in \mathcal A_t(\lambda)} \{ \eta^{(k)} \} \geq \mathbf H_t(\lambda)+(\log\log n)^{3\cdot 30^{t-1}} \sqrt{\mathbf H_t(\lambda) \cdot (\log \log n)^{10}} \right) \\
    &+ \mathtt P\left( \max_{k \in \mathcal A_t(0)} \{ \eta^{(k)} \} \leq \mathbf H_t(0) - \sqrt{\mathbf H_t(0)(\log \log n)^4} \right) \\
    \leq\ & \Big(\sum_{\max\{\mathtt N_t^+, \mathtt n_t^-\} \leq \lambda \leq \mathbf M_t / 2} |\mathcal A_t(\lambda)|\cdot\exp \big( -\tfrac{(\log \log n)^{6\cdot 30^{t-1}}\cdot ((\log \log n)^{10} \vee\log |\mathcal A_t(\lambda)|)}{3}\big) \Big)+ o(\frac{1}{\log n}) \\
    \leq\ & \sum_{\max\{\mathtt N_t^+, \mathtt n_t^-\} \leq \lambda \leq \mathbf M_t / 2}\exp(-\lambda (\log n)^{1-\alpha^{t-1}}(\log\log n)^{1.5 \cdot 30^{t-1}}) + \frac{\mathbf M_t\exp(-\frac{(\log \log n)^{10}}{4})}{2}\\
    +\ & o(\frac{1}{\log n}) = o(\frac{1}{\log n})\,,
\end{align*}
Also, for $\frac{\mathbf M_t}{2} \leq \lambda \leq \mathbf M_t$ we have
\begin{align*}
    &\Pb\Big( \mathsf{Pois}\left( \frac{n(\mathbf M_t-\lambda)^{\alpha}}{\mathbf Z_t} \right) \geq \frac{n\mathbf M_t^\alpha}{\mathbf Z_t} -\varphi(\lambda) \Big) \leq \Pb\Big( \mathsf{Pois}\left( \frac{n(\mathbf M_t/2)^{\alpha}}{\mathbf Z_t} \right) \geq \frac{n\mathbf M_t^\alpha}{\mathbf Z_t} -2\varphi(\mathbf M_t) \Big) \\
    \leq\ & \exp(-\frac{1}{4}\mathbf M_t (\log n)^{1-\alpha^{t-1}} (\log \log n)^{1.5 \cdot 30^{t-1}})\leq \exp(-\frac{1}{8}\lambda (\log n)^{1-\alpha^{t-1}} (\log \log n)^{1.5 \cdot 30^{t-1}}) \,.
\end{align*}
Therefore, the probability in \eqref{eq-mathbf-Bt-small-prob} is bounded by
\begin{align*}
    o(\frac{1}{\log n})+\sum_{\frac{\mathbf M_t}{2} \leq \lambda \leq \mathbf M_t} \exp( -\frac{1}{8} \lambda(\log n)^{1-\alpha^{t-1}} (\log\log n)^{1.5\cdot 30^{t-1}} ) =o( \frac{1}{\log n} ) \,,
\end{align*}
leading to \eqref{eq-mathbf-Bt-small-prob}. Next, we show \eqref{eq-toolarge-layers-lead-small-prob} for $\alpha^t < 2$. Note that in this case $\mathtt n_t^- < \mathtt N_t^- < \mathtt N_t^+ < \mathtt n_t^{+}$, so the third inequality in \eqref{eq-At-ub-lb} applies to $|\mathcal A_t(\lambda)|$ for all $\lambda \in [\mathtt N_t^-,\mathtt N_t^+]$. Define
\begin{align*}
    \mathtt N_t = \mathtt N_t^- \cdot (\log\log n)^{3\cdot 30^{t-1}} \in [\mathtt N_{t}^-, \mathtt N_t^{+}] \,.
\end{align*}
We have
\begin{align}
    &\log |\mathcal A_t[1,\mathtt N_t^-]| \leq \log \Big( \sum_{0 \leq i \leq \mathtt n_t^-}|\mathcal A_t(i)| + \sum_{\mathtt n_t^- < i \leq \mathtt N_t^-}|\mathcal A_t(i)| \Big) \nonumber\\
    \leq\ & \log\Big( \mathtt n_t^- (\log \log n)\exp(\exp((40\alpha)^t \sqrt{\log \log n})) + \mathtt N_t^- \exp((\log \log n)^{30^{t-1}} (\log n)^{1-\alpha^{t-1}} \mathtt N_t^-)\Big) \nonumber \\
    \leq\ & \log\Big(2\mathtt N_t^- \exp((\log \log n)^{30^{t-1}} (\log n)^{1-\alpha^{t-1}} \mathtt N_t^-) \Big) \nonumber \\
    \leq\ & 2(\log \log n)^{30^{t-1}} (\log n)^{1-\alpha^{t-1}} \mathtt N_t^- \leq (\log n)^{2-\alpha^t}\,. \label{eq-upperbound-At-[1,Nt-]}
\end{align}
Therefore, we have
\begin{align*}
&\quad\,\,\mathbf H_t(0) - \mathbf H_t(\mathtt N_t) =\tfrac{n(\mathbf M_t)^{\alpha}}{\mathbf Z_t} -\tfrac{n(\mathbf M_t-\mathtt N_t)^{\alpha}}{\mathbf Z_t}  \leq \tfrac{n\alpha}{\mathbf Z_t} \cdot \mathtt N_t(\mathbf M_t)^{\alpha - 1} \\
&\leq (\log \log n)^{3\cdot30^{t-1}} \mathtt N_t \cdot \sqrt{\tfrac{n}{\mathbf Z_t}} \cdot \mathbf M_t^{\alpha-1} \leq \tfrac{1}{2}\sqrt{\mathtt N_t(\log n)^{1-\alpha^{t-1}}(\log \log n)^{-30^{t-1}} }\cdot \sqrt{\tfrac{n}{\mathbf Z_t}} \cdot \mathbf M_t^{\alpha / 2} \\
&\leq \sqrt{\mathtt N_t(\log n)^{1-\alpha^{t-1}}(\log \log n)^{-30^{t-1}} }\cdot \sqrt{\tfrac{n}{\mathbf Z_t}} \cdot \mathbf M_t^{\alpha / 2} \\
&- 3\cdot \sqrt{2\mathtt N_t^-(\log n)^{1-\alpha^{t-1}}(\log \log n)^{30^{t-1}} }\cdot \sqrt{\tfrac{n}{\mathbf Z_t}} \cdot \mathbf M_t^{\alpha / 2} - 10\\
&\leq \sqrt{\mathbf H_t(\mathtt N_t)\cdot \log |\mathcal A_t(\mathtt N_t)|} - 3 \sqrt{\mathbf H_t(0) \cdot \log |\mathcal A_t[1,\mathtt N_t^-]|} - 10\,.
\end{align*}
Therefore, by Corollary~\ref{cor-est-maximum-Poisson} with $h = 9$, we have (denote by $\mathsf{PM}(m,\lambda)$ the maximum of $m$ independent $\mathsf{Pois}(\lambda)$ variables) the probability in \eqref{eq-toolarge-layers-lead-small-prob} is bounded by
\begin{align*}
    & \mathtt P\left( \mathsf{PM}\left( |\mathcal A_t[1,\mathtt N_t^-]|, \mathbf H_t(0) \right) \geq \mathbf H_t(0) + 3 \sqrt{\mathbf H_t(0) \log |\mathcal A_t[1,\mathtt N_t^-]|} \right) \\
    &+\mathtt P\left( \mathsf{PM}\left( |\mathcal A_t(\mathtt N_t)|, \mathbf H_t(\mathtt N_t) \right) \leq \mathbf H_t(\mathtt N_t) + \sqrt{\mathbf H_t(\mathtt N_t) \log |\mathcal A_t(\mathtt N_t)|} \right)  \\
    \leq\ & \exp(-|\mathcal A_t[1,\mathtt N_t^-]|^{1/3}) + \frac{1}{|\mathcal A_t(\mathtt N_t)|^3} \overset{\eqref{eq-At-ub-lb}}{\leq} o(\frac{1}{\log n})\,,
\end{align*}
which proves \eqref{eq-toolarge-layers-lead-small-prob}. 

\noindent{\bf Case 2:} $\alpha^t>2$. Define $\mathtt g_t=\frac{(\log n)^{\alpha^{t-1}-1}}{(4\log\log n)^{30^{t-1}}}$. In this case, by $\Ecal_{6,t}$, the maximum of $\mathbf X_t$ is unique and 
\begin{align*}
    (2\log \log n)^{-30^t} (\log n)^{\alpha^{t-1}-1} \leq \min\left\{ \lambda \geq 1:\mathcal A_t(\lambda) \neq \emptyset \right\} \leq (2\log \log n)^{3\cdot 30^t} (\log n)^{\alpha^{t-1}-1}  \,.
\end{align*}
Our goal is to prove that $\mathcal L(\mathbf X_{t+1}) = \mathcal L(\mathbf X_t)$ with probability $1-o(1)$. Without loss of generality, we may assume that $\mathcal L(\mathbf X_t)=\mathcal A_t(0)=\{ 1 \}$, where we have $\max_{k \in [n]}\eta^{(k)} = \eta^{(1)}$. Note that for $\mathtt g_t\leq \lambda \leq \frac{\mathbf M_t}{2}$, we have
\begin{align*}
    \mathbf H_t(0) - \mathbf H_t(\lambda) &= \frac{n(\mathbf M_t^\alpha - (\mathbf M_t - \lambda)^\alpha)}{\mathbf Z_t} \geq \frac{\lambda n}{\mathbf Z_t} \left( \frac{\mathbf M_t}{2} \right)^{\alpha-1} \\
    &\geq \frac{\lambda(\log n)^{\alpha^t - \alpha^{t-1}}}{(\log \log n)^{3\cdot 30^{t-1}}}+(\log \log n)^{30^t}\sqrt{\mathbf H_t(0)} \,.
\end{align*}
Therefore, we have
\begin{align}
    &\mathtt P\Big( \exists\ 1 \leq \lambda \leq \frac{\mathbf {M_t}}{2},\ \max_{ k \in \mathcal A_t(\lambda) } \{ \eta_{(k)} \} \geq \eta^{(1)} - \frac{\lambda(\log n)^{\alpha^t-\alpha^{t-1}}}{(\log \log n)^{4\cdot 30^{t-1}}} \Big) \nonumber  \\
    \leq\ & \mathtt P\Big( \eta^{(1)} \leq \mathbf H_t(0) - (\log\log n)^{30^t} \sqrt{\mathbf H_t(0)} \Big) \nonumber  \\
    &+ \sum_{\mathtt g_t \leq \lambda \leq \mathbf M_t/2} |\mathcal A_t(\lambda)| \cdot \mathtt P\Big( \mathsf{Pois}(\mathbf H_t(\lambda)) \geq \mathbf H_t(\lambda) + \tfrac{\lambda(\log n)^{\alpha^t-\alpha^{t-1}}}{2(\log \log n)^{3\cdot 30^{t-1}}} \Big) \nonumber \\
    \leq\ & o(1)+\sum_{\mathtt g_t \leq \lambda \leq \mathbf M_t/2} |\mathcal A_t(\lambda)|\exp(-\lambda^2 (\log n)^{\alpha^t-2\alpha^{t-1}+o(1)} ) \,.  \label{eq-prob-exist-small-level-excessively-large-1}
\end{align}
Since $\mathtt g_t = (\log n)^{\alpha^{t-1}-1+o(1)}$, we have
\begin{align*}
    &\lambda^2(\log n)^{\alpha^t - 2\alpha^{t-1}+o(1)} \geq \max\left\{ (\log n)^{o(1)}, \lambda(\log n)^{\alpha^t - \alpha^{t-1}-1+o(1)} \right\}  \\
    \geq\ & (\log n)^{\alpha^t-2} \max\left\{ (\log n)^{o(1)},\lambda(\log n)^{1-\alpha^{t-1}+o(1)} \right\} \geq (\log n)^{\Omega(1)} \cdot \log |\mathcal A_t(\lambda)|\,,
\end{align*}
where the first and the second inequality hold since $\alpha^t > 2$, and the third inequality holds by \eqref{eq-At-ub-lb}. Therefore, we have
\begin{align}
    \eqref{eq-prob-exist-small-level-excessively-large-1} &\leq o(1)+ \sum_{\mathtt g_t \leq \lambda \leq \mathbf M_t/2} \exp( -\frac{1}{2} \lambda^2(\log n)^{\alpha^t-2\alpha^{t-1}+o(1)} ) \nonumber \\
    &\leq o(1) + \exp(-(\log n)^{\alpha^t-2+o(1)}) = o(1)\,.\label{eq-prob-exist-small-level-excessively-large-2}
\end{align}
Next, for $\frac{\mathbf M_t}{2}<\lambda\leq \mathbf M_t$ we have
\begin{align*}
    \mathbf H_t(0) - \mathbf H_t(\lambda) \geq \mathbf H_t(0) - \mathbf H_t\left( \frac{\mathbf M_t}{2} \right) \geq \frac{\mathbf M_t(\log n)^{\alpha^t-\alpha^{t-1}}}{2(\log\log n)^{3\cdot 30^{t-1}}}+(\log \log n)^{30^t}\sqrt{\mathbf H_t(0)} \,.
\end{align*}
Therefore, 
we have
\begin{align}
    &\mathtt P\Big( \exists\ \frac{\mathbf {M_t}}{2} \leq \lambda \leq \mathbf M_t,\ \max_{k \in \mathcal A_t(\lambda)} \{ \eta^{(k)} \} \geq \eta^{(1)} - \tfrac{\lambda(\log n)^{\alpha^t-\alpha^{t-1}}}{(\log\log n)^{4\cdot 30^{t-1}}} \Big) \nonumber  \\
    \leq\ & \mathtt P\left( \eta^{(1)} \leq \mathbf H_t(0)-(\log\log n)^{30^t} \sqrt{\mathbf H_t(0)} \right) \nonumber  \\
    &+ n \max_{\mathbf M_t/2<\lambda\leq\mathbf M_t} \mathtt P\left( \mathsf{Pois}(\mathbf H_t(\lambda)) \geq \mathbf H_t(\lambda) + \tfrac{\mathbf M_t(\log n)^{\alpha^t-\alpha^{t-1}}}{2(\log\log n)^{3\cdot 30^{t-1}}} \right) \nonumber\\
    \leq\ & o(1) + n \sum_{ \mathbf M_t/2<\lambda\leq\mathbf M_t } \exp( -\frac{1}{16\mathbf H_t(\lambda)} \cdot \left( \tfrac{\mathbf M_t(\log n)^{\alpha^t-\alpha^{t-1}}}{(\log\log n)^{3\cdot 30^{t-1}}} \right)^2 ) \nonumber\\
    &+ n \sum_{ \mathbf M_t/2<\lambda\leq\mathbf M_t } \exp(-\frac{1}{8} \cdot\tfrac{\mathbf M_t(\log n)^{\alpha^t-\alpha^{t-1}}}{(\log\log n)^{3\cdot 30^{t-1}}} ) \nonumber \\
    \leq\ & o(1)+n\mathbf M_t\exp(-(\log n)^{\alpha^t}) = o(1)\,, \label{eq-prob-exist-small-level-excessively-large-3}
\end{align}
which yields that
\begin{align*}
    &\mathtt P\left( \mathsf{Pois}(\mathbf H_t(\lambda)) \geq \mathbf H_t(\lambda) + \tfrac{\mathbf M_t(\log n)^{\alpha^t-\alpha^{t-1}}}{2(\log \log n)^{3\cdot 30^{t-1}}} \right) \\
    \leq\ & \exp(-\frac{t^2_{\text{lemma}}}{2(t_{\text{lemma}} + \lambda_{\text{lemma}})}) \leq \exp(-\frac{t^2_{\text{lemma}}}{4t_{\text{lemma}} }) + \exp(-\frac{t^2_{\text{lemma}}}{4\lambda_{\text{lemma}} })\,.
\end{align*}
Combining \eqref{eq-prob-exist-small-level-excessively-large-2} and \eqref{eq-prob-exist-small-level-excessively-large-3}, we have
\begin{align}
    \mathtt P\left( \exists \lambda >0, \max_{ k \in \mathcal A_t(\lambda) } \eta^{(k)} \geq \eta^{(1)} - \frac{\lambda(\log n)^{\alpha^t - \alpha^{t-1}}}{(\log \log n)^{4\cdot 30^{t-1}}} \right) = o(1) \,,\label{eq-large-gap-when-t-large}
\end{align}
and therefore $\mathtt P(\mathcal L(\mathbf X_{t+1})=\mathcal L(\mathbf X_t))=1-o(1)$ from \eqref{eq-couple-error-at-most-4}. 

Finally, combining all the three cases, we have shown \eqref{eq-prob-E-3-fail}.

\subsubsection{Proof of (\ref{eq-prob-E-4-fail})}
Define
\begin{align}
    \mathcal A_{t+1}(\lambda,j) &= \left\{ k \in \mathcal A_{t}(\lambda): \mathbf X^{(k)}_{t+1} = \max_{i \in \mathcal A_t(\lambda)} \mathbf X_{t+1}^{(i)}-j \right\} \,.  \label{eq-def-mathcal-A-t-lambda-j}
\end{align}
In addition, recall \eqref{eq-couple-error-at-most-4} and \eqref{eq-couple-identical-at-least-half}, define
\begin{align}
    &\mathcal A'_{t+1}(j) = \left\{ k \in [n]: \mathbf \eta^{(k)} = \max_{i \in [n]} \{ \eta^{(i)}\} - j \right\} \,,  \label{eq-def-mathcal-A-t-j-prime}  \\
    &\mathcal A'_{t+1}(\lambda,j) = \left\{ k \in \mathcal A_{t}(\lambda): \eta^{(k)} = \max_{i \in \mathcal A_{t}(\lambda)} \{ \eta^{(i)} \} - j \right\} \,;  \label{eq-def-mathcal-A-t-lambda-j-prime}
\end{align}
Define $\Ecal'_{4,t+1}$ to be the event that 
\begin{equation}{\label{eq-At-prime-ub-lb}}
    |\mathcal A'_{t+1}(j)| 
    \begin{cases}
        \leq\exp( \frac{(\log\log n)\exp((40\alpha)^{t+1}\sqrt{\log \log n})}{2}), \ & 1 \leq j \leq 2\mathtt n_{t+1}^- \,; \\
        \leq\exp( \frac{(\log\log n)^{30^{t}}(\log n)^{1-\alpha^{t}} j}{2}), & j \geq 0.5\mathtt n_{t+1}^- \,; \\
        \geq \exp(2(\log\log n)^{-30^{t}}(\log n)^{1-\alpha^{t}}j), & 0.5\mathtt n_{t+1}^- \leq j \leq 2\mathtt n_{t+1}^+ \,.
    \end{cases}
\end{equation}
if $\alpha^{t}<2$, and to be the event that
\begin{align*}
    |\mathcal A_{t+1}'(j)| &\leq \exp(\frac{1}{2}(\log \log n) \exp((40\alpha)^{t+1}\sqrt{\log \log n})) \\
    &+ \exp(\frac{1}{2}(\log \log n)^{30^{t}}(\log n)^{1-\alpha^{t}}j)  \mbox{ for all } j \geq 1 
\end{align*}
if $\alpha^{t}\geq 2$. Note that if the events in \eqref{eq-couple-error-at-most-4} and \eqref{eq-couple-identical-at-least-half} hold, we have
\begin{align*}
    \mathcal A_{t+1}(\lambda) \subset \cup_{-8\leq u\leq 8} \mathcal A_{t+1}'(\lambda+u), \quad |\mathcal A_{t+1}(\lambda)|> \min_{-8\leq u\leq 8} \frac{1}{2}|\mathcal A_{t+1}'(\lambda+u)| \,.  
\end{align*}
Thus, by definition of $\Ecal_{4,t}$ and $\Ecal'_{4,t}$, we have
\begin{align*}
    \mathtt P\left( \Ecal_{4,t+1}^c \right) &\leq \mathtt P\left( ( \Ecal'_{4,t+1})^c \right) + o(1) \,,
\end{align*}
which shows that we only need to deal with $\Ecal'_{4,t+1}$. Again, we will split into two different cases.

\noindent{\bf Case 1}: $\alpha^{t}<2$. Note that in this case we have $2\mathtt n_{t}^- < \mathtt N_t^{-} < \mathtt N_t^+ < 0.5\mathtt n_t^{+}$. We start by proving the first inequality in \eqref{eq-At-prime-ub-lb}. Under the event in \eqref{eq-mathbf-Bt-small-prob}, for $0 \leq j \leq 2\mathtt n_{t+1}^-$ we have
\begin{align*}
    \forall k \in \cup_{\lambda>\mathtt N_t^+} \mathcal A_t(\lambda), \quad \eta^{(k)} < \max_{ 1 \leq i \leq n } \{ \eta^{(i)} \} - \varphi(\lambda) < \max_{ 1 \leq i \leq n } \{ \eta^{(i)} \} - j \,.
\end{align*}
Thus, we have
\begin{equation*}
    \mathcal A_{t+1}'(j) \subset \cup_{0 \leq \lambda \leq \mathtt N_t^+} \cup_{0 \leq j' \leq j} \mathcal A_{t+1}'(\lambda,j') \,.
\end{equation*}
For $\mathtt n_t^- < \lambda \leq \mathtt N_t^+$ and $1 \leq j \leq 2\mathtt n_{t+1}^-$ we have $\lambda \leq (\log n)^{1+\alpha^{t-1}-\alpha^t+o(1)} = (\log n)^{1-\Omega(1)}$. Thus, \eqref{eq-At-ub-lb} applies to $|\mathcal A_t(\lambda)|$, which yields that 
\begin{align*}
    &10\max\left\{ j\sqrt{ \tfrac{\log |\mathcal A_t(\lambda)|}{\mathbf H_t(\lambda)}},\ \exp(\sqrt{\log (\mathbf H_t(\lambda))}) \right\}   \\
    \leq\ & 10 \max\left\{ j\sqrt{\frac{\lambda(\log\log n)^{4\cdot 30^{t-1}} (\log n)^{1-\alpha^{t-1}}}{\mathbf M_t^\alpha}},\ \exp (\sqrt{\log (\mathbf M_t^{\alpha})} ) \right\} \\
    \leq\ & 10 \max\left\{ j (\log\log n)^{10\cdot 30^{t-1}} (\log n)^{1-\alpha^t}, \ \exp(\alpha^{(t+1)/2}\sqrt{\log\log n}) \right\}\\
    \leq\ & \exp( 0.2(40\alpha)^{t+1}\sqrt{\log\log n})\,,
\end{align*}
where the first inequality follows from $\Ecal_{1,t}$ and $\Ecal_{2,t}$, and the third inequality follows from $\Ecal_{2,t}$. In addition, define
\begin{align}{\label{eq-def-Core-t}}
    \mathtt{Core}_t = \left\{ \lambda \leq \mathtt N_t^{+} : |\mathcal A_t(\lambda)| \leq \exp(\exp((40\alpha)^t \sqrt{\log \log n})) \right\} \,.
\end{align}
For $0 \leq j \leq 2\mathtt n_{t+1}^-$, we have
\begin{align*}
    &\exp( \frac{1}{2}(\log\log n)\exp((40\alpha)^{t+1}\sqrt{\log\log n})) \\
    \geq\ &3\mathbf M_t \mathtt n_{t+1}^- \exp( (\log\log n) \exp( 0.5(40\alpha)^{t+1}\sqrt{\log \log n}))  \\
    \geq\ & 2\sum_{\substack{0 \leq \lambda \leq \mathtt n_{t}^-, 0 \leq j' \leq j}} |\mathcal A_t(\lambda)| + 2\sum_{\substack{\mathtt n_t^- < \lambda \leq \mathtt N_t^+ ,  \lambda \in \mathtt{Core}_t, 0 \leq j' \leq j }} |\mathcal A_t(\lambda)| \\
    &+ \sum_{\substack{\mathtt n_t^- < \lambda \leq \mathtt N_t^+ , \lambda \in \mathtt{Core}^c_t, 0 \leq j' \leq j }} \exp\bigg( 10 \max\left\{ j'\sqrt{ \frac{\log |\mathcal A_t(\lambda)|}{\mathbf H_t(\lambda)}}, \exp( \sqrt{\log(\mathbf H_t(\lambda))} ) \right\} \bigg) \,.
\end{align*}
Note that for all $\lambda \in \mathtt{Core}^c_t$, we have
\begin{align}
    \exp( \exp( 2\sqrt{\log(\mathbf H_t(\lambda))} ) ) &\overset{\Ecal_{1,t},\Ecal_{2,t}}{\leq} \exp(\exp(2\alpha^t\sqrt{\log \log n})) \nonumber \\
    &\leq |\mathcal A_t(\lambda)| \leq \exp( (\mathbf H_t(\lambda))^{\frac{2}{\alpha+1}} ) \,.  \label{eq-corelemma-conditions-satisfied}
\end{align}
Thus, all level sets $\mathcal A_t(\lambda)$ with $\lambda \in \mathtt{Core}_t^c$ satisfy the conditions of Lemma~\ref{lem-maximum-core-level-properties}. Thus, by Item (2) of Lemma~\ref{lem-maximum-core-level-properties}, we have
\begin{align}
    & \mathtt P \left( |\mathcal A_{t+1}'(j)| \geq \exp( \frac{1}{2}(\log \log n)\exp((40\alpha)^{t+1}\sqrt{\log\log n}) ); \eqref{eq-mathbf-Bt-small-prob} \mbox{ holds} \right) \nonumber \\
    \leq\ & \sum_{ \substack{ \lambda \in \mathtt{Core}_t^c , \mathtt n_t^- < \lambda \leq \mathbf M_t \\ 0 \leq j' \leq j} } \mathtt P\Big( |\mathcal A_{t+1}'(\lambda ,j')| \geq \exp( 10\max\bigg\{ j'\sqrt{ \frac{\log|\mathcal A_t(\lambda)|}{\mathbf H_t(\lambda)}}, \exp(\sqrt{\log(\mathbf H_t(\lambda))}) \bigg\} ) \Big) \nonumber\\
    \leq\ & 2(\log n)^{2\alpha^t+o(1)} \exp(-0.5\exp(\sqrt{\log \log n})) = o(1)\,, \label{eq-alpha-small-first-step-level-sets-decompose-1}
\end{align}
which proves the first inequality of \eqref{eq-At-prime-ub-lb}. 

Then, we prove the second inequality of \eqref{eq-At-prime-ub-lb}. Under the event in \eqref{eq-mathbf-Bt-small-prob}, for all $j<\varphi(\lambda)$ we have 
\begin{align*}
    \max_{k \in \mathcal A_t(\lambda)} \{ \eta^{(k)} \} \leq \max_{k \in [n]} \{ \eta^{(k)} \} - \varphi(\lambda) < \max_{k \in [n]} \{ \eta^{(k)} \} - j \,.
\end{align*}
Thus, we have
\begin{equation*}
    \mathcal A_{t+1}'(j) \subset \left( \cup_{0 \leq \lambda \leq \mathtt N_t^+} \cup_{0 \leq j' \leq j} \mathcal A_{t+1}'(\lambda,j') \right) \cup \cup_{\lambda \geq \mathtt N_t^+} \cup_{j\geq\varphi(\lambda)} \mathcal A_{t+1}'(\lambda,j')  \,.
\end{equation*}
For $\mathtt n_{t}^- < \lambda \leq \mathtt N_t^+$ and $j \geq 0.5\mathtt n_{t+1}^-$, \eqref{eq-At-ub-lb} applies to $|\mathcal A_t(\lambda)|$ and therefore
\begin{align*}
    & 10\max\left\{ j\sqrt{ \frac{\log|\mathcal A_t(\lambda)|}{\mathbf H_t(\lambda)}}, \exp( \sqrt{\log(\mathbf H_t(\lambda))} ) \right\}  \\
    \leq\ & 10\max\left\{ j\sqrt{\frac{\lambda(\log\log n)^{30^{t-1}} (\log n)^{1-\alpha^{t-1}}}{\mathbf M_t^\alpha}}, \exp(\sqrt{\log(\mathbf H_t(\lambda))} ) \right\}  \\
    \leq\ & 10\max\left\{ (\log\log n)^{7\cdot 30^{t-1}}j(\log n)^{1-\alpha^t}, \exp(\alpha^{(t+1)/2}\sqrt{\log\log n}) \right\}  \\
    \leq\ & 10(\log \log n)^{7\cdot 30^{t-1}}j(\log n)^{1-\alpha^t} +10\exp(\alpha^{t+1}\sqrt{\log \log n})\,,
\end{align*}
where the second inequality holds by $\lambda \leq \mathtt N_t^+$. In addition, for $\lambda>\mathtt N_t^+$ such that $\varphi(\lambda) \leq j$ we have 
\begin{align*}
    \log|\mathcal A_t(\lambda)| \leq \lambda(\log \log n)^{30^{t-1}}(\log n)^{1-\alpha^{t-1}} < j(\log \log n)^{4\cdot 30^{t-1}}(\log n)^{1-\alpha^t}\,.
\end{align*}
Therefore, for $j \geq 0.5\mathtt n_{t+1}^-$ we have (recall \eqref{eq-def-mathtt-n-t-plus-minus})
\begin{align*}
    &\exp(\frac{1}{2} \cdot j(\log n)^{1-\alpha^t}(\log \log n)^{30^t}) \\
    \geq\ & 4\mathbf M_t \mathtt n_t^- \left( \exp( j(\log n)^{1-\alpha^t}(\log\log n)^{10\cdot 30^{t-1}} ) + \exp((\log \log n)\exp((40\alpha)^t\sqrt{\log \log n})) \right) \nonumber \\
    \geq\ & 2\sum_{\substack{0 \leq \lambda \leq \mathtt n_t^- \\ 0 \leq j' \leq j}} |\mathcal A_t(\lambda)| + 2\sum_{\substack{\mathtt n_t^- < \lambda \leq \mathtt N_t^+, \lambda \in \mathtt{Core}_t \\ 0 \leq j' \leq j}} |\mathcal A_t(\lambda)| + 2\sum_{\substack{\lambda > \mathtt N_t^+\\ \varphi(\lambda) \leq j}}|\mathcal A_t(\lambda)|  \\
    &+ \sum_{\substack{\mathtt n_t^- < \lambda \leq \mathtt N_t^+, \lambda \in \mathtt{Core}^c_t \\ 0 \leq j' \leq j}} \exp\bigg( 10 \max\left\{ j'\sqrt{\frac{\log|\mathcal A_t(\lambda)|}{\mathbf H_t(\lambda)}}, \exp(\sqrt{\log(\mathbf H_t(\lambda))} ) \right\} \bigg) \,. 
\end{align*}
Thus, by Item (2) of Lemma~\ref{lem-maximum-core-level-properties}, we have
\begin{align}
    & \mathtt P\left( |\mathcal A_{t+1}'(j)| \geq \exp(\frac{1}{2} \cdot j(\log n)^{1-\alpha^t}(\log \log n)^{30^t}); \eqref{eq-mathbf-Bt-small-prob} \mbox{ holds} \right) \nonumber \\
    \leq\ & \sum_{ \substack{ \lambda \in \mathtt{Core}_t^c , \mathtt n_t^- < \lambda \leq \mathbf M_t \\ 0 \leq j' \leq j} } \mathtt P\Big( |\mathcal A_{t+1}'(\lambda ,j')| \geq \exp \Big( 10 \max\left\{ j'\sqrt{\frac{\log|\mathcal A_t(\lambda)|}{\mathbf H_t(\lambda)}}, \exp(\sqrt{\log(\mathbf H_t(\lambda))} ) \right\} \Big) \Big) \nonumber\\
    \leq\ & 2(\log n)^{2\alpha^t+o(1)} \exp(-0.5\exp(\sqrt{\log \log n})) =o(1)\,,\label{eq-alpha-small-first-step-level-sets-decompose-2}
\end{align}
which proves the second inequality of \eqref{eq-At-prime-ub-lb}. 

Finally we prove the last inequality of \eqref{eq-At-prime-ub-lb}. Note that for $0.5\mathtt n_{t+1}^- < j < 2\mathtt n^+_{t+1}$ and $\mathtt N_t^- \leq \lambda \leq \mathtt N_t^+$, \eqref{eq-At-ub-lb} applies to $|\mathcal A_t(\lambda)|$ and therefore
\begin{align}
2S(\mathbf H_t(\lambda), |\mathcal A_t(\lambda)|) &\leq \exp\big(\sqrt{\log \mathbf H_t(\lambda)}\big) \cdot \sqrt{\tfrac{\mathbf H_t(\lambda)}{\log |\mathcal A_t(\lambda)|}} \leq \exp(\alpha^{\tfrac{t+1}{2}} \sqrt{\log \log n} ) \cdot \sqrt{\tfrac{(\log n)^{\alpha^t}}{\lambda (\log n)^{1-\alpha^{t-1}} }} \nonumber \\
&\leq \exp(\tfrac{1}{2}\alpha^{t+1} \sqrt{\log \log n}) \cdot (\log n)^{\alpha^t - 1} < j\,, \label{eq-lem-core-lb-satisfied-1}\\
j \leq \tfrac{\log n}{(\log \log n)^{30^t}} &\leq \tfrac{1}{10}\sqrt{ (\log \log n)^{-10\cdot 30^{t-1}} \lambda (\log n)^{\alpha^t + 1 -\alpha^{t-1}} } \nonumber \\
&\leq \tfrac{1}{10}\sqrt{\mathbf H_t(\lambda) \cdot \log |\mathcal A_t(\lambda)|}\,.\label{eq-lem-core-lb-satisfied-2}
\end{align}
Combining \eqref{eq-lem-core-lb-satisfied-1} and \eqref{eq-lem-core-lb-satisfied-2} with \eqref{eq-corelemma-conditions-satisfied}, the conditions of Item (3) of Lemma~\ref{lem-maximum-core-level-properties} are satisfied and therefore
\begin{align}
&\mathtt P\Big( \mbox{ exists }0.5\mathtt n_{t+1}^- < j < 2\mathtt n_{t+1}' \,,|\mathcal A'_{t+1}(j)| < \exp\big(\tfrac{2(\log n)^{1-\alpha^{t}}j}{(\log \log n)^{30^{t}}}\big)  \Big) \nonumber \\
\leq\ & \eqref{eq-nextmax-in-mathtt-Nts} + \sum_{0.5\mathtt n_{t+1}^- < j < 2\mathtt n_{t+1}' } \sum_{\lambda \in [\mathtt N_t^- , \mathtt N_t^+]}\mathtt P\Big( |\mathcal A'_{t+1}(\lambda, j)| < \exp\big(\tfrac{1.5(\log n)^{1-\alpha^{t}}j}{(\log \log n)^{30^{t}}}\big) \Big) \nonumber \\
\leq\ & o(1) + \sum_{0.5\mathtt n_{t+1}^- < j < 2\mathtt n_{t+1}' }  \sum_{\lambda \in [\mathtt N_t^- , \mathtt N_t^+]}\mathtt P\Bigg( |\mathcal A'_{t+1}(\lambda, j)| < \exp\Big(\tfrac{j}{4} \sqrt{\tfrac{(\log n)^{2-\alpha^t}}{(\log \log n)^{10\cdot 30^{t-1}}\mathbf H_t(\lambda)}}\Big)\Bigg) \nonumber\\
\leq\ &  o(1) + \sum_{0.5\mathtt n_{t+1}^- < j < 2\mathtt n_{t+1}' } \sum_{\lambda \in [\mathtt N_t^- , \mathtt N_t^+]}\mathtt P\Bigg( |\mathcal A'_{t+1}(\lambda, j)| < \exp\Big(\tfrac{j}{4}\sqrt{\tfrac{\log |\mathcal A_t(\lambda)|}{\mathbf H_t(\lambda)}}\Big)\Bigg)\nonumber \\
\leq\ &  o(1) + 2(\log n)^{2\alpha^t+o(1)} \exp(-0.5\exp(\sqrt{\log \log n})) = o(1)\,,\label{eq-alpha-small-first-step-level-sets-decompose-4}
\end{align}
which shows the last inequality of \eqref{eq-At-prime-ub-lb}. 
Combining \eqref{eq-alpha-small-first-step-level-sets-decompose-1}, \eqref{eq-alpha-small-first-step-level-sets-decompose-2} and \eqref{eq-alpha-small-first-step-level-sets-decompose-4} implies that $\mathtt P( (\Ecal_{4,t+1}')^c)=o(1)$.

\noindent{\bf Case 2}: $\alpha^t=2$. In this case, define
\begin{align*}
    \mathtt{R}_j = [0,\mathtt n_{t}^-] \cup \left\{ 0 < \lambda \leq \mathbf M_t: \frac{\lambda(\log n)^{\alpha^t-\alpha^{t-1}}}{(\log\log n)^{3\cdot 30^{t-1}}} \leq j \right\} \mbox{ for } j \geq 0 \,.
\end{align*}
In the event $\Ecal_{2,t}$, we have $|\mathtt R_j| \leq \mathbf M_t = (\log n)^{\Omega(1)}$. In addition, on the event in \eqref{eq-mathbf-Bt-small-prob}, we have $\mathcal A_t(\lambda) \cap \mathcal A_{t+1}'(j) = \emptyset \mbox{ for all }\lambda \not\in \mathtt R_j$. Thus, in this case we have
\begin{align*}
    & |\mathcal A'_{t+1}(j)|<\sum_{\lambda \in \mathtt R_j} |\mathcal A_t(\lambda)| \\
    \leq\ & \mathtt n_t^-\cdot\left( \exp((\log\log n) \exp((40\alpha)^{t}\sqrt{\log \log n})) + \exp( (\log\log n)^{30^{t-1}}(\log n)^{1-\alpha^{t-1}}\mathtt n_t^-) \right)\\
    +\ & \sum_{\lambda\in \mathtt R_j \setminus [0,\mathtt N_t^+]} \left( \exp((\log\log n) \exp((40\alpha)^{t}\sqrt{\log \log n})) + \exp( (\log\log n)^{30^{t-1}}(\log n)^{1-\alpha^{t-1}}\lambda) \right) \\
    \leq\ & \exp(\tfrac{1}{4}(\log \log n)\exp((40\alpha)^{t+1}\sqrt{\log \log n}))\\
    + \ & \exp(\tfrac{1}{2}(\log \log n)^{30^{t}}(\log n)^{1-\alpha^{t}}j)\big(|\mathtt R_j|\cdot\mathbf 1_{\{j \leq (\log n)^{1-\alpha^t}\}} + \mathbf 1_{\{j > (\log n)^{1-\alpha^t}\}}\big)\\
    \leq\ & \exp(\tfrac{1}{2}(\log \log n)\exp((40\alpha)^{t+1}\sqrt{\log \log n}))
    + \exp(\tfrac{1}{2}(\log \log n)^{30^{t}}(\log n)^{1-\alpha^{t}}j)\,,
\end{align*}
where the second inequality holds by the definition of $\Ecal_{4,t}$.

\noindent{\bf Case 3}: $\alpha^t > 2$. In this case, define
\begin{align*}
    \mathtt{R}_j = \left\{ 0 < \lambda \leq \mathbf M_t: \frac{\lambda(\log n)^{\alpha^t-\alpha^{t-1}}}{(\log\log n)^{4\cdot 30^{t-1}}} \leq j \right\} \mbox{ for } j \geq 0 \,.
\end{align*}
On $\Ecal_{2,t}$, we have $|\mathtt R_j| \leq \mathbf M_t = (\log n)^{\Omega(1)}$. In addition, on the event in \eqref{eq-large-gap-when-t-large}, we have $\mathcal A_t(\lambda) \cap \mathcal A_{t+1}'(j) = \emptyset \mbox{ for all }\lambda \not\in \mathtt R_j$.
Thus, we have
\begin{align*}
    & |\mathcal A'_{t+1}(j)|<\sum_{\lambda \in \mathtt R_j} |\mathcal A_t(\lambda)| \\
    \leq\ & \sum_{\lambda\in \mathtt R_j} \left( \exp((\log\log n) \exp((40\alpha)^{t}\sqrt{\log \log n})) + \exp( (\log\log n)^{30^{t-1}}(\log n)^{1-\alpha^{t-1}}\lambda) \right) \\
    \leq\ & \exp(\tfrac{1}{4}(\log \log n)\exp((40\alpha)^{t+1}\sqrt{\log \log n})) \\
    +\ & \exp(\tfrac{1}{2}(\log \log n)^{30^{t}}(\log n)^{1-\alpha^{t}}j) \cdot \big(|\mathtt R_j|\cdot\mathbf 1_{\{j \leq (\log n)^{1-\alpha^t}\}} + \mathbf 1_{\{j > (\log n)^{1-\alpha^t}\}}\big) \\
    \leq\ & \exp(\tfrac{1}{2}(\log \log n)\exp((40\alpha)^{t+1}\sqrt{\log \log n}))
    + \exp(\tfrac{1}{2}(\log \log n)^{30^{t}}(\log n)^{1-\alpha^{t}}j)\,,
\end{align*}
where the second inequality holds both when $\alpha^{t-1} < 2$ and $\alpha^{t-1} \geq 2$ by the definition of $\Ecal_{4,t}$.

Combining the three cases above, we have proved \eqref{eq-prob-E-4-fail}.

\subsubsection{Proof of (\ref{eq-prob-E-5-fail})}

Again, we will split into three cases:

\noindent{\bf Case 1}: $\alpha^t < 2$. Then we have $\mathtt n_t^- \leq \mathtt N_t^- < \mathtt N_t^+$. Define the set 
\begin{align}
    \mathcal A_{\text{core}, t} = \cup_{\lambda \leq \mathtt N_t^+} \mathcal A_t(\lambda) = \mathcal A_t[0,\mathtt N_t^+] \,. \label{eq-def-A-core-set}
\end{align}
In addition, define
\begin{align*}
    \mathcal H_t = \Big\{ (\log \mathbf H_t(0) )^{-30^t} \leq \frac{ \max_{k \in \mathcal A_{\text{core}, t} } \{ \eta^{(k)} \} -\max_{k \in \mathcal A_{\text{core}, t}}^{(2)} \eta^{(k)} }{ \sqrt{ \mathbf H_t(0)/\log|A_{\text{core},t}| } } \leq (\log \mathbf H_t(0) )^{30^t} \Big\} \,.
\end{align*}
Then, by the fact that $\Ecal_{4,t}$ holds, we have
\begin{align}
    \log |\mathcal A_{\text{core}, t}| &\leq \log\left( \mathtt N_t^+ \exp( (\log\log n)^{11\cdot30^{t-1}}(\log n)^{1-\alpha^{t-1} +1+\alpha^{t-1}-\alpha^t} ) \right) \nonumber\\
    &\leq 2(\log\log n)^{11\cdot30^{t-1}}(\log n)^{1-\alpha^{t-1}+1+\alpha^{t-1}-\alpha^t}\nonumber \\
    &\leq (\log \log n)^{23\cdot30^{t-1}} \log\left( \exp((\log \log n)^{-11\cdot 30^{t-1}} (\log n)^{1-\alpha^{t-1}+1+\alpha^{t-1}-\alpha^t} ) \right)\nonumber\\
    &\leq (\log \mathbf H_t(0))^{30^{t}-1} \log|\mathcal A_t(\mathtt N_t^-)| \,, \label{eq-Acore-t-ub}
\end{align}
and
\begin{align}
    \frac{0.01\mathbf H_t(0)}{\log \mathbf H_t(0)} &> \log n\geq \log |\mathcal A_{\text{core}, t}| \geq \log|\mathcal A_t(\mathtt N_t^-)| \nonumber \\
    &\geq (\log \log n)^{-11\cdot 30^{t-1}} (\log n)^{2-\alpha^t} > \log \log \mathbf H_t(0)\,,\label{eq-Acore-t-lb}
\end{align}
Therefore, we have $\log|\mathcal A_{\text{core},t}| = (\log n)^{2-\alpha^t+o(1)}$ and thus $\sqrt{\tfrac{\mathbf H_t(\mathtt N_t^+)}{\log |A_{\text{core}, t}|} } \cdot (\log \mathbf H_t(\mathtt N_t^+) )^{30^t} = (\log n)^{\alpha^t-1+o(1)} < 0.01\log n$, which gives
\begin{align}
    &\mathtt P\Big( \mathcal H_t \cap \Big\{ \max_{k \in [n]} \{ \eta^{(k)} \} \neq \max_{k \in \mathcal A_{\text{core},t}} \{ \eta^{(k)} \} \mbox{ or } \max^{(2)}_{k \in [n]} \{ \eta^{(k)} \} \neq \max_{k \in \mathcal A_{\text{core},t}}^{(2)} \{ \eta^{(k)} \} \Big\} \Big)\nonumber\\
    \leq\ &\mathtt P\Big( \mathcal H_t \cap \Big\{ \max_{k \in \mathcal A_{\text{core},t}^c} \eta^{(k)} > \max_{k \in \mathcal A_{\text{core},t}} \{ \eta^{(k)} \} - \log n \Big\} \Big)=  \eqref{eq-mathbf-Bt-small-prob} = o(1)\,.\label{eq-bound-probability-outside-core-1}
\end{align}
Next, define
\begin{align*}
    \Ecal_{\mathsf{boundH}} = \Big\{ & \exists \lambda \in [0, \mathtt N_t^+], \max_{j \in \mathcal A_t(\lambda)} \{ \eta^{(j)} \} = \max_{k \in \mathcal A_{\text{core}, t}} \{ \eta^{(k)} \} \,, \\
    & \mathbf H_t(\lambda) \leq \mathbf H_t(\mathtt N_t^-) + (\log \mathbf H_t(0))^{\frac{30^t+1}{4}} \sqrt{\mathbf H_t(\mathtt N_t^-)} \Big\} \,.
\end{align*}
Then, we have
\begin{align}
    \mathtt P(\Ecal_{\mathsf{boundH}}) &\geq \mathtt P\Big( \max_{ k \in \mathcal A_{\text{core},t} } \{ \eta^{(k)} \} = \max_{ k \in \mathcal A_t[\mathtt N_t^-,\mathtt N_t^+] } \{ \eta^{(k)} \} \Big) \nonumber \\
    &\geq 1 - \eqref{eq-mathbf-Bt-small-prob} - \eqref{eq-toolarge-layers-lead-small-prob} = 1-o(1)\,,\label{eq-prob-Acore-have-max}
\end{align}
Combining \eqref{eq-Acore-t-ub}, \eqref{eq-Acore-t-lb} and \eqref{eq-prob-Acore-have-max}, given the fact that $\cap_{1\leq h \leq 6 , 1 \leq s \leq t}\Ecal_{h,s}$ holds, the conditions of Items (1) and (2) of Lemma~\ref{lem-Poisson-maximum-1st-minus-2nd} are satisfied with 
\begin{align*}
    (\lambda, m, \lambda_*, c) &= (\mathbf H_t(0),|\mathcal A_{\text{core}, t}|,\mathbf H_t(\mathtt N_t^-), 30^t-1)\,,
\end{align*}
Also, note that under the event in \eqref{eq-couple-error-at-most-4} we have
\begin{align*}
\Big|(\max_{k \in [n]} \{ \mathbf X_t^{(k)} \} - \max_{k \in [n]}^{(2)} \{ \mathbf X_t^{(k)} \}) - (\max_{k \in [n] } \{ \eta^{(k)} \} - \max_{k \in [n]}^{(2)} \{\eta^{(k)}\})\Big| \leq 8\,.
\end{align*}
Therefore, by Lemma~\ref{lem-Poisson-maximum-1st-minus-2nd}, \eqref{eq-bound-probability-outside-core-1}, \eqref{eq-couple-error-at-most-4}, and the fact that $\log \mathbf H_t(0) \leq \log ((\log n)^{\alpha^t+1}) \leq 3 \log \log n$ we have 
\begin{align*}
    &\mathtt P(\Ecal^c_{5,t+1}) 
    \leq \mathtt P\Big( \mathcal H_t \cap \left\{ \max_{k \in [n]} \{ \eta_{(k)} \} \neq \max_{k \in \mathcal A_{\text{core},t}} \{\eta^{(k)}\} \mbox{ or } \max_{k \in [n]}^{(2)} \{ \eta^{(k)} \} \neq \max_{k \in \mathcal A_{\text{core},t}} \{ \eta^{(k)} \} \right\} \Big)  \nonumber\\
    + \ & \mathtt P(\mathcal H_t^c)  + \mathtt P\left( |\eta^{(i)} - \mathbf X_{t+1}^{(i)}| > 4 \mbox{ for some } i \in [n] \right) \leq \eqref{eq-bound-probability-outside-core-1} + o(1) = o(1)\,.
\end{align*}

\noindent{\bf Case 2}: $\alpha^t = 2$. Then we have $\mathtt n_t^- > \mathtt N_t^-$. Define the set 
\begin{align}
    \mathcal A_{\text{core}, t} = \cup_{\lambda \leq \mathtt n_t^-} \mathcal A_t(\lambda) = \mathcal A_t[0,\mathtt n_t^-] \,. \label{eq-def-A-core-set-critical}
\end{align}
In addition, define
\begin{align*}
    \mathcal H_{\mathsf{lb};t} &= \Big\{ \tfrac{ \max_{k \in \mathcal A_{\text{core}, t} } \{ \eta^{(k)} \} - \max_{k \in\mathcal A_{\text{core}, t}}^{(2)} \eta^{(k)} }{ \sqrt{ \mathbf H_t(0)/\log|A_{\text{core},t}| } } \geq (\log \mathbf H_t(0) )^{-30^t} \Big\} \,,\\
    \mathcal H_{\mathsf{ub};t} &= \Big\{ \tfrac{ \max_{k \in \mathcal A_{\text{core}, t} } \{ \eta^{(k)} \} - \max_{k \in\mathcal A_{\text{core}, t}}^{(2)} \eta^{(k)} }{\log n } \leq (2\log \log n )^{3\cdot30^t+\sqrt{\log \log n}} \Big\} \,,
\end{align*}
where $\log|A_{\text{core},t}| \leq \log(\mathtt n_t^-)+(\log \log n)\exp((40\alpha)^t\sqrt{\log \log n}) \leq (\log \log n)^{0.1\sqrt{\log \log n}}$ by the definition of $\Ecal_{4,t}$. Then, we have
\begin{align}
\mathtt P(\Ecal^c_{5,t+1})\leq \mathtt P\left( \mathcal H_{\mathsf{lb};t} \cap \mathcal H_{\mathsf{ub};t} \cap \Ecal^c_{5,t+1} \right)+ \mathtt P(\mathcal H_{\mathsf{lb};t}^c) + \mathtt P(\mathcal H_{\mathsf{ub};t}^c)\,.\label{eq-Ecal-5-t+1-critical-reduction}
\end{align}
Next, note that we have
\begin{align}
    \frac{0.01\mathbf H_t(0)}{\log \mathbf H_t(0)} &> \log n,\label{eq-Acore-t-lb-critical}
\end{align}
thus, given the fact that $\cap_{1\leq h \leq 6 , 1 \leq s \leq t}\Ecal_{h,s}$ holds, the conditions of Item (1) of Lemma~\ref{lem-Poisson-maximum-1st-minus-2nd} are satisfied with $(\lambda, m) = (\mathbf H_t(0),|\mathcal A_{\text{core},t}|)$.
By Item (1) of Lemma~\ref{lem-Poisson-maximum-1st-minus-2nd} we have 
\begin{align}
    \mathtt P(\mathcal H_{\mathsf{lb};t}^c) = o(1)\,.\label{eq-critical-lb-c-bound}
\end{align} 
Next, to give an upper bound of $\mathtt P(\mathcal H_{\mathsf{ub};t}^c)$, we consider two cases: $\max_{k \in \mathcal A_{t}(\mathtt n_t^-)} \eta^{(k)} \neq \max_{k \in \mathcal A_{\text{core},t}} \eta^{(k)}$ and $\max_{k \in \mathcal A_{t}(\mathtt n_t^-)} \eta^{(k)} = \max_{k \in \mathcal A_{\text{core},t}} \eta^{(k)}$. For the case where $\max_{k \in \mathcal A_{t}(\mathtt n_t^-)} \eta^{(k)} \neq \max_{k \in \mathcal A_{\text{core},t}} \eta^{(k)}$, we define $\gamma_t = \log\log n \exp((40\alpha)^t\sqrt{\log \log n})$. Note that for any $0 \leq \lambda \leq \mathtt n_t^-$ we have $\gamma_t \geq \log |\mathcal A_t(\lambda)|$ and
\begin{align}
&\mathbf H_t(\lambda) - \mathbf H_t(\mathtt n_t^-) \leq \tfrac{n(\mathbf M_t^\alpha - (\mathbf M_t - \mathtt n_t^-)^\alpha)}{\mathbf Z_t} \leq \alpha\mathbf M_t^{\alpha - 1}\mathtt n_t^- \nonumber \\
\leq\ & 2^{\alpha^{t-1}}\alpha(\log n) \exp( 0.1(40\alpha)^{t}\sqrt{\log\log n} ) \nonumber \\
\leq \ & (\log n)\cdot(2\log \log n)^{3\cdot 30^t + \sqrt{\log \log n}}-10 \sqrt{\gamma_t\cdot\mathbf H_t(0)}\nonumber \\
\leq\ & (\log n)\cdot(2\log \log n)^{3\cdot 30^t + \sqrt{\log \log n}} - 3\sqrt{\gamma_t \cdot \mathbf H_t(\lambda)} - \sqrt{\log \log n \cdot \mathbf H_t(\mathtt n_t^-)}\,,
\end{align}
note that $|\mathcal A_t(\mathtt n_t^-)| \geq 1$ by definition of $\Ecal_{4,t}$. Therefore, we have
\begin{align}
&\mathtt P(\mathcal H_{\mathsf{ub};t}^c\cap \{\max_{k \in \mathcal A_{t}(\mathtt n_t^-)} \eta^{(k)} \neq \max_{k \in \mathcal A_{\text{core},t}} \eta^{(k)}\} ) \nonumber \\
\leq\ & \sum_{0 \leq \lambda < \mathtt n_t^-}\mathtt P\Big(\max_{k \in \mathcal A_{t}(\lambda)} \eta^{(k)} \geq \mathbf H_t(\lambda) + 3\sqrt{\gamma_t \cdot \mathbf H_t(\lambda)} \Big)\nonumber \\
+\ & \mathtt P\Big(\max_{k \in \mathcal A_{t}(\mathtt n_t^-)} \eta^{(k)} \leq \mathbf H_t(\mathtt n_t^-) - \sqrt{\log \log n \cdot \mathbf H_t(\mathtt n_t^-)} \Big) \nonumber \\
\leq\ & \sum_{0 \leq \lambda < \mathtt n_t^-} \mathtt P\Big( \mathsf{PM}(\exp(\gamma_t) , \mathbf H_t(\lambda))\geq \mathbf H_t(\lambda) + 3\sqrt{\gamma_t \cdot \mathbf H_t(\lambda)}\Big) \nonumber \\
+\ & \mathtt P\Big( \mathsf{Pois}(\mathbf H_t(\mathtt n_t^-)) \leq \mathbf H_t(\mathtt n_t^-) - \sqrt{\log \log n \cdot \mathbf H_t(\mathtt n_t^-)} \Big) \nonumber \\
\leq\ & \mathtt n_t^- \cdot \exp(-3\gamma_t) + \exp(-\tfrac{\log \log n}{3}) = o(1)\,.\label{eq-critical-ub-c-bound-1}
\end{align} 
Next, for the case where $\max_{k \in \mathcal A_{t}(\mathtt n_t^-)} \eta^{(k)} = \max_{k \in \mathcal A_{\text{core},t}} \eta^{(k)}$, since $|\mathcal A_t(\mathtt n_t^-)| \geq 2$ (by the definition of $\Ecal_{4,t}$),  
\begin{align}
&\mathtt P(\mathcal H_{\mathsf{ub};t}^c\cap \{\max_{k \in \mathcal A_{t}(\mathtt n_t^-)} \eta^{(k)} = \max_{k \in \mathcal A_{\text{core},t}} \eta^{(k)}\} ) \nonumber \\
\leq\ &\mathtt P(\{\max_{k \in \mathcal A_{t}(\mathtt n_t^-)} \eta^{(k)}-\min_{k \in \mathcal A_{t}(\mathtt n_t^-)} \eta^{(k)} \geq (\log n)\cdot (2\log \log n)^{3\cdot30^t+\sqrt{\log \log n}}\})\nonumber \\
\leq\ & 1-\mathtt P\Big(|\mathsf{Pois}(\mathbf H_t(\mathtt n_t^-)) - \mathbf H_t(\mathtt n_t^-)| \leq (\log \log n)^{3\cdot30^t+\sqrt{\log \log n}}\sqrt{\mathbf H_t(\mathtt n_t^-)}\Big)^{|\mathcal A_t(\mathtt n_t^-)|} \nonumber \\
\leq\ & 1 - \Big(1-\exp\big(-\tfrac{(\log \log n)^{6\cdot30^t+2\sqrt{\log \log n}}}{3}\big)\Big)^{\exp((\log \log n)^{\sqrt{\log \log n}})} = o(1)\,,\label{eq-critical-ub-c-bound-2}
\end{align}
Therefore, combining \eqref{eq-critical-ub-c-bound-1} and \eqref{eq-critical-ub-c-bound-2} we have
\begin{align}
\mathtt P(\mathcal H_{\mathsf{ub};t}^c)= \eqref{eq-critical-ub-c-bound-1} + \eqref{eq-critical-ub-c-bound-2} = o(1)\,.\label{eq-critical-ub-c-bound}
\end{align}
Next, note that under $\mathcal H_{\mathsf{lb};t}$ we have 
\begin{align}
    &\mathtt P\left( \mathcal H_{\mathsf{ub};t} \cap \mathcal H_{\mathsf{lb};t}\cap\Ecal_{5,t+1}^c \right)\nonumber\\
    \leq \ & \mathtt P\left(\left\{ \max_{k \in [n]} \{ \eta^{(k)} \}  \neq \max_{k \in \mathcal A_{\text{core},t}} \{ \eta^{(k)} \}  \right\}\right) + \mathtt P\Big(\Big\{\max_{k \in [n]} \{ \eta^{(k)} \}  = \max_{k \in \mathcal A_{\text{core},t}} \{ \eta^{(k)} \} \Big\}\nonumber \\
    \cap\ &\Big\{ \max_{k \in \mathcal A_{\text{core},t}^c} \eta^{(k)} > \max_{k \in \mathcal A_{\text{core},t}} \{ \eta^{(k)} \} - \sqrt{\tfrac{\mathbf H_t(0)}{\log |\mathcal A_{\text{core},t}|}}\cdot(\log \mathbf H_t(0))^{-30^t} \Big\} \Big)\nonumber \\
    \leq\ & \eqref{eq-mathbf-Bt-small-prob} = o(1)\,.\label{eq-bound-probability-outside-core-2}
\end{align}
Also, note that under the event in \eqref{eq-couple-error-at-most-4} we have
\begin{align*}
\left|(\max_{k \in [n]} \{ \mathbf X_t^{(k)} \} - \max_{k \in [n]}^{(2)} \{ \mathbf X_t^{(k)} \}) - (\max_{k \in [n] } \{ \eta^{(k)} \} - \max_{k \in [n]}^{(2)} \{\eta^{(k)}\})\right| \leq 8\,.
\end{align*}
Therefore, combining \eqref{eq-critical-lb-c-bound}, \eqref{eq-critical-ub-c-bound}, \eqref{eq-couple-error-at-most-4} and the fact that $\log \mathbf H_t(0) \leq \log ((\log n)^{\alpha^t+1}) \leq 3 \log \log n$ we have $\eqref{eq-Ecal-5-t+1-critical-reduction}\leq \eqref{eq-bound-probability-outside-core-2} + o(1)  = o(1)$.

\noindent\textbf{Case 3:} $\alpha^t > 2$. Suppose that $\mathbf X_t^{(l'_t)}$ is the second largest in $\mathbf X_t$ and $l'_t \in \mathcal A_t(\lambda')$ for some $\lambda' > 0$. Then, we have 
\begin{align*}
(\log n)^{\alpha^{t-1}-1} (2\log\log n)^{-\cdot 30^{t}-\sqrt{\log \log n}} \leq \lambda' \leq (\log n)^{\alpha^{t-1}-1} (2\log\log n)^{3\cdot 30^{t} + \sqrt{\log \log n}}
\end{align*}
by definition of $\Ecal_{5,t}$. Therefore, we have
\begin{align}
    \mathtt P\Big( \max_{k \in [n]} \{ \eta^{(k)} \} - \max_{k \in [n]}^{(2)} \{ \eta^{(k)} \} \leq \frac{(\log n)^{\alpha^t-1}}{(\log \log n)^{30^{t+1}+\sqrt{\log\log n}}} \Big) \leq \eqref{eq-large-gap-when-t-large} = o(1) \,. \label{eq-case2-Ecal-4,t+1-1}
\end{align}
Next, we have
\begin{align*}
    &\mathbf H_t(0) - \mathbf H_t(\lambda') \leq \frac{n\mathbf M_t^{\alpha - 1}}{\mathbf Z_t}\cdot (\log n)^{\alpha^{t-1}-1} (2\log\log n)^{3\cdot 30^{t} + \sqrt{\log \log n}} \\
    \leq\ & (\log n)^{\alpha^t - 1}(\log\log n)^{4\cdot 30^{t}+\sqrt{\log \log n}} \\
    \leq\ & 2(\log n)^{\alpha^t-1}(\log\log n)^{4\cdot 30^{t}+\sqrt{\log \log n}} - (\log \log n)^{10}\left( \sqrt{\mathbf H_t(0)} + \sqrt{\mathbf H_t(\lambda')} \right) \,.
\end{align*}
Therefore, by the Poisson Chernoff bound we have
\begin{align}
    &\mathtt P\left(\max_{k \in [n]} \{ \eta^{(k)} \} - \max_{k \in [n]}^{(2)} \{ \eta^{(k)} \}\geq (\log n)^{\alpha^t - 1} (\log \log n)^{30^{t+1} + \sqrt{\log \log n}} \right) \nonumber\\
    \leq\ & \eqref{eq-large-gap-when-t-large} + \mathtt P\left( \eta^{(l_t)} \geq \mathbf H_t(0) + (\log \log n)^{10}\sqrt{\mathbf H_t(0)} \right)\nonumber\\
    +\ & \mathtt P\left( \eta^{(l'_t)} \leq \mathbf H_t(\lambda') - (\log \log n)^{10}\sqrt{\mathbf H_t(\lambda')} \right) = o(1) \,. \label{eq-case2-Ecal-4,t+1-2}
\end{align}
Also, note that under the event in \eqref{eq-couple-error-at-most-4} we have
\begin{align*}
\left|(\max_{k \in [n]} \{ \mathbf X_t^{(k)} \} - \max_{k \in [n]}^{(2)} \{ \mathbf X_t^{(k)} \}) - (\max_{k \in [n] } \{ \eta^{(k)} \} - \max_{k \in [n]}^{(2)} \{\eta^{(k)}\})\right| \leq 8\,.
\end{align*}
Combining \eqref{eq-case2-Ecal-4,t+1-1}, \eqref{eq-case2-Ecal-4,t+1-2} and \eqref{eq-couple-error-at-most-4} yields $\mathtt P(\Ecal_{5,t+1}^c) = o(1)$.

\subsubsection{Proof of (\ref{eq-prob-E-6-fail})}

On the event $\mathcal E_{5,t+1}$, we have $\mathcal L(\mathbf X_{t+1})$ is a singleton. The rest of the arguments on $\mathcal L(\mathbf X_{t+1})$ stated by the event $\mathcal E_{5,t+1}$ are valid only when $\alpha^t \neq 2$. Therefore, we may split into two cases:

\noindent{\bf Case 1:} $\alpha^t<2$. Since we have already proved that with probability $1-o(1)$, the maximum among $\{\eta_{t+1}^{(i)}\}_{1 \leq i \leq n}$ comes only from the level sets $\mathcal A_t[\mathtt N_t^-,\mathtt N_t^+]$. Also, by the definition of $\Ecal_{5,s}$ for $s \in [t]$, we have $|\cup_{s \in [t]}\mathcal L(\mathbf X_s)| \leq t-1 + \exp((\log \log n)^3)$. As a result, since all the elements in the same $\mathcal A_t(\lambda)$ are symmetric given $\lambda$, $\mathcal L(\mathbf X_{t+1})$ is a singleton that differs from any element in $\cup_{s \in [t]}\mathcal L(\mathbf X_s)$ with probability at least
\begin{align*}
    1- \max_{ \lambda\in [\mathtt N_t^-,\mathtt N_t^+] } \Big\{ \frac{t-1 + \exp((\log \log n)^3)}{|\mathcal A_t(\lambda)|} \Big\} = 1-o(1) \,,
\end{align*}
where the last equality holds by the fact that $[\mathtt N_t^- , \mathtt N_t^+] \subset [\mathtt n_t^-,\mathtt n_t^+]$ when $\alpha^t < 2$, which gives $\min_{\lambda\in [\mathtt N_t^-,\mathtt N_t^+] }|\mathcal A_t(\lambda)| \geq \exp((\log n)^{1-\alpha^{t-1}}\mathtt N_t^-) = \exp((\log n)^{\Omega(1)})$ by the fact that $\Ecal_{4,t}$ holds.

\noindent{\bf Case 2:} $\alpha^t>2$. Since we have already proved that with probability $1-o(1)$, the leading candidate of the $(t+1)$-st round is identical to the leading candidate of the $t$-th round, i.e. $\mathcal L(\mathbf X_{t}) = \mathcal L(\mathbf X_{t+1})$.
\end{proof}

\section{Proof of Theorem~\ref{thm-slow-ordering-sub-critical}}{\label{sec:proof-subcritical-slow}}

In this section, we prove Theorem~\ref{thm-slow-ordering-sub-critical}. Intuitively, the absorption time is exponentially large with high probability if a candidate's victory requires an exponentially unlikely large-deviation event to happen. We formalize this by observing that earning a large fraction of votes in a round will cause a decrease in the expected number of votes in the next round when $\alpha<1$.

\begin{proof}[Proof of Theorem~\ref{thm-slow-ordering-sub-critical}]
    Assume that $n$ is a sufficiently large integer. For a state $\mathbf X=(\mathbf X^{(1)},\ldots,\mathbf X^{(n)})$, we say $\mathbf X$ is \emph{large} if 
    \begin{align*}
        \mathbf M_{\mathbf X}= \max_{1 \leq i \leq n} \left\{ \mathbf X^{(i)} \right\} > \frac{3}{4}n
    \end{align*}
    and \emph{small} otherwise. Then it suffices to show that for some constant $C=C(\alpha)>0$,
    \begin{align*}
        \Pb_n^\alpha\left( \mathbf X_{t+1} \mbox{ is large} \,;\mathbf X_{t} \mbox{ is small} \right) \leq \exp(-Cn)\,.
    \end{align*}
    Note that given that $\mathbf X_t$ is small, for all $1 \leq i \leq n$ we have (below we use $\alpha<1$)
    \begin{align*}
        \frac{ (\mathbf X_{t}^{(i)})^\alpha }{ \mathbf Z_t } \leq \frac{ (\mathbf X_{t}^{(i)})^\alpha \mathbf{1}_{\{\mathbf X_t^{(i)}>0\}}}{ (\mathbf X_{t}^{(i)})^\alpha + (n-\mathbf X_{t}^{(i)})^\alpha } \leq \frac{ \mathbf{1}_{\{\mathbf X_t^{(i)}>0\}} }{ 1+(\frac{n}{\frac{3n}{4}}-1)^\alpha } \leq \frac{1}{1+(1/3)^\alpha} < \frac{3}{4} \,.
    \end{align*}
    Therefore, defining $\widehat{C}=\widehat{C}_\alpha=\frac{3}{4}-\frac{1}{1+(1/3)^\alpha}>0$, we have
    \begin{align*}
        \Pb_n^\alpha\left( \mathbf X_{t+1}^{(i)} > \frac{3n}{4} \mid \mathbf X_t \mbox{ is small} \right) \leq \Pb_n^\alpha\Big( \mathsf{Bin}\Big( n,\frac{3}{4}-\widehat{C}_\alpha \Big) > \frac{3n}{4} \Big) \leq \exp(-\frac{n \widehat{C}_\alpha^2}{5})\,.
    \end{align*}
    Therefore, taking $C=\frac{\widehat{C}_\alpha^2}{6}$, by a union bound we have
    \begin{align*}
         &\Pb_n^\alpha\left( \mathbf X_{t+1} \mbox{ is large} \mid \mathbf X_t \mbox{ is small} \right) \leq n\exp(-\frac{n\widehat{C}_\alpha^2}{5}) \leq \exp(-Cn)\,,
    \end{align*}
    Note that the initial state is small, and absorption implies a first transition from small to large. Thus, for all $k \in \mathbb Z_+$ we have
    \begin{align*}
        \Pb_n^\alpha(\mathcal T\geq k) \geq \prod_{0\leq t \leq k-1}\Pb_n^\alpha\left( \mathbf X_{t+1} \mbox{ is small} \mid \mathbf X_t \mbox{ is small} \right) \geq (1-\exp(-Cn))^k\,,
    \end{align*}
    which proves our claim.
\end{proof}

\section*{Acknowledgement}
The model studied in this paper was inspired by a discussion with Francesco Maria Saettone after the passing of Pope Francis. We are very grateful to Ori Gurel-Gurevich for pointing out that the case $\alpha=1$ is equivalent to the iterated composition of random functions. We also thank Ofer Zeitouni for fruitful discussions. G.C. and Z.L. are partially supported by NSFC Major Program (Project No. 12595294), NSFC Key Program (Project No. 12231002), and the New Cornerstone Investigator Program (Project No. NCI202501).

\bibliographystyle{plain}
\bibliography{ref}

\end{document}